\documentclass[11pt]{amsart}

\usepackage[utf8]{inputenc}

\usepackage[margin=1in]{geometry}

\usepackage{hyperref}

\usepackage{bbm}
\usepackage{enumitem}
\usepackage{mathabx}
\usepackage{xcolor}
\usepackage{amsmath,amssymb}
\newcommand{\vertiii}[1]{{\left\vert\kern-0.25ex\left\vert\kern-0.25ex\left\vert #1 
		\right\vert\kern-0.25ex\right\vert\kern-0.25ex\right\vert}}

\theoremstyle{plain}
\newtheorem{thm}{Theorem}[section]
\newtheorem{cor}[thm]{Corollary}
\newtheorem{lemma}[thm]{Lemma}
\newtheorem{prop}[thm]{Proposition}

\theoremstyle{definition}
\newtheorem{defi}[thm]{Definition}

\newtheorem{rem}[thm]{Remark}

\numberwithin{equation}{section}

\usepackage{amsmath}
\allowdisplaybreaks[1]

\DeclareMathOperator{\NN}{\mathbb{N}}

\DeclareMathOperator*{\esssup}{ess\,sup}

\DeclareMathOperator{\TT}{\mathbb{T}^2}
\DeclareMathOperator{\RR}{\mathbb{R}}
\DeclareMathOperator{\ZZ}{\mathbb{Z}}
\DeclareMathOperator{\Sol}{Sol}
\DeclareMathOperator{\divv}{div}
\DeclareMathOperator{\esslim}{esslim}

\title{Martingale solutions to the stochastic thin-film equation in two dimensions}
\author{Max Sauerbrey\textsuperscript{1}}
\thanks{\textsuperscript{1}Delft Institute of Applied Mathematics, TU Delft, The Netherlands. \textit{Email:} \href{mailto:M.Sauerbrey@tudelft.nl}{M.Sauerbrey@tudelft.nl}}
\keywords{Thin-film equation, Noise, $\alpha$-Entropy estimates, Stochastic compactness method}
\subjclass[2010]{35R60,  	76A20}

\begin{document}
	\maketitle
	
	\noindent
\textbf{Abstract.}
	We construct solutions to the stochastic thin-film equation with quadratic mobility and Stratonovich gradient noise in the physically relevant dimension $d=2$ and allow in particular for solutions with non-full support. The construction relies on a Trotter-Kato time-splitting scheme, which was recently  employed  in $d=1$. The additional analytical challenges due to the higher spatial dimension are overcome using $\alpha$-entropy estimates and corresponding tightness arguments.

		\vspace{1cm}
		
	\noindent
	\textbf{R\'esum\'e.}
	Nous construisons des solutions de l'équation aux dérivées partielles stochastique des couches minces avec une mobilité quadratique et un forçage stochastique gradient de  type Stratonovich en dimension $d=2$, physiquement pertinente.  Les solutions à support non plein sont autorisées. La construction repose sur
	une méthode de Trotter-Kato en fractionnant l'intervalle de temps, récemment été utilisée dans le cas $d=1$. Les difficultés supplémentaires, dues à la dimension spatiale supérieure, sont surmontées à l'aide d'estimations de l'$\alpha$-entropie  et d'arguments de tension correspondants.

\vspace{1cm}
\section{Introduction}
The general stochastic thin-film equation is of the form 
\begin{equation}\label{Eq40}
\partial_t u \,=\, - \divv (m(u) \nabla (\Delta u-F'(u))) + \divv ( \sqrt{m(u)} \partial_t{W})
\end{equation}
and models the height of a thin liquid film $u(t,x)$ on a surface under the influence of thermal fluctuations. The function $m$ is the mobility function, $F$ is an interface potential  and  $\partial_t{W}$ is spatio-temporal noise. The stochastic thin-film equation was independently introduced in \cite{DMES2005} and \cite{GruenMeckeRauscher2006} as a lubrication approximation of the stochastic Navier-Stokes equation and explains, as shown in \cite{GruenMeckeRauscher2006}, discrepancies between  simulations of deterministic thin-films and experiments for films with small heights.

The mobility function has typically the form $m(u)=u^n$ for some real exponent $n$ depending on the boundary condition imposed at the boundary between the film and the substrate. In particular, imposing the no-slip condition for the Navier-Stokes equations, one obtains the mobility  $m(u)=u^3$ and using the Navier-slip condition with slip length $l$ one obtains $m(u)=lu^2+u^3$, which is up to scaling reasonably well approximated by $m(u)=u^2$ for small film heights $u\ll l$. The interface potential $F$ models forces between the molecules of the fluid film and the surface and takes for example the form $F(u)=u^{-8}-u^{-2}+1$ in case of the $6$-$12$ Lennard-Jones potential. If one neglects these forces and allows in particular for the qualitatively interesting situation of solutions without full support, one chooses $F(u)=0$ instead. 

The first existence result for  solutions to the stochastic thin-film equation was  proved in \cite{fischer_gruen_2018} for $m(u)=u^2$ with non-zero interface potential and the It\^{o} interpretation of the noise term  in dimension $d=1$ using a Galerkin approximation. The interface potential $F(u)$ is there assumed to become singular at $u=0$, which ensures that the constructed solutions stay strictly positive for all times. In \cite{GessGann2020}, it was pointed out that the Stratonovich interpretation of \eqref{Eq40} is more natural, since the use of It\^{o} noise would yield additional correction terms in the lubrication approximation. Moreover, the use of Stratonovich noise allowed the authors to employ a splitting of the deterministic and stochastic dynamics and  construct solutions to  \eqref{Eq40} with $m(u)=u^2$ and $F(u)=0$, again in $d=1$. In particular, the result allows for initial values without full support, which is from a  mathematical viewpoint more challenging, due to the loss of parabolicity of \eqref{Eq40} in this case. 
Based on a Galerkin approximation,  solutions to \eqref{Eq40} in  $d=1$ with $F(u)=0$ and  $m(u)= u^n$ with $n\in [\frac{8}{3}, 4)$ were recently constructed in \cite{dareiotis2021nonnegative}. Although no interface potential is assumed, the mathematical analysis is based on an entropy estimate which ensures that the constructed solution stays  positive almost everywhere.
We also mention that in \cite{cornalba2018priori}  a-priori estimates for \eqref{Eq40} were deduced at the cost of adding additional compensation terms on the right-hand side and it was shown that a local existence result would extend to a global one based on these estimates.  The first existence result in the physically relevant dimension  $d=2$ was recently proved in \cite{metzger2021existence} for $m(u)=u^2$ and a non-zero interface potential as an adaption  of \cite{fischer_gruen_2018} to the higher-dimensional setting. 
Finally, in  \cite{gruenklein2021} a result similar to the one in \cite{GessGann2020} was shown using a different regularization procedure and additionally so-called $\alpha$-entropy estimates were derived. Both, the higher dimensional setting as well as the $\alpha$-entropy estimates are central elements of the current article and we point out that it was  developed independently of \cite{metzger2021existence} and  \cite{gruenklein2021}.

In the current article, we generalize the approach form \cite{GessGann2020} to the case $d=2$, i.e. we show existence of martingale weak solutions to \eqref{Eq40} with $m(u)=u^2$, $F(u)=0$ and  Stratonovich noise on the two-dimensional torus.
 The higher spatial dimension leads to additional mathematical challenges due to the reduced gain of integrability after employing the Sobolev embedding theorem. Indeed, 
in 	\cite{GessGann2020}  the control of the surface energy
\[
\int_{\mathbb{T}} |u'|^2 \, dx
\] 
 suffices to show convergence of the nonlinear terms from the sequence of approximate solutions. As apparent from the  deterministic setting \cite{Passo98ona}, the additional control of the dissipation terms of the $\alpha$-entropy
\[
-\int_{\TT} u^{\alpha+1}\, dx, \;\; \alpha\in (-1,0)
\] 
is necessary to deduce convergence of the nonlinear terms in the two-dimensional case. Hence, to adapt the time splitting approach from \cite{GessGann2020}, we have to additionally control the $\alpha$-entropy along the splitting scheme and use the more delicate limiting procedure from  \cite{Passo98ona} compared to the one-dimensional case \cite{BERNIS1990}. Combining this with the stochastic compactness method is the key challenge of this article.  Moreover, compared to the independently proven result in \cite{metzger2021existence}, where  an existence result in the two-dimensional case based on the dissipation of the classical entropy 
\[-
\int_{\TT}\log (u)\, dx
\]
is given, we do allow for  solutions with a contact line between the fluid film and the solid, which is a qualitatively interesting object to study, see for example \cite{FischerNitti2019} and the references therein. 

\subsection{Main Result}We state the existence result which we will prove in the course of this article. Choosing $m(u)=u^2$, $F(u)=0$ and interpreting the noise in the Stratonovich sense, we obtain the SPDE
\begin{equation}\label{Eq_STFEq} 
du_t \,=\, - \divv (u_t^2 \nabla \Delta u_t)	{\,dt\,} +\, \divv ( u_t\circ dW_t) \;\; \text{ on }\TT,
\end{equation}where $W_t$ is specified as follows.
We let $(\psi_l)_{l\in\mathbb{N}}$ be the orthonormal basis in $H^2(\TT, \RR^2)$ consisting of the eigenfunctions to the periodic Laplacian in the first and second component respectively, i.e. every $\psi_l$ is of the form $(\xi_k, 0)$ or $(0, \xi_k)$ for some $k\in \ZZ^2$, where 
\begin{equation}\label{Eq41}
\xi_k(x,y)\,=\, \frac{\tilde{\xi}_{k_1}(x)\,\tilde{\xi}_{k_2}(y)}{\sqrt{1+(2\pi|k|)^2+(2\pi|k|)^4}}
\end{equation}
and 
\begin{equation}\label{Eq42}
\tilde{\xi}_{ j}(x)\,=\, \begin{cases}
	\sqrt{2}\cos(2\pi { j} x), &{ j}<0,\\
	1,& { j}=0,\\
	\sqrt{2}\sin(2\pi { j} x), & { j}>0.\\
\end{cases}
\end{equation}
Moreover, we let $\Lambda = (\lambda_l)_{l\in \NN} \in l^2(\NN)$ satisfy the symmetry relation
\begin{equation}\label{sym_noise}
	\lambda_l\,=\lambda_{\tilde{l}} \;\; \text{ whenever }\;\; \psi_l=(\xi_k,0)\, \wedge \psi_{\tilde{l}}=(0,\xi_k)\;\; \text{ for some $k\in \ZZ^2$.}
\end{equation}
Then \begin{equation}\label{Eq50}
W_\Lambda(t)\,=\, \sum_{l=1}^\infty \lambda_l \beta^{(l)}_t\psi_l\end{equation}
for  independent Brownian motions $(\beta^{(l)})_{l\in \NN}$
defines a centered Gaussian process on $H^2(\TT,\RR^2)$ with the covariance operator $Qf =\sum_{l=1}^\infty {\lambda_l^2} (f,\psi_l)_{H^2(\TT, \RR)} \psi_l$.  Inserting $W_\Lambda$ as the driving process in \eqref{Eq_STFEq}, writing the Stratonovich integral  in It\^{o}-form, and writing $J=u^2\nabla \Delta u$ in the weak form from \cite[Eq. (3.2)]{Passo98ona} yields the following notion of weak martingale solutions to \eqref{Eq_STFEq}.
\begin{defi}\label{Defi_main}
	Let $T\in (0,\infty)$ and $q\in (2, \infty)$. A weak martingale solution to \eqref{Eq_STFEq} with $q'$-regular non linearity on $[0,T]$ consists out of a filtered probability space satisfying the usual conditions, a family of independent Brownian motions $(\beta^{(l)})_{l\in \mathbb{N}}$, a continuous process $(u(t))_{t\in [0,T]}$ in $H^1_w(\TT)$ together with a random variable $J$ with values in $L^2(0,T; L^{q'}(\TT{, \RR^2}))$  such that
	\begin{enumerate}[label=(\roman*)]
		\item
		\label{item2000}
		 $u(t)$, $J|_{[0,t]}$ are $\mathfrak{F}_t$-measurable as random variables in $H^1(\TT)$ and $L^2(0,t; L^{q'}(\TT, {\RR^2}))$, respectively, for every $t\in [0,T]$,
		\item \label{Item1000} $|\nabla u| \in L^3(\{u>0\})$ and for all $\eta \in L^\infty(0,T; W^{2, \infty}(\TT))$ it holds
		almost surely 
		\begin{align}\label{Defi_J_is_nonlin_part}\begin{split}\int_0^T\int_{\TT} J\cdot \eta \, dx\,dt\,=\,&
				\int_0^T\int_{\{u({ t })>0\}} |\nabla u|^2 \nabla u \cdot \eta \, dx \,{ dt }\\&\,+\,
				\int_0^T\int_{\{u({ t })>0\}} u |\nabla u|^2 \divv \eta\, dx \,{ dt }\\&\,+\,
				2\int_0^T\int_{\{u({ t })>0\}} u\,\nabla^T u D\eta \nabla u\, dx \,{ dt }\\&\,+\,
				\int_0^T\int_{\TT} u^2\nabla u\cdot \nabla \divv\eta\, dx \,{ dt }
			\end{split}
		\end{align}  
		\item and for all $\varphi\in W^{1, q}(\TT)$ we have	\begin{align}\begin{split}
				\label{Defi_SPDE}
			\left<
			{u}(t), \varphi
			\right>\, - \,\left<
			u_0, \varphi
			\right>\,=\,&
			\int_0^t -\left<\divv({J}), \varphi\right> \, ds
			\, +\,{ \frac{1}{2} }\sum_{l=1}^\infty\int_0^t \lambda_l^2\left<\divv (\divv({u}(s)\psi_l)\psi_l), \varphi\right>\, ds\\&+
			\sum_{l=1}^\infty \lambda_l \int_{{ 0}}^t\left< \divv({u}(s)\psi_l), \varphi \right>\, d{\beta}^{(l)}_s
			\end{split}
		\end{align}almost surely for all $t\in [0,T]$.
	\end{enumerate}
\end{defi}
\begin{rem}\begin{enumerate}[label = (\roman*)]
	\item By the weak continuity in $H^1(\TT)$ any solution $u$ in the sense of Definition \ref{Defi_main} satisfies
		\begin{equation}\label{Eq37}
			\sup_{0\le  t\le T} \|u(t)\|_{H^1(\TT)}\,<\, \infty
		\end{equation}
		almost surely.
		\item The measurability assumption on $J$ in item \ref{item2000} ensures that all the terms on the right-hand side of \eqref{Defi_SPDE} are adapted. Interpreting $J$ as an element of the distribution space $\mathcal{D}'(\RR\times \TT)$, one can equivalently demand that $J$ is adapted to $\mathfrak{F}$ in the sense of distributions \cite[Definition 2.2.13]{BreitFeireislHofmanova2018}. This follows by density of $C_c^\infty ((0,t)\times \TT)$ in $L^2(0,t; L^{q}(\TT, {\RR^2}))$, separability of $L^2(0,t; L^{q'}(\TT, {\RR^2}))$,  and the equivalence of weak and Borel measurability in separable Banach spaces \cite[Proposition 1.1.1]{AnainBanachspaces1}.
	\end{enumerate}
\end{rem}
In course of this article, we will derive the following existence result.
\begin{thm}\label{Thm_main}
	Let $\mu$ be a probability distribution on  $H^1(\TT)$  supported on the non-negative functions,  $T\in (0,\infty)$, $q\in (2, \infty)$ and $\alpha\in (-1,0)$. Then there exists a weak martingale solution to \eqref{Eq_STFEq} on $[0,T]$ with $q'$-regular non linearity  satisfying 
	$u(0)\sim \mu$. Moreover,
	\begin{enumerate}[label=(\roman*)]
		\item $u(t)\ge 0$ almost surely for all  $t\in [0,T]$,
		\item we have for $p\in (0,\infty)$ the estimates 
		\begin{align} \label{Eq34}
				{E}\left[
			\sup_{0\le t\le T} \|{u}\|_{H^1(\TT)}^p
			\right]\, &\lesssim_{\Lambda, p, T} \,\int \|\cdot\|^p_{H^1(\TT)}\,d\mu,\\\label{Eq35}
			{E}\left[ \|J\|_{L^2(0,T; L^{q'}(\TT, { \RR^2}))}^\frac{p}{2} \right]\, &\lesssim_{\Lambda, p,q , T}\, \int \|\cdot\|^p_{H^1(\TT)}\,d\mu
		\end{align}
	\item and it holds the additional  spatio-temporal regularity 
	\begin{equation}\label{Eq13}
		u^{\frac{\alpha+3}{2}}\,\in \, L^2(0,T; H^2(\TT))\;\; \text{ and }\;\;u^{\frac{\alpha+3}{4}}\in L^4(0,T; W^{1,4}(\TT))
	\end{equation}
	almost surely.
	\end{enumerate}
\end{thm}
\begin{rem} 
	\begin{enumerate}[label=(\roman*)]
		\item
		 We point out that we allow for the  right-hand sides of \eqref{Eq34} and \eqref{Eq35} to be infinite, in which case the corresponding estimate trivializes.
		\item
		 We note that
	the differential identity
	\[
	\nabla u\,=\, \frac{4}{\alpha+3}u^{\frac{1-\alpha}{4}}\nabla u^\frac{\alpha+3}{4},
	\]
	the Sobolev embedding theorem,  \eqref{Eq13} and \eqref{Eq37}  imply that $|\nabla u|\in L^{4-}([0,T]\times \TT)$ almost surely and in particular the integrability condition from Definition \ref{Defi_main} \ref{Item1000}.
\end{enumerate}
\end{rem}
\subsection{Discussion of the result}
Theorem \ref{Thm_main} generalizes \cite[Theorem 1.2]{GessGann2020} to the setting in two dimensions and is therefore together with the independently developed result  \cite[Theorem 3.5]{metzger2021existence} the first existence result for the stochastic thin-film equation in higher dimension. 
As in the one-dimensional case, the time splitting approach is not only suitable to construct solutions to the stochastic thin-film equation, but suggests a numerical approach for their simulation as well.
The assumption $\Lambda\in l^2(\NN)$ on the noise \eqref{Eq50}  is the same as in \cite{fischer_gruen_2018,GessGann2020,dareiotis2021nonnegative}, where we refer the reader for an interpretation of the expansion \eqref{Eq50} in terms of a spatial correlation function of the noise to the exposition in \cite{bloemker_2005}. The additionally imposed  symmetry condition \eqref{sym_noise} expresses that the coordinate-wise noise processes are distributed according to the same Gaussian law in $H^2(\TT)$. This is a physically reasonable assumption since the noise is induced by thermal fluctuations and its distribution depends consequently on its position but not on its direction.  The same symmetry condition appears in \cite[Eq. (2.19)]{metzger2021existence}, where the use of Stratonovich noise is discussed, which indicates that it is an important assumption to treat the stochastic thin-film equation in higher dimensions.  We point out that in \cite{metzger2021existence}, the expansion \eqref{Eq50} in terms of eigenfunctions of the periodic Laplacian is relaxed to a  slightly more general assumption.

In contrast to the existence results from the mentioned articles, there is no integrability assumption on the initial distribution required in Theorem \ref{Thm_main}. This is achieved by using a decomposition of the initial value in countably many parts which are each almost surely bounded in $H^1(\TT)$. Then one can  construct approximate solutions and apply tightness arguments for each of these parts separately and add them together afterwards. The only important feature of \eqref{Eq_STFEq} for this to work is that $u(t)=0$ is a solution to it. We remark that these kind of reductions  to bounded or  integrable initial values are well-known and can be achieved in the setting of probabilistically strong solutions via localization or changing the probability measure, see \cite[Proposition 4.13]{agresti2020nonlinear_PartII} or \cite[Theorem 6.9.2]{krylov2002introduction} for examples. We use the  decomposition of the initial value instead, since it is more compatible with the stochastic compactness method as well as the estimates  \eqref{Eq34} and  \eqref{Eq35}. The moment estimates \eqref{Eq34} and \eqref{Eq35} for $p<2$ are also new for the stochastic  thin-film equation and are obtained from the estimates for higher moments.

\subsection{Outline and discussion of the proof}
In section \ref{Sec_deterministic} we review the existence result for weak solutions to the deterministic thin-film equation in two dimensions from \cite{Passo98ona} and state properties of the obtained solutions which are immediate from their construction. Additionally, we show that there is a measurable solution operator using the measurable selection theorem, which is important to combine these results with the stochastic setting. This approach is to the author's knowledge new and might be of interest also for other situations, where  a measurable solution operator is required.

In section \ref{Sec_Strat_SPDE} we consider the regularized Stratonovich SPDE
\begin{equation}\label{Eq60}
du_t\, =\, \epsilon \Delta u_t \,{ dt}\,+\, \divv ( u_t\circ dW_t) \;\; \text{ on }\TT
\end{equation}
and establish well-posedness in $H^1(\TT)$ using the monotone operator approach to SPDEs. The  coercivity estimates \eqref{Eq_coerc_1}, \eqref{Eq_coerc_2} are obtained analogously to the one-dimensional case \cite[Eq. (A.9)]{GessGann2020} and require only some  adaptions to multivariable calculus, where the symmetry condition \eqref{sym_noise} is used. Their uniformity in $\epsilon$ is key to letting later on $\epsilon \searrow 0$ and eliminating the regularization term $\epsilon \Delta u_t$ from \eqref{Eq60}. We note that this  procedure is well-known and refer the reader to the article \cite{Gerencs_r_2014} and the references therein for more information on degenerate  parabolic SPDEs.
However, the general result \cite[Theorem 2.1]{Gerencs_r_2014} does not directly apply to 
\begin{equation}\label{Eq38}
du_t\, =\, \divv ( u_t\circ dW_t) \;\; \text{ on }\TT
\end{equation}
and the coercivity estimates are unique to our particular situation.

In section \ref{Sec_time_discr}, we start constructing approximate solutions to \eqref{Eq_STFEq} by splitting the stochastic and deterministic dynamics along a time stepping scheme with step length $\delta$. Using the properties of the solutions to the deterministic thin-film equation and the solutions to \eqref{Eq60}, we derive estimates on the approximate solutions which are uniform in $\epsilon$ and $\delta$. The procedure is analogous to the one-dimensional case, but we note that we take the slightly different approach to let $\epsilon\searrow 0$ afterwards to be able to apply It\^{o}'s formula to the whole time splitting scheme. After these estimates are obtained, it is straightforward to deduce tightness statements on the approximating sequence in $\epsilon$ and employ the Skohorod-Jakubowski theorem to obtain an almost surely convergent, equally distributed subsequence. Usually, the parabolic regularization procedure does not require to pass to another probability space, see again \cite{Gerencs_r_2014}, but it is in our case convenient to  ensure convergence of the solutions to the deterministic equation as well.

Finally, in section \ref{Sec_time_limit}, we derive additional estimates on the approximating sequence by controlling the entropy production along the stochastic dynamics by means of It\^{o}'s formula. Using the obtained estimates, we show additional tightness properties of powers of the solution by an adaption of the compactness argument in \cite[Lemma 2.5]{Passo98ona}, which is compatible with our splitting scheme. These arguments are unique to the higher-dimensional setting and distinguish our approach from the one-dimensional case. We employ the Skohorod-Jakubowski theorem once more to let $\delta\searrow 0$ and identify the limit as a solution to \eqref{Eq_STFEq} combining the methods from \cite[Theorem 3.2]{Passo98ona} and \cite[Section 5.2]{GessGann2020}. As a result of the construction the additional estimates \eqref{Eq34}, \eqref{Eq35}, and the regularity properties \eqref{Eq13} follow.

 The reason to use  the time-splitting approach instead of a linear parabolic regularization is that it  directly yields non negative solutions, because the deterministic result \cite{Passo98ona} provides  non negative solutions and the regularized stochastic part of the equation admits a maximum principle.  Since we are dealing with a fourth order equation, a linear parabolic regularization of the whole equation would yield possibly negative solutions, which lack a reasonable physical interpretation. However, a more delicate, nonlinear regularization is possible as demonstrated in the one-dimensional case \cite{dareiotis2021nonnegative} or \cite{gruenklein2021}, but would require a longer proof.

\subsection{Notation}We use the notation $\lesssim$ to indicate that an inequality holds up to a universal constant and write $\lesssim_{p_1,\dots}$ if the constant depends on nothing but the parameters $p_1, \dots$. Similarly, we write $C$ for a universal constant and $C_{p_1,\dots}$, if the constant depends on $p_1, \dots$.
We write 
\begin{equation*}
	G_\alpha(t)\,=\, \int_1^t\int_1^s\tau^{\alpha-1} d\tau\,ds, \;\; \alpha \in \RR
\end{equation*} for the  (mathematical) $\alpha$-entropy, and point out for later reference that
\begin{equation}\label{Eq_alpha_entropy_expl}
	G_\alpha(t) =  \frac{t^{\alpha+1}}{\alpha(\alpha+1)} +r_\alpha(t)
	,\;\; t\ge 0,
\end{equation}
if $\alpha\in (-1,0)$, where $r_\alpha$ is a first order polynomial.  We use classical notation for differential operators, i.e. write $\nabla f$, $\divv(f)$, $\Delta f$ for the gradient, divergence and Laplacian of a function or a vector field $f$, respectively. Moreover, we write $Hf$ for the Hessian matrix and use the notational  convention that a differential operator is only applied to the first function appearing afterwards such that e.g.
\[
\nabla fg \,=\, g (\nabla f) ,\;\; \text{ but }\;\; \nabla (fg)\,=\, f (\nabla g)\,+\, g(\nabla f).
\]
We denote our domain, the $2$-torus, by $\TT$. We write $L^p(\TT)$, $W^{k,p}(\TT)$ and $H^k(\TT)$ for the Lebesgue, Sobolev and Bessel potential spaces on $\TT$ with integrability and smoothness exponents $p,k$, where more information on periodic spaces can be found in \cite[Section 3]{schmeisser1987topics}. We note that if $k$ is an integer, we equip  $H^k(\TT)$ with the equivalent $W^{k, 2}(\TT)$-inner product.
We write $L^p(\TT, \RR^2)$, $W^{k,p}(\TT, \RR^2)$ and $H^k(\TT, \RR^2)$ for the corresponding spaces of vector fields and equip them with the direct sum norm and set for the special  case $p=2$
\[\|(f_1, f_2)\|_{L^2(\TT, \RR^2)}^2\,=\, 
\|f_1\|^2_{{L^2}(\TT)}\,+\, \|f_2\|^2_{L^2(\TT)}, \;\;
\|(f_1, f_2)\|^2_{H^k(\TT, \RR^2)}\,=\, 
\|f_1\|^2_{H^k(\TT)}\,+\, \|f_2\|^2_{H^k(\TT)}
\] 
to preserve the Hilbert space structure. We write $\left<f, g\right>$ for the dual pairing in $L^2(\TT)$ and in $L^2(\TT, \RR^2)$ depending on $f,g$ being functions or vector fields.
If $(S, \nu)$ is a measure space and $X$ is a Banach space, we write $L^p(S,\nu,X)$ for the Bochner space of strongly measurable, $p$-integrable, $X$-valued functions on $(S, \nu)$. For details  we refer to \cite[Section 1]{AnainBanachspaces1}.
If it is clear which measure is considered, we use also the notation $L^p(S,X)$ and if $S=[s,t]$ and $\nu$ the Lebesgue measure $L^p(s,t;X)$.
Moreover, we write $C(0,T; X)$, $H^1(0,T; X)$,
$C^\gamma(0,T; X)$ and $W^{\gamma, p}(0,T; X)$, for the space of continuous functions,  first-order Sobolev space, H\"older and Sobolev-Slobodetskii space on $[0,T]$ with values in $X$, where we will only consider fractional exponents $\gamma \in (0,1)$. The corresponding H\"older semi-norm is denoted by $
[\cdot]_{\gamma,X}$ and for precise definitions of these spaces we refer to \cite[Section 2]{amann_embedding}. If a Banach space $X$ is considered with its weak or weak-* topology, we express this by writing $X_w$ or $X_{w*}$ respectively. Lastly we mention that we write $L_2(H_1, H_2)$ for the space of Hilbert-Schmidt operators between two Hilbert spaces $H_1$ and $H_2$.

\section{The deterministic thin-film equation}
\label{Sec_deterministic}
In this section we summarize the existence result for weak solutions to the deterministic thin-film equation in the special case of quadratic mobility
\begin{equation}\label{Eq_det_eq}
\partial_t v \,=\, -\divv(v^2\nabla \Delta v)
\end{equation}
from \cite{Passo98ona}. Moreover, we show that the solutions can be chosen in a measurable way, which will be important later. We remark that in \cite{Passo98ona} solutions to \eqref{Eq_det_eq} are constructed on a domain with  Neumann boundary conditions, but the arguments translate verbatim to the periodic setting. First, we recall the definition of weak solutions to \eqref{Eq_det_eq} from \cite[Definition 3.1]{Passo98ona}.

\begin{defi}\label{Defi_sol_det}Let $q\in (2, \infty]$ and $T>0$. 
	A weak solution to the (deterministic) thin-film equation on $[0,T]$ with $q'$-regular non linearity is a tuple
	\[
	(v,J)\,\in \, L^\infty(0,T;H^1(\TT))\cap H^1(0,T; W^{-1,q'}(\TT))\,\times \,L^2(0,T; L^{q'}(\TT, \mathbb{R}^2)){,}
	\] such that $\partial_t v= -\divv J$ in $L^2(0,T; W^{-1,q'}(\TT))$, $\nabla v \in L^3(\{
	v>0\}, \mathbb{R}^2)$ and 
	\begin{align}\label{Eq_weak_form_J}\begin{split}\int_0^T\int_{\TT} J\cdot \eta \, dx\,dt\,=\,&
			\int_0^T\int_{\{v(t)>0\}} |\nabla v|^2 \nabla v \cdot \eta \, dx \,dt\\&\,+\,
			\int_0^T\int_{\{v(t)>0\}} v |\nabla v|^2 \divv \eta\, dx \,dt\\&\,+\,
			2\int_0^T\int_{\{v(t)>0\}} v\,\nabla^T v D\eta \nabla v\, dx \,dt\\&\,+\,
			\int_0^T\int_{\TT} v^2\nabla v\cdot \nabla \divv\eta\, dx \,dt
		\end{split}
	\end{align}
	for all $\eta\in L^\infty(0,T;W^{2,\infty}(\TT,\RR^2)))$.
\end{defi}
\begin{rem}\label{Rem_version_det_sol}By Rellich's theorem, see \cite[Theorem 6.3, p.168]{adams2003sobolev}, and the Aubin-Lions lemma \cite[Corollary 5]{Simon1987} there is a compact embedding 
	\[L^\infty(0,T;H^1(\TT))\cap H^1(0,T; W^{-1,q'}(\TT)) \,\hookrightarrow\,
	C(0,T;L^r(\TT))
	\]
	for any $r\in[1,\infty)$. In the following, we will always identify a weak solution to the thin-film equation with its $L^r(\TT)$-continuous version. By \cite[Lemma II.5.9]{boyer2012mathematical} this version is weakly continuous as a mapping with values in $H^1(\TT)$.
\end{rem}
 The identity \eqref{Eq_weak_form_J} is a weak formulation of $J= u^2\nabla \Delta u$.
The following existence statement is given in \cite[Theorem 3.2]{Passo98ona}, where we add some quantitative estimates which follow from the construction in \cite{Passo98ona} and are proved in detail in appendix \ref{AppendixA}.

\begin{thm}\label{Thm_ex_sol_det}
	Let $v_0\in H^1(\TT)$ be non negative, $q\in (2,\infty]$, $T>0$ and $\alpha\in (-1,0)$. Then there exists a weak solution $(v,J)$ to the thin-film equation on $[0,T]$ with $q'$-regular non linearity and $v(0)= v_0$, which satisfies the following properties for universal constants $0<{C_\alpha},C_{ q}<\infty$. 
	
	\begin{enumerate}[label=(\roman*)]
		\item \label{Item_det_cons_of_mass} We have 
		for all $t\in [0,T]$ that
		\begin{equation*}
			\int_{\TT} v(t, \cdot )\, dx\,=\, \int_{\TT} v_0 \, dx\;\; \text{ and }\;\; v(t,\cdot )\ge 0.
		\end{equation*}
		\item \label{Item_det_energy_est}It holds the energy estimate
		\begin{equation*}
			\sup_{0\le t\le T}\|\nabla v(t) \|_{L^2(\TT,\RR^2)}\,\le \, \|\nabla v_0\|_{L^2(\TT,\RR^2)}.
		\end{equation*}
		\item \label{Item_det_energy_est_J} It holds that \begin{align*}&
				\|J\|_{L^2(0,T;L^{q'}(\TT,\RR^2))}^2\,+ \, C_{ q}
				\|\nabla v(T)\|_{L^2(\TT,\RR^2)}^2\left(\|\nabla v(T)\|_{L^2(\TT,\RR^2)}^2+\left|
				\int_{\TT} v_0\, dx\right|^2
				\right)\\\le\,&
				C_{ q}\|\nabla v_0\|_{L^2(\TT,\RR^2)}^2\left(\|\nabla v_0\|_{L^2(\TT,\RR^2)}^2+\left|
				\int_{\TT} v_0\, dx\right|^2
				\right).\end{align*}
		\item \label{Item_det_entr_est}We have  the $\alpha$-entropy estimate
		\begin{align*}
			\int_{\TT} G_\alpha(v(T, \cdot))\, dx \,+\,\frac{1}{{C_\alpha}} \int_0^T\int_{\TT} |H v^{\frac{\alpha+3}{2}}|^2+|\nabla v^{\frac{\alpha+3}{4}}|^4 \, dx\, dt\,\le \, 
			\int_{\TT} G_\alpha( v_0)\,dx.
		\end{align*}
	\end{enumerate}
\end{thm}
The following result can be proved  along the lines of \cite[Lemma 2.5, Proposition 2.6, Corollary 2.7, Theorem 3.2]{Passo98ona}.

\begin{prop}\label{Prop_conv_sol_det}Let $q\in (2, \infty]$, $T>0$ and  $(v_n, J_n)_{n\in \NN}$ be a sequence of non negative weak solutions to the deterministic thin-film equation on $[0,T]$ with $q'$-regular non linearity. Assume that there is an $\alpha\in (-1,0)$ such that $v_n$, $J_n$, $v_n^{\frac{\alpha+3}{2}}$ and $ v_n^{\frac{\alpha+3}{4}}$ are uniformly bounded in 
	\begin{equation}\label{Eq3}
		L^\infty(0,T;H^1(\TT)), \;L^2(0,T; L^{q'}(\TT,\RR^2)), \; L^2(0,T; H^2(\TT)), \;L^4(0,T; W^{1,4}(\TT))
	\end{equation} respectively.
Then for a subsequence we have
	\begin{enumerate}
		[label=(\roman*)]
		\item $v_n \rightharpoonup^* v$ in $L^\infty(0,T; H^1(\TT))$,
		\item $J_n\rightharpoonup J$ in $L^2(0,T; L^{q'}(\TT,\RR^2))$,
		\item $v_n^{\frac{\alpha+3}{2}}\rightharpoonup  v^{\frac{\alpha+3}{2}}$ in $L^2(0,T; H^2(\TT))$,
		\item $v_n^{\frac{\alpha+3}{4}}\rightharpoonup v^{\frac{\alpha+3}{4}}$ in $L^4(0,T; W^{1,4}(\TT))$ 
	\end{enumerate} and the limit $(v,J)$ is  a non-negative weak solution to the  thin-film equation  with $q'$-regular non-linearity.
	
\end{prop}
Finally, we give proof to the existence of a measurable solution operator.
To this end, we define the set $\mathcal{X}_{q,T}$ as the topological product of the spaces \eqref{Eq3}
equipped with the respective weak  and weak-* topologies. Moreover, we write $B_X(r)$ for the ball in $X$ centered at the origin with radius $r$,  if $X$ is a normed space.

\begin{cor}\label{Cor_mble_sol_map}Let $q\in (2, \infty]$, $T>0$ and $\alpha \in (-1,0)$. There is a  Borel-measurable mapping
	\begin{align}\label{Eq2}
		\mathcal{S}_{\alpha, q,T}\colon  \left\{v_0\in H^1(\TT)|v_0 \ge 0\right\} \to \mathcal{X}_{q,T}, \,v_0 \mapsto (v,J, v^{\frac{\alpha+3}{2}}, v^{\frac{\alpha+3}{4}}), 
	\end{align}
which assigns to every initial value a weak solution to the thin-film equation on $[0,T]$, which satisfies the properties (i)-(iv) of Theorem \ref{Thm_ex_sol_det}.
\end{cor}
\begin{proof}We define for $v_0 $ in the domain of \eqref{Eq2} the set of all weak solutions to the stochastic thin-film equation with initial value $v_0$ and $q'$-regular non linearity  satisfying (i)-(iv) from Theorem~\ref{Thm_ex_sol_det} together with its corresponding powers by $\Sol (v_0)\subset \mathcal{X}_{q,T}$. 
	We write $X_i$ for the $i$-th space in \eqref{Eq3} and observe that if $\|v_0\|_{H^1(\TT)}\le n$ for some $n\in \NN$ the a-priori bounds of Theorem \ref{Thm_ex_sol_det} yield that
	\[
	\Sol (v_0)\,\subset  \, \mathcal{X}^{(n)}_{q,T} \, \coloneq \,
	\bigtimes_{i=1}^4B_{X_i}(r_{i,n}) 
	\]
	for suitably chosen $r_{i,n}$. We equip each $B_{X_i}(r_{i,n}) $ again with the weak  (weak-*) topology of the respective space $X_i$ and $\mathcal{X}^{(n)}_{q,T}$ with the resulting product topology. We note that each $B_{X_i}(r_{i,n}) $ is metrizable by the separability of the  (pre-) dual of $X_i$, see \cite[Proposition 1.2.29, Corollary 1.3.22]{AnainBanachspaces1} and consequently also the topological product $\mathcal{X}^{(n)}_{q,T}$.
	Moreover,  $\mathcal{X}^{(n)}_{q,T}$ is compact as a consequence of Tychonoff's and the Banach-Alaoglu theorem and therefore in particular a Polish space. Let $(v_{0,j})_{j\in \NN}$ be a sequence in
	\[
	\left\{v_0\in H^1(\TT)|v_0 \ge 0, \|v_0\|_{H^1(\TT)}\le n\right\}
	\]
	converging to $v_{0, *}$ in $H^1(\TT)$ and \begin{equation}\label{Eq4}(v_j,J_j, v_j^{\frac{\alpha+3}{2}}, v_j^{\frac{\alpha+3}{4}})\,\in\, \Sol(v_{0,j}). \end{equation}
	 Then the measurable selection theorem as in \cite[Corollary 103, p.506]{ethier2005markov} yields a Borel-measurable solution map
	\begin{equation}\label{Eq6}
	\mathcal{S}_{\alpha, q,T}^{(n)}\colon  \left\{v_0\in H^1(\TT)|v_0 \ge 0, \|v_0\|_{H^1(\TT)}\le n\right\} \to \mathcal{X}_{q,T}^{(n)}, \,v_0 \mapsto (v,J, v^{\frac{\alpha+3}{2}}, v^{\frac{\alpha+3}{4}})\, \in\, \Sol (v_0), 
	\end{equation}
	if we can verify that a subsequence of
	\[
	(v_j,J_j, v_j^{\frac{\alpha+3}{2}}, v_j^{\frac{\alpha+3}{4}})_{j\in \NN}
	\]
	converges to an element of $\Sol(v_{0,*})$. Since \eqref{Eq4} lies in $\mathcal{X}^{(n)}_{q,T}$,  its components are  uniformly bounded in \eqref{Eq3}. Therefore, we can apply Proposition \ref{Prop_conv_sol_det} and obtain that
		\[
	(v_j,J_j, v_j^{\frac{\alpha+3}{2}}, v_j^{\frac{\alpha+3}{4}})\,\to \, 
	(v,J, v^{\frac{\alpha+3}{2}}, v^{\frac{\alpha+3}{4}})
	\]
	for a subsequence in $\mathcal{X}_{q,T}^{(n)}$, where $(v,J)$ is a non negative weak solution to the thin-film equation with $q'$-regular non linearity. By \cite[Corollary 5]{Simon1987} we have 
	$v_j\to v$ in $C(0,T; L^2(\TT))$ and in particular $v_j(0)\to v(0)$ in $L^2(\TT)$. Consequently we must have $v(0)=v_{0,*}$. By lower semi-continuity of the norm with respect to weak and weak-* convergence we deduce that $(v,J)$ satisfies all the properties (i)-(iv) of Theorem \ref{Thm_ex_sol_det} and therefore \[
	(v,J, v^{\frac{\alpha+3}{2}}, v^{\frac{\alpha+3}{4}})\in \Sol(v_{0,*}).
	\]
	Hence, the measurable selection theorem indeed yields a Borel measurable map \eqref{Eq6}. Finally,  we define $\mathcal{S}_{\alpha, q,T}v_0 = \mathcal{S}_{\alpha, q,T}^{(n)} v_0$ if $n-1\le \|v_0\|<n$. Since balls in $H^1(\TT)$ are Borel sets, $\mathcal{S}_{\alpha, q,T}$ has the desired properties.
\end{proof}

\section{A regularized linear Stratonovich SPDE on $H^1(\TT)$}\label{Sec_Strat_SPDE}
In this section we show that the regularized version of the stochastic part in \eqref{Eq_STFEq}
\begin{equation}\label{Eq_reg_problem}
		dw_t \,=\, \epsilon \Delta w_t \, {dt} \, +\, \divv ( w_t\circ dW_t)
\end{equation} is well-posed using the variational approach to SPDEs \cite[Chapter 4]{liu2015stochastic}. A key ingredient to checking the sufficient conditions for well-posedness is the spatial isotropy condition on the noise \eqref{sym_noise}. Throughout this section, we fix a filtered probability space $(\Omega, \mathfrak{A}, P)$ satisfying the usual conditions with a sequence of independent real-valued Brownian motions $(\beta^{(l)})_{l\in \NN}$ and an $\epsilon\in (0,1)$. The main statement of this section reads as follows.
\begin{thm}\label{Thm_Sec_Reg_SPDE}Let $p\in [2, \infty)$, $T\in [0, \infty)$ and $w_0\in L^p(\Omega, H^1(\TT))$ be $\mathfrak{F}_0$-measurable. Then there exists a unique continuous, adapted $H^1(\TT)$-valued process $w$ such that $w\in L^2([0,T]\times \Omega,  H^2(\TT))$ and 
	\begin{equation}\label{Eq_reg_SPDE}
	w(t) \, = \, w_0\,+\, \int_0^t\epsilon \Delta w(s)\,+\, \frac{1}{2} \sum_{l=1}^{\infty} \lambda_l^2  \divv (\divv(w(s)\psi_l) \psi_l )\, ds\,+\, \sum_{l=1}^\infty\lambda_l \int_0^t \divv (w(s) \psi_l)\, d\beta^{(l)}_s
	\end{equation}
	for every $t\in [0,T]$. Moreover, $w$ satisfies
	\begin{equation}\label{Eq_est_reg_SPDE}
	E\left[
	\sup_{0\le t\le T} \|w(t)\|_{H^1(\TT)}^p 
	\right]\,\lesssim_{p,T} \,
	E\left[\|w_0\|_{H^1(\TT)}^p 
	\right],
\end{equation} 
 almost surely we have \begin{equation}\label{Eq_cons_mass_stoch_part}
	\int_{\TT} w(t)\, dx \,=\, \int_{\TT} w_0\, dx
\end{equation}and if $w_0\ge 0$ also $w(t)\ge 0$ for all $t\in [0,T]$.
\end{thm}
\begin{rem}\label{Rem_1} We convince ourselves that all the terms from \eqref{Eq_reg_SPDE} are well-defined.
	By  \eqref{Eq41} it holds \begin{equation}\label{boundedness_psil}
		\sup_{|\alpha|\le 2}\,
		\sup_{k\in \ZZ^2}\, \|\partial_\alpha \xi_k\|_{L^\infty(\TT)} \,<\, \infty.
	\end{equation}
and therefore we have 
\begin{equation}\label{Eq5}\|\divv (\divv(w\psi_l) \psi_l )\|_{L^2(\TT)}\, \lesssim \, \|w\|_{H^2(\TT)}\;\;\text{ and }\;\;
	\|\divv(w\psi_l)\|_{H^1(\TT)}\, \lesssim \, {\|w\|_{H^2(\TT)}}
\end{equation}
for every $w\in H^2(\TT)$. Using the first estimate we derive that 
\[
E\left(\int_0^T\left\|\epsilon \Delta w(t)\,+\, \frac{1}{2} \sum_{l=1}^{\infty} \lambda_l^2 \divv (\divv(w\psi_l) \psi_l )\right\|_{L^2(\TT)}^2dt\right)\, \lesssim_{\Lambda}\, \|w\|_{L^2([0,T]\times \Omega,  H^2(\TT))}^2,
\]
and consequently the deterministic integral in \eqref{Eq_reg_SPDE} exists almost surely as a Bochner integral in $L^2(\TT)$. Using the second estimate from \eqref{Eq5}, one derives by the martingale moment inequality and It\^o's isometry that
\begin{align*}&
	E\left(\sup_{0\le t\le T} \left\|\sum_{l=n}^m \lambda_l \int_0^t \divv(w(s)\psi_l)\, d\beta^{(l)}_s\right\|_{H^1(\TT)}^2\right)\, \lesssim\, E\left( \left\|\sum_{l=n}^m \lambda_l \int_0^T \divv(w(t)\psi_l)\, d\beta^{(l)}_t\right\|_{H^1(\TT)}^2\right)\\=\, &E\left[
	\sum_{l=n}^m \lambda_l^2 \int_0^T \left\|\divv(w(t)\psi_l)\right\|_{H^1(\TT)}^2
	 \, dt\right] \, {\lesssim} \, \left( \sum_{l=n}^m \lambda_l^2\right) \|w\|_{L^2([0,T]\times \Omega, H^2(\TT))}^2
\end{align*}
and the latter part converges to $0$ as $n,m\to \infty$. Therefore, the series of stochastic integrals in \eqref{Eq_reg_SPDE} converges to a continuous square-integrable martingale in $H^1(\TT)$.
\end{rem}
In order to treat the equation \eqref{Eq_reg_SPDE} within the variational setting \cite[Chapter 4]{liu2015stochastic}, we introduce the operators
\begin{align*}&
A^\epsilon\colon \,H^2 (\TT)\to L^{2}(\TT), \,w \,\mapsto \,
\epsilon \Delta w\,+\, \frac{1}{2}\sum_{l=1}^\infty \lambda_l^2  \divv (\divv(w\psi_l) \psi_l ),
\\&
B\colon\, H^2(\TT)\to L_2(H^2(\TT,\RR^2),H^1(\TT)), w \mapsto \left[v\mapsto \sum_{l=1}^\infty\lambda_l (v,\psi_l)_{H^2(\TT{ ,\RR^2 })} \divv (w \psi_l)\right].
\end{align*}
As in Remark \ref{Rem_1} we conclude that the operators $A^\epsilon$ and $B$ are well-defined, linear and bounded. In the following lemma we verify coercivity of $(A^\epsilon,B)$. Its proof is  similar to \cite[Lemma A.3]{GessGann2020}, but nevertheless contained to stress the necessity of assumption \eqref{sym_noise}.
\begin{lemma}\label{Lemma_coercive}
	There exists a constant $C_\Lambda<\infty$ such that
	\begin{align}\label{Eq_coerc_1}
		2\left<A^\epsilon w,w\right>\,+\, \sum_{l=1}^\infty\left\| B(w)[\psi_l] \right\|_{L^2(\TT)}^2 \, &\le \, C_\Lambda\|w\|_{L^2(\TT)}^2\,-\, 2\epsilon\| w\|_{H^1(\TT)}^2,\\\label{Eq_coerc_2}
		2\left<\nabla A^\epsilon w,\nabla w\right>\,+\, \sum_{l=1}^\infty\left\| \nabla B(w)[\psi_l] \right\|_{L^2(\TT, \RR^2)}^2 \, &\le \, C_\Lambda\|w\|_{H^1(\TT)}^2\,-\, 2\epsilon\|\nabla w\|_{H^1(\TT,\RR^2)}^2
	\end{align}
	for all $w\in H^2(\TT)$.
\end{lemma}

\begin{proof}
	By continuity of the involved operators, it suffices to verify \eqref{Eq_coerc_1} and \eqref{Eq_coerc_2} for $w\in C^\infty(\TT)$. We  first observe that
	\begin{align*}
		\left<
		A^\epsilon w,w
		\right>\,=\,&
		-\epsilon \|\nabla w\|_{L^2(\TT)}^2\,-\,\frac{1}{2}\sum_{l=1}^\infty \lambda_l^2 \left<
		\divv (w\psi_l)\psi_l, \nabla w
		\right>\\=\,&
		-\epsilon \|\nabla w\|_{L^2(\TT)}^2\,-\,\frac{1}{2}
		\sum_{l=1}^\infty \lambda_l^2\|\psi_l\cdot \nabla w\|_{L^2(\TT)}^2\,+\,\frac{1}{4}
		\sum_{l=1}^\infty \lambda_l^2\left<
		w^2, \divv(\divv (\psi_l)\psi_l )
		\right>,
	\end{align*}
	where we have used the identity $\frac{1}{2}\nabla w^ 2=w\nabla w$ in the second line.
	Utilizing the same identity again, we obtain  
	\begin{align*}&
		\left\|  B(w)[\psi_l] \right\|_{L^2(\TT)}^2
		\,=\, \lambda_l^2\|\divv( w\psi_l)\|_{L^2(\TT)}^2
		\\=\, &\lambda_l^2\left(\|\psi_l\cdot \nabla w\|_{L^2(\TT)}^2+2\left<w\nabla w, \divv(\psi_l)\psi_l\right> + \left<w^2, \divv(\psi_l)^2\right> \right)\\=\,&\lambda_l^2\left(\|\psi_l\cdot \nabla w\|_{L^2(\TT)}^2- \left<w^2,\psi_l \cdot \nabla  \divv(\psi_l)\right> \right).\end{align*}
	Considering the bound \eqref{boundedness_psil}  we can calculate 
	\begin{align*}&
		2\left<A^\epsilon w,w\right>\,+\, \sum_{l=1}^\infty\left\| B(w)[\psi_l] \right\|_{L^2(\TT)}^2 \\ = \,&-2\epsilon \|\nabla w\|_{L^2(\TT)}^2\,+\, \frac{1}{2}\sum_{l=1}^\infty
		\lambda_l^2\left<
		w^2,  \divv(\divv (\psi_l)\psi_l )-\psi_l \cdot \nabla  \divv(\psi_l)
		\right>\\\le \,&
		C_\Lambda \|w\|^2_{L^2(\TT)}\,-\, 2\epsilon \|\nabla w\|_{L^2(\TT)}^2
	\end{align*}
	for a suitable constant $C_\Lambda<\infty$. Enlarging $C_\Lambda$ by $2$ yields \eqref{Eq_coerc_1}.
	For  \eqref{Eq_coerc_2} we observe  that
	\begin{equation*}
		\left<\nabla
		A^\epsilon w, \nabla w
		\right>\,=\,
		-\frac{1}{2}\sum_{l=1}^\infty\lambda_l^2
		\left< \divv  (\divv (w\psi_l)\psi_l ), \Delta w\right>\,-\,\epsilon \|\Delta w\|_{L^2(\TT)}^2.
	\end{equation*}
	To further analyze the involved series, we set  $\mu_k=\lambda_l$ in the situation of \eqref{sym_noise} and rewrite
	\begin{align*}-\frac{1}{2}\sum_{l=1}^\infty\lambda_l^2
		\left< \divv \cdot (\divv (w\psi_l)\psi_l ), \Delta w\right>\,=\,
		-\frac{1}{2}\sum_{k\in \ZZ^2}\mu_k^2
		\left< \divv  (\xi_k\nabla (w\xi_k) ), \Delta w\right>.
	\end{align*}
Before moving on, we notice that
\begin{equation}\label{squares}
	\tilde{\xi}_{j}^2(x)\,\stackrel{\eqref{Eq42}}{=}\, \begin{cases}
		1+\cos(4\pi {j} x), &{j}<0,\\
		1,& {j}=0,\\
		1-\cos(4\pi {j} x), & {j}>0\\
	\end{cases}
\end{equation}and therefore 
\[
\xi_k^2(x{,y})\,=\, \frac{\tilde{\xi}_{k_1}^2{(x)}\,\tilde{\xi}_{k_2}^2{(y)}}{1+(2\pi|k|)^2+(2\pi|k|)^4}
\] yields the bound
\begin{equation}\label{bounedness_xi2}
	\sup_{|\alpha|\le 4}\,
	\sup_{k\in \ZZ^2}\, \|\partial_\alpha \xi_k^2\|_{L^\infty(\TT)} \,<\, \infty.
\end{equation} 
	Using this estimate, the product rule for $\Delta$, integration by parts, as well as the differential identities
	\begin{equation}\label{aux_eq16}
		\nabla (\nabla f\cdot \nabla g)=Hf\nabla g+Hg \nabla f
		\;\;\text{ and }\;\;
		Hf\nabla f= \frac{1}{2}\nabla |\nabla f|^2
	\end{equation} we calculate
	\begin{align*}&
		\left<
		\divv  (\xi_k\nabla (w\xi_k) ), \Delta w
		\right>\,=\,	\left<\nabla \xi_k\cdot \nabla (w\xi_k)+\xi_k \Delta (w\xi_k), \Delta w
		\right>\\=\,&
		\left<\xi_k^2 \Delta w +\frac{3}{2} \nabla w\cdot \nabla \xi_k^2+\frac{1}{2}w\Delta\xi_k^2, \Delta w
		\right>\\=\,&\|\xi_k\Delta w\|_{L^2(\TT)}^2-\frac{3}{2}\left<
		Hw\nabla \xi_k^2+H\xi_k^2 \nabla w,\nabla w
		\right>+ \frac{1}{2}\left<\Delta \xi_k^2, \frac{1}{2}\Delta w^2-\nabla w\cdot \nabla w \right>
		\\\ge\,&\|\xi_k\Delta w\|_{L^2(\TT)}^2+\frac{3}{4}\left<
		\Delta \xi_k^2,|\nabla w|^2
		\right>+\frac{1}{2}\left<\Delta^2 \xi_k^2, w^2\right>-C\|w\|_{H^1(\TT)}^2
		\\\ge\,&\|\xi_k\Delta w\|_{L^2(\TT)}^2-C\|w\|_{H^1(\TT)}^2.
	\end{align*}
Here, we have enlarged the constant $C<\infty$ from the second last to the last line.
	Concerning the other summand in \eqref{Eq_coerc_2}, we observe that by integration by parts
	\[
	\|\nabla B(w)[(\xi_k,0)]\|_{L^2(\TT, \RR^2)}^2\,=\, \mu_k^2\|\nabla \partial_{1} (w\xi_k)\|_{L^2(\TT, \RR^2)}^2 \,=\,\mu_k^2\left<\Delta (w\xi_k),  \partial_{11} (w\xi_k)\right>.
	\]
	Rewriting the expression $	\|\nabla B(w)[(0,\xi_k)] \|_{L^2(\TT, \RR^2)}^2 $ analogously yields that
	\[
	\|\nabla B(w)[(\xi_k,0)]\|_{L^2(\TT, \RR^2)}^2 +\|\nabla B(w)[(0,\xi_k)]\|_{L^2(\TT, \RR^2)}^2 \,=\,\mu_k^2\|\Delta (w\xi_k) \|_{L^2(\TT)}^2.
	\]
	Using  again the product rule for $\Delta$, the bound \eqref{boundedness_psil}, integration by parts and the formulas from \eqref{aux_eq16}, we can estimate the latter term by
	\begin{align*}&\|\Delta (w\xi_k) \|_{L^2(\TT)}^2\,=\,
		\left\|\xi_k \Delta w+2\nabla w\cdot \nabla \xi_k + w\Delta\xi_k  \right\|_{L^2(\TT)}^2\\=\, &\|\xi_k \Delta w\|_{L^2(\TT)}^2+ \|2\nabla w\cdot \nabla \xi_k+w\Delta\xi_k\|_{L^2(\TT)}^2
		\, +\,2\left<\xi_k \Delta w,2\nabla w\cdot \nabla \xi_k+w\Delta\xi_k
		\right>\\\le\,&
		\|\xi_k \Delta w\|_{L^2(\TT)}^2+C\|w\|_{H^1(\TT)}^2
		-2\left<\nabla w, \nabla \left[\xi_k w\Delta \xi_k+
		2\xi_k \nabla w\cdot \nabla \xi_k
		\right]\right>
		\\=\,&\|\xi_k \Delta w\|_{L^2(\TT)}^2+C\|w\|_{H^1(\TT)}^2-2\left<\nabla w, \xi_k \Delta \xi_k \nabla w+w \nabla [\xi_k \Delta \xi_k]\right>\\&-4\left<\nabla w,  (\nabla \xi_k \otimes \nabla \xi_k)\nabla w+
		\xi_k Hw\nabla \xi_k +\xi_kH\xi_k \nabla w
		\right>\\\le \,&
		\|\xi_k \Delta w\|_{L^2(\TT)}^2+C\|w\|_{H^1(\TT)}^2+\left<w^2, \Delta(\xi_k\Delta\xi_k)\right>+2\left<|\nabla w|^2, \divv(\xi_k \nabla \xi_k)\right>\\\le \,&
		\|\xi_k \Delta w\|_{L^2(\TT)}^2+C\|w\|_{H^1(\TT)}^2.
	\end{align*}
We enlarged again the constant $C<\infty$ from line to line. Moreover, in the last line we have employed that
	\[
	\|\Delta(\xi_k \Delta \xi_k)\|_{L^\infty(\TT)}\,=\, (2\pi|k|)^2\|\Delta \xi_k ^2\|_{L^\infty(\TT)}\,\le \,\frac{ 2(2\pi|k|)^2(4\pi |k|)^2 }{1+(2\pi |k|)^2+(2\pi |k|)^4}\,\le \, 8
	\]
	by $\Delta \xi_k =  -(2\pi |k|^2)\xi_k $ and \eqref{squares}. Combining all the previous estimates we finally obtain that
	\begin{align*}&2\left<\nabla
		A^\epsilon w, \nabla w
		\right>+\sum_{l=1}^\infty\left\|\nabla  B(w)[\psi_l] \right\|_{L^2(\TT, \RR^2)}^2 \\\le \,& -2\epsilon\|\Delta w\|_{L^2(\TT)}^2-\sum_{k\in \ZZ^2} \mu_k^2 \left[\|\xi_k\Delta w\|_{L^2(\TT)}^2-C\| w\|_{H^1(\TT)}^2 \right]\\&\hspace{2.8cm}+\sum_{k\in \ZZ^2} \mu_k^2 \left[\|\xi_k\Delta w\|_{L^2(\TT)}^2+C\| w\|_{H^1(\TT)}^2 \right]\\\le \,&
		C_\Lambda\| w\|_{H^1(\TT)}^2-2\epsilon\|\Delta w\|_{L^2(\TT)}^2.
	\end{align*}
	We arrive at \eqref{Eq_coerc_2} by enlarging $C_\Lambda$ by $2$.
\end{proof}
\begin{proof}[Proof of Theorem \ref{Thm_Sec_Reg_SPDE}]The existence and uniqueness assertion follows, if we verify the assumptions of  \cite[Theorem~4.2.4]{liu2015stochastic} on the couple $(A^\epsilon,B)$ considered on the Gelfand triple \[H^2(\TT)\subset H^1(\TT)\subset L^2(\TT).\] Here we equip $H^2(\TT)$ with the equivalent Bessel potential norm to ensure that the usual norm in $L^2(\TT)$ coincides with the norm of the dual of $H^2(\TT)$ under the pairing in $H^1(\TT)$, for details see appendix \ref{AppendixAB}. Hemicontinuity and boundedness  of $A^\epsilon$ follow from $A^\epsilon\in L(H^2(\TT), L^2(\TT))$. Coercivity is obtained by adding \eqref{Eq_coerc_1} and \eqref{Eq_coerc_2} together. By linearity, coercivity implies weak monotonicity. The proof of \eqref{Eq_est_reg_SPDE} translates verbatim from the one-dimensional case \cite[Proposition A.2]{GessGann2020} and \eqref{Eq_cons_mass_stoch_part} follows from testing \eqref{Eq_reg_SPDE} with $\mathbbm{1}_{\TT}$. The claim regarding non negativity of $w$ is a consequence of the maximum principle for second-order parabolic SPDEs \cite[Theorem 4.3]{krylovIto}, which holds by analogous reasoning also on $\TT$.
\end{proof}

\section{Time discretization scheme with degenerate limit}\label{Sec_time_discr}
In this section we fix $N\in \NN$. The goal of this section is to construct for a given end time $T$ and an initial value $u_0$ a weak martingale solution to the split-up problem 
\begin{equation}\label{Eq_scheme_deg}
	\begin{cases}
		u(t)\,=\, v(2(t-j\delta)+j\delta ),& j\delta \le t< (j+\frac{1}{2})\delta,\\
		u(t)\,=\,w(2(t-(j+\frac{1}{2})\delta)+j\delta ),& (j+\frac{1}{2})\le t< (j+1)\delta,\\
		\partial_t v \,=\, - \divv(v^2 \nabla\Delta v), & \text{on }[j\delta, (j+1)\delta), \\
		d w_t \,=\,  \divv(w_t \circ dW_t), & \text{on }[j\delta, (j+1)\delta), \\
	\end{cases}
\end{equation}
where $\delta = \frac{T}{N+1}$  and $j\in \{0,\dots, N \}$. Starting at the initial value $u_0$ the process $u(t)$ satisfies alternately the deterministic thin-film equation  and the purely stochastic  equation \eqref{Eq38} on time intervals of length $\frac{\delta}{2}$ and yields thus a time splitting scheme for the stochastic thin-film equation \eqref{Eq_STFEq}. 
During the construction we derive  bounds which are uniform in $N$, and will be important in the final section, where we take the time step limit $N\to \infty$  to construct a solution to the original problem. We refer the interested reader for more information on the time-splitting procedure to \cite{gyongy_krylov_2003}.
The main statement of this section is the following.

\begin{thm}	\label{Thm_degenerate_sol}
	Let $T\in (0,\infty)$,  $q\in (2, \infty)$ and $\alpha \in (-1,0)$. We assume that $u_0$ is a non negative random variable in $H^1(\TT)$ and set $R^{(k)}=\{k-1\le \|u_0\|_{H^1(\TT)}< k\}$ and $u_0^{(k)}= \mathbbm{1}_{R^{(k)}} u_0$ for every $k\in \NN$. Then there exists a  probability space $(\tilde{\Omega}, \tilde{\mathfrak{A}}, \tilde{P})$ with a filtration $\tilde{\mathfrak{F}}$ satisfying the usual conditions, a family of independent Brownian motions $(\tilde{\beta}^{(l)})_{l\in \NN}$, random variables $\mathbbm{1}_{\tilde{R}^{(k)}}$, $H^1_w(\TT)$-continuous processes $\tilde{u}^{(k)}$  and $L^2(0, T; L^{q'}(\TT))$-valued random variables $\tilde{J}^{(k)}$ for $k\in \NN$, such that  $\tilde{u}^{(k)}$, $\tilde{J}^{(k)}$ and the processes $\tilde{v}^{(k)}$ and $\tilde{w}^{(k)}$ defined by
	\[
	\begin{cases}
		\tilde{u}^{(k)}(t)\,=\, \tilde{v}^{(k)}(2(t-j\delta)+j\delta ),& j\delta \le t< (j+\frac{1}{2})\delta,\\
		\tilde{u}^{(k)}(t)\,=\,\tilde{w}^{(k)}(2(t-(j+\frac{1}{2})\delta)+j\delta ),& (j+\frac{1}{2})\le t< (j+1)\delta
	\end{cases}
	\]
	satisfy the following.
	\begin{enumerate}[label=(\roman*)]
		\item \label{Item_IV}
		The sequence $(\mathbbm{1}_{\tilde{R}^{(k)}}, \tilde{u}^{(k)}(0))_{k\in \NN}$ has the same distribution as $(\mathbbm{1}_{{R}^{(k)}}, {u}_0^{(k)})_{k\in \NN}$, in particular we have $\sum_{k=1}^\infty \tilde{u}^{(k)}(0)\sim u_0$. Moreover, $ \tilde{u}^{(k)}$ and  $\tilde{J}^{(k)}$ are $\tilde{P}$-almost surely zero outside of the set $\tilde{R}^{(k)}$.
		\item \label{Item_meas} $\tilde{u}^{(k)}(t)$  and $\tilde{J}^{(k)}|_{[0,t]}$ are $\tilde{\mathfrak{F}}_t$-measurable as random variables in $H^1(\TT)$ and $L^2(0,t; L^{q'}(\TT))$  for every $t\in [0,T]$ and  $k\in \NN$.
		 \item \label{Item_sol_Det_deg_case} The tuples $(\tilde{v}^{(k)},\tilde{J}^{(k)})$ are $\tilde{P}$-almost surely solutions to the deterministic thin-film equation on $[j\delta, (j+1)\delta)$ satisfying property (iv) from Theorem \ref{Thm_ex_sol_det} with initial value $\tilde{u}^{(k)} (j\delta)$ for every $j=0,\dots, N$.
		   \item \label{Item_SPDE_deg}For $k\in \NN$,  $\varphi\in H^1(\TT)$ and $t\in [j\delta , (j+1)\delta) $ we have 
	 that 
	 \begin{align*}\left< \tilde{w}^{(k)}(t),
	 	\varphi\right>\,-\,\left< \tilde{w}^{(k)}(j\delta),
	 	\varphi\right>\,=\,&  \frac{1}{2}\sum_{l=1}^{\infty} \lambda_l^2\int_{j\delta }^t \left<  \divv (\divv(\tilde{w}^{(k)}(s)\psi_l) \psi_l ), \varphi \right>\, ds
	 	\\&+\, \sum_{l=1}^\infty \lambda_l \int_{j
	 		\delta}^t\left< \divv(\tilde{w}^{(k)}(s)\psi_l), \varphi \right>\, d\beta^{(l)}_s.
	 \end{align*}
	 \item \label{Item_est_deg} For every $k\in \NN${, $p\in (0,\infty)$} we have  
	 \begin{align*}
	 		\tilde{E}\left[
	 		\sup_{0\le t\le T} \|\tilde{u}^{(k)}\|_{H^1(\TT)}^p
	 		\right]\, &\lesssim_{\Lambda, p, T} \, E\left[\|u^{(k)}_0\|^p_{H^1(\TT)}\right],\\
	 		\tilde{E}\left[ \|\tilde{J}^{(k)}\|_{L^2(0,T; L^{q'}(\TT))}^\frac{p}{2} \right]\, &\lesssim_{\Lambda, p,q , T}\, E\left[\|u^{(k)}_0\|^p_{H^1(\TT)}\right].
	 \end{align*}
 \item \label{Item_est_Hoelder} Moreover, for any $\gamma \in (0, \frac{1}{2})$ and $K\in (1, \infty)$ it holds
 \[
 \tilde{P}\left(\left\{\left[\tilde{u}^{(k)}\right]_{\gamma ,W^{-1, q'}(\TT)} >K \right\}\right)
 \,\lesssim_{\Lambda, q, \gamma, T}\, \frac{1+E\left[\|u^{(k)}_0\|_{H^1(\TT)}^2\right]}{K}.
 \]
	\end{enumerate}
\end{thm}

\subsection{Construction and analysis of a regularized scheme}
Let $u_0\in L^\infty(\Omega , H^1(\TT))$ be non negative. Up to extension and completion of the probability space we can assume that there exists a filtration $\mathfrak{F}$ satisfying the usual conditions with a family of independent Brownian motions $(\beta^{(l)})_{l\in \NN}$ such that $u_0$ is $\mathfrak{F}_0$-measurable.
\begin{rem}
	The construction with initial value $u_0$ within this subsection will in subsection \ref{Sec_deg_lim} be applied to each of the cut-off parts $u_0^{(k)}$ from  Theorem \ref{Thm_degenerate_sol}. This justifies the strong assumption $u_0\in L^\infty(\Omega , H^1(\TT))$ here.
\end{rem}
We fix for the rest of this subsection also   $T\in (0,\infty)$,  $q\in (2, \infty)$, $\alpha \in (-1,0)$, $\epsilon \in (0,1)$ and apply the operator $\mathcal{S}_{\alpha, q, \delta}$ from Corollary \ref{Cor_mble_sol_map} to the initial value $u_0$. We define $v_\epsilon|_{[0, \delta )}, J_\epsilon|_{[0, \delta)}$ as the version of the solution which is in $C(0,\delta; L^2(\TT))$ and in particular continuous in $H^1_w(\TT)$, see Remark \ref{Rem_version_det_sol}. Moreover, we define $w_\epsilon|_{[0,\delta)}$ as the solution to \eqref{Eq_reg_SPDE}  with initial value $\lim_{t\nearrow \delta } v_\epsilon(t)$. Notice that since $v_\epsilon|_{{[0,\delta)}}$ fulfills the properties (i) and (ii) of Theorem \ref{Thm_ex_sol_det}, we have 
\[
E\left[\|\lim_{t\nearrow \delta } v_\epsilon\|_{H^1(\TT)}^p\right] \, {\lesssim }\, 
E\left[\|u_0\|_{H^1(\TT)}^p\right] 
\]
for any $p\in [2, \infty)$,
and therefore  Theorem \ref{Thm_Sec_Reg_SPDE} is indeed applicable and yields a non-negative solution $w_\epsilon|_{[0, \delta)}$. In particular, the terminal value $\lim_{t\nearrow \delta } w_\epsilon(t)$ lies again in $L^p(\Omega , H^1(\TT))$. We repeat this and obtain inductively weak solutions $v|_{[j\delta, (j+1)\delta)}$ to \eqref{Eq_det_eq} and variational solutions $w_\epsilon|_{[j\delta, (j+1)\delta)}$ to \eqref{Eq_reg_problem} for $j\in \{1, \dots, N\}$. Finally, we define the $H^1_w(\TT)$-continuous, adapted process 
\[
u_{\epsilon}(t )=\begin{cases}
	v_{\epsilon}(2(t-j\delta)+j\delta ),& j\delta \le t< (j+\frac{1}{2})\delta,\\
	w_{\epsilon}(2(t-(j+\frac{1}{2})\delta)+j\delta ),& (j+\frac{1}{2})\le t< (j+1)\delta.
\end{cases}
\]
for $t\in [0,T)$. We note that we set for the final time $u_\epsilon(T)= \lim_{t\nearrow \delta } w_\epsilon(t)$. The divergence form of \eqref{Eq_det_eq},  \eqref{Eq_reg_problem}, and an application of It\^o's formula yield the following estimates along the whole time-splitting scheme. 

\begin{lemma}\label{Lemma_Est}
	It holds almost surely that
	\begin{equation}\label{Eq_cons_mass_reg_scheme}
		\int_{\TT} u_\epsilon(t)\, dx\, =\, 
		\int_{\TT} u_0\, dx.
	\end{equation} for all $t\in [0,T]$.
Moreover,we have additionally 
\begin{equation}\label{Eq_Energy_est_reg_scheme}
E\left[
\sup_{0\le t\le T} \|u_\epsilon\|_{H^1(\TT)}^p
\right]\, \lesssim_{\Lambda, p, T} \, E\left[\|u_0\|^p_{H^1(\TT)}\right]
\end{equation}
for $p\in (0, \infty)$.
\end{lemma}
\begin{proof}
	The equality \eqref{Eq_cons_mass_reg_scheme} follows from its respective counterparts from Theorem \ref{Thm_ex_sol_det} (i) and \eqref{Eq_cons_mass_stoch_part}. Next, we apply It\^o's formula to the composition of the functional $\|\nabla \cdot\|_{L^2(\TT,\RR^2)}^2$   with the process $w_\epsilon$, which yields that
		\begin{align}\begin{split}
			\label{Eq_Ito1}
			\|\nabla w_\epsilon(t) \|_{L^2(\TT,\RR^2)}^2\,=\, \|\nabla w_\epsilon(j\delta) \|_{L^2(\TT, \RR^2)}^2\,+2\int_{j\delta}^t\left<\nabla {w_\epsilon}(s),\nabla A^\epsilon ({w_\epsilon}(s))\right>\,ds&\\+\,\sum_{l=1}^\infty \lambda_l\int_{j\delta}^t2
			\left<\nabla \divv(w_\epsilon(s)\psi_l), \nabla w_\epsilon (s) \right>\, d\beta^l_s\,+\, \sum_{l=1}^\infty \lambda_l^2 \int_{j\delta}^t\|\nabla 
			\divv({w_\epsilon}(s)\psi_l)\|_{L^2(\TT, \RR^2)}^2
			\,ds&.
		\end{split}
	\end{align}
 for $t\in [j\delta, (j+1)\delta)$.  A justification of the applicability of It\^o's formula is given in appendix \ref{AppendixB}. As pointed out in \eqref{Eq99}, the martingale given by the series of stochastic integrals, which we denote by $M_{2,j}$, has quadratic variation 
 \[4
 \sum_{l=1}^{\infty}\lambda_l^2\int_{j\delta}^t
 \left<\nabla \divv(w_\epsilon(s)\psi_l), \nabla w_\epsilon(s) \right>^2\, ds{.}
 \]
  Combining \eqref{Eq_Ito1} with \eqref{Eq_coerc_2} we conclude  that 
  \begin{align*}&
 	\|\nabla w_\epsilon(t) \|_{L^2(\TT, \RR^2)}^2\, - \, \|\nabla w_\epsilon(j\delta) \|_{L^2(\TT, \RR)}^2 - M_{2,j}(t) \,\lesssim_{\Lambda} \,
 	\int_{j\delta}^t \|w_\epsilon(s)\|_{H^1(\TT)}^2\, ds
 \end{align*}
and for the  endpoint $t=(j+1)\delta$  \begin{align}\begin{split}
		\label{Eq_aux14}&
		\|\nabla v_\epsilon((j+1)\delta) \|_{L^2(\TT, \RR^2)}^2\, - \, \|\nabla w_\epsilon(j\delta) \|_{L^2(\TT, \RR)}^2 - M_{2,j}((j+1)\delta)) \, \lesssim_{\Lambda} \, 
		\int_{j\delta}^{(j+1)\delta} \|w_\epsilon(s)\|_{H^1(\TT)}^2\, ds.
	\end{split}
\end{align}
 By Theorem \ref{Thm_ex_sol_det} \ref{Item_det_energy_est} we have
 \[
 \|\nabla w_{N,\epsilon}(j\delta)\|_{L^2(\TT,\RR)}^2\, \le \,
 \|\nabla v_{N,\epsilon}(j\delta)\|_{L^2(\TT,\RR)}^2,
 \] 
 such that a telescoping sum argument yields
 \begin{equation}\label{Eq_2_martingale_expansion}
 	\|\nabla w_\epsilon(t) \|_{L^2(\TT, \RR^2)}^2\, - \, \|\nabla u_0 \|_{L^2(\TT, \RR)}^2 - M_{2}(t) \, \lesssim_{\Lambda} \, 
 	\int_0^t \|w_\epsilon(s)\|_{H^1(\TT)}^2\, ds
 \end{equation}
 for $t\in [0, T]$. The appearing process $M_2$ is defined by the sum of orthogonal  martingales
 \begin{equation*}
 	M_2(t)\, =\, \sum_{k=0}^{j-1}
 	M_{2,k}((k+1)\delta) \,+\,	M_{2,j}(t), \;\; t\in [j\delta,(j+1)\delta)
 \end{equation*} and has therefore quadratic variation
 \begin{equation*}
 	4\sum_{l=1}^{\infty}\lambda_l^2\int_{0}^t 
 	\left<\nabla \divv(w_\epsilon(s)\psi_l), \nabla w_\epsilon(s) \right>^2\, ds.
 \end{equation*} For $p\ge 2$ we deduce 
 from \eqref{Eq_2_martingale_expansion} with help of the inequality $(a+b+c)^{\frac{p}{2}}\lesssim_p a^\frac{p}{2}+b^\frac{p}{2}+c^\frac{p}{2}$ and the Burkholder-Davis-Gundy inequality that 
\begin{align}\begin{split}\label{Eq_aux15}
		&
		E\left[\sup_{0\le s\le t}	\|\nabla w_\epsilon(s) \|_{L^2(\TT, \RR^2)}^p\right]\, - \, {C_p}E\left[ \|\nabla u_0 \|_{L^2(\TT, \RR)}^p \right] \\ \lesssim_{\Lambda, p} \, &E \left[
		\left(\int_0^t \|w_\epsilon(s)\|_{H^1(\TT)}^2\, ds \right)^\frac{p}{2}\,+\,\left(\int_{0}^t 
		\left<\nabla \divv(w_\epsilon(s)\psi_l), \nabla w_\epsilon(s) \right>^2\, ds\right)^\frac{p}{4} \right].
	\end{split}
\end{align}
To estimate the latter expression we observe that
\[
\nabla\divv(w\psi_l)\,=\, Hw \psi_l\, +\,D\psi_l \nabla w\,+\,w\nabla \divv\psi_l\, +\,
\divv(\psi_l)\nabla w
\]
and due to \eqref{boundedness_psil} and \eqref{aux_eq16} consequently
\begin{equation}\label{Eq_aux1}
	\left|\left<
	\nabla\divv(w\psi_l), \nabla w
	\right>\right|\,\lesssim\,\|\nabla w\|_{L^2(\TT, \RR)}\|w\|_{H^1(\TT)}
\end{equation}
for $w\in H^2(\TT)$.
We conclude with help of Young's inequality that
\begin{align*}&
	E\left[\left(\int_{0}^t 
	\left<\nabla \divv(w_\epsilon(s)\psi_l), \nabla w_\epsilon(s) \right>^2\, ds\right)^\frac{p}{4} \right]\, {\lesssim } \,
		E\left[	\sup_{0\le s\le t} \|\nabla w_\epsilon(s) \|_{L^2(\TT, \RR^2)}^\frac{p}{2}\left(\int_0^t \|w_\epsilon(s)\|_{H^1(\TT)}^2\, ds\right)^\frac{p}{4} \right]\\\le \, & 
		\frac{\kappa}{2} 	E\left[{\sup_{0\le s\le t} \|\nabla w_\epsilon(s)\|_{L^2(\TT,\RR^2)}^{p}}\right]\,+\, \frac{1}{2 \kappa }
		E\left[ \left(\int_0^t \|w_\epsilon(s)\|_{H^1(\TT)}^2\, ds \right)^\frac{p}{2}\right]
\end{align*}
for any $\kappa >0$. An appropriate choice of $\kappa$ and \eqref{Eq_aux15} yield that
\begin{align*}
		&\frac{1}{2}
		E\left[\sup_{0\le s\le t}	\|\nabla w_\epsilon(s) \|_{L^2(\TT, \RR^2)}^p\right]\, - \, {C_p} E\left[ \|\nabla u_0 \|_{L^2(\TT, \RR)}^p \right] \, \lesssim_{\Lambda, p} \, E \left[
		\left(\int_0^t \|w_\epsilon(s)\|_{H^1(\TT)}^2\, ds \right)^\frac{p}{2} \right]\\\lesssim_{p,T}\,&
		 E \left[
		\int_0^t \|w_\epsilon(s)\|_{H^1(\TT)}^p\, ds  \right]\, \lesssim_p\,
		 E \left[
		\left(\int_0^t  \left(\int_{\TT} u_0\, dx\right)^p\,+\,\|\nabla w_\epsilon(s)\|_{L^2(\TT, \RR^2)}^p\, ds \right) \right]\\\lesssim_{T}\,&  E\left[
		\|u_0\|^p_{L^2(\TT)}
		\right] \, +\, \int_0^t E\left[
		\sup_{0\le \tau \le s}	\|\nabla w_\epsilon(\tau) \|_{L^2(\TT, \RR^2)}^p
		\right]\, ds.
\end{align*}
We additionally  employed Jensen's and the  Poincar\'e inequality here. Since $u_0\in L^p(\Omega, H^1(\TT))$,  the monotone function
\[
t\, \mapsto \, E\left[\sup_{0\le s\le t}	\|\nabla w_\epsilon(s) \|_{L^2(\TT, \RR^2)}^p\right]
\]
takes finite values by \eqref{Eq_est_reg_SPDE}  and Theorem \ref{Thm_ex_sol_det} (i), (ii) and therefore an application of Gr\"onwall's inequality yields
\begin{equation}\label{Eq_aux11}
E\left[\sup_{0\le s\le T}	\|\nabla w_\epsilon(s) \|_{L^2(\TT, \RR^2)}^p\right]
\,\lesssim_{\Lambda, p, T} \, E\left[ \| u_0 \|_{H^1(\TT)}^p \right] 
\end{equation}
In order to obtain the above inequality  also for $p\in (0,2)$ we observe that 
$\mathbbm{1}_Ru_\epsilon$ coincides with the process $u_{\epsilon,R}$ obtained by constructing the splitting scheme with  initial value $\mathbbm{1}_Ru_0$ 
for  $R\in \mathfrak{F}_0$. Indeed, from the properties in Theorem \ref{Thm_ex_sol_det} (i)
we conclude that $
\mathcal{S}_{\alpha, q,\delta}$ maps $0$ to the solution which is  $0$ for all times. Consequently, we have 
\[
v_{\epsilon,R}|_{[0,\delta)}\, =\, \mathbbm{1}_Rv_\epsilon|_{[0,\delta)}\;\;\text{ and }\;\; w_{\epsilon,R}(0 )\,=\, \mathbbm{1}_Rw_\epsilon(0 ).
\]
Therefore $w_{\epsilon,R}|_{[0,\delta)}$ and  $\mathbbm{1}_Rw_\epsilon|_{[0,\delta)}$ are both solutions to \eqref{Eq_reg_SPDE} and have the same initial value such that $w_\epsilon^{(R)}|_{[0,\delta)}\, =\, \mathbbm{1}_Rw_\epsilon|_{[0,\delta)}$. It is left to apply the uniqueness statement from Theorem \ref{Thm_Sec_Reg_SPDE} and repeat these arguments on $[j\delta, (j+1)\delta)$ for $j=1, \dots N$. Hence applying \eqref{Eq_aux11} to $w_{\epsilon, R}$ with exponent $p=2$ yields
\[
E\left[
\mathbbm{1}_R \sup_{0\le s\le T}	\|\nabla w_\epsilon(s) \|_{L^2(\TT, \RR^2)}^2\right]\, \lesssim_{\Lambda, T} \, E\left[ 
\mathbbm{1}_R
\|u_0\|_{H^1(\TT)}^2
\right].
\]
Since $R\in \mathfrak{F}_0$ was arbitrary, it follows that
\[
E\left[\left. \sup_{0\le s\le T}	\|\nabla w_\epsilon(s) \|_{L^2(\TT, \RR^2)}^2\right|\mathfrak{F}_0\right]\, \lesssim_{\Lambda, T} \,
\|u_0\|_{H^1(\TT)}^2.
\]
For $p\in (0,2)$ we can use Jensen's inequality to deduce that
\[
E\left[\left.\sup_{0\le s\le T}	\|\nabla w_\epsilon(s) \|_{L^2(\TT, \RR^2)}^p\right|\mathfrak{F}_0\right]\,\le\,
E\left[\left. \sup_{0\le s\le T}	\|\nabla w_\epsilon(s) \|_{L^2(\TT, \RR^2)}^2\right|\mathfrak{F}_0\right]^\frac{p}{2}
\, \lesssim_{\Lambda, T}\,
\|u_0\|_{H^1(\TT)}^p
\]
and it is left to take the expectation.
Finally, we use \ref{Thm_ex_sol_det} (ii) to obtain \eqref{Eq_aux11} with $w_\epsilon$ replaced by $u_\epsilon$ which  together with \eqref{Eq_cons_mass_reg_scheme} implies \eqref{Eq_Energy_est_reg_scheme}. 
\end{proof}
\begin{lemma}\label{Lemma_J_new} We have 
	\[
	E\left[ \|J_\epsilon\|_{L^2(0,T; L^{q'}(\TT))}^\frac{p}{2} \right]\, \lesssim_{\Lambda, p,q , T}\, E\left[\|u_0\|^p_{H^1(\TT)}\right]\]
	 for $p\in (0, \infty)$.
\end{lemma}
\begin{proof}
	We observe that as a consequence of Theorem \ref{Thm_ex_sol_det} (iii) and \eqref{Eq_cons_mass_reg_scheme}
	\begin{align*}
		\|J_\epsilon\|_{L^2(j\delta, (j+1)\delta; L^{q'}(\TT))}^2\,\lesssim_{q} \,& \|\nabla v_{\epsilon}(j\delta)\|^4_{L^2(\TT,\RR^2)}-\|\nabla w_{\epsilon}(j\delta)\|^4_{L^2(\TT,\RR^2)}\\&+\,\left(\int_{\TT}u_0\, dx\right)^2 \left(\|\nabla v_{\epsilon}(j\delta)\|^2_{L^2(\TT,\RR^2)}-\|\nabla w_{\epsilon}(j\delta)\|^2_{L^2(\TT,\RR^2)}\right).  
	\end{align*}
	Using that 
	\begin{align*}
		\|J_\epsilon\|_{L^2(0,T; L^{q'}(\TT))}^2 \,=\,
		\sum_{j=0}^{N} \|J_\epsilon\|_{L^2(j\delta, (j+1)\delta; L^{q'}(\TT))}^2
	\end{align*}
	we obtain the bound
	\begin{align*}
			\|J_\epsilon\|_{L^2(0,T; L^{q'}(\TT))}^2 \,\lesssim_{q} \,& \sum_{j=0}^{N} \|\nabla v_{\epsilon}(j\delta)\|^4_{L^2(\TT,\RR^2)}-\|\nabla w_{\epsilon}(j\delta)\|^4_{L^2(\TT,\RR^2)}\\&+\,\sum_{j=0}^{N}\left(\int_{\TT}u_0\, dx\right)^2 \left(\|\nabla v_{\epsilon}(j\delta)\|^2_{L^2(\TT,\RR^2)}-\|\nabla w_{\epsilon}(j\delta)\|^2_{L^2(\TT,\RR^2)}\right).
	\end{align*}
	Applying the $\frac{p}{4}$-th power and using that $(a_1+\dots+a_4)^\frac{p}{4}\lesssim_p a_1^\frac{p}{4}+\dots +a_4^\frac{p}{4}$ we conclude 
	\begin{align}\begin{split}\label{Eq_aux10}
			\|J_\epsilon\|_{L^2(0,T; L^{q'}(\TT))}^\frac{p}{2} \,\lesssim_{p,q} \,& \left|0\vee\left(
				\|\nabla u_0\|^4_{L^2(\TT,\RR^2)}-\|\nabla w_{\epsilon}(T-\delta)\|^4_{L^2(\TT,\RR^2)}\right)\right|^\frac{p}{4}
			\\&\,+
	\left|\sum_{j=0}^{N-1} \|\nabla v_{\epsilon}((j+1)\delta)\|^4_{L^2(\TT,\RR^2)}-\|\nabla w_{\epsilon}(j\delta)\|^4_{L^2(\TT,\RR^2)}\right|^\frac{p}{4}\\&+\, \left(\int_{\TT}u_0\, dx\right)^\frac{p}{2} \left|0\vee\left(\|\nabla u_0\|^2_{L^2(\TT,\RR^2)}-\|\nabla w_{\epsilon}(T-\delta)\|^2_{L^2(\TT,\RR^2)}\right)\right|^\frac{p}{4}\\&+\,\left|\left(\int_{\TT}u_0\, dx\right)^2
				\sum_{j=0}^{N-1} \|\nabla v_{\epsilon}((j+1)\delta)\|^2_{L^2(\TT,\RR^2)}-\|\nabla w_{\epsilon}(j\delta)\|^2_{L^2(\TT,\RR^2)}\right|^\frac{p}{4}.
		\end{split}
	\end{align}
	The expectation of the first and the third summand of the right-hand side of \eqref{Eq_aux10} can be each estimated by 
	$E\left[\|u_0\|_{H^1(\TT)}^p\right]$. To control also the second term we apply It\^o's formula, see e.g. \cite[Theorem 15.19]{kallenberg1997foundations}, to the composition of $(\cdot)^2$ with  the real-valued semimartingale \eqref{Eq_Ito1} and obtain that
	\begin{align*}
			\|\nabla w_\epsilon(t) \|_{L^2(\TT, \RR^2)}^4\,=& \, \|\nabla w_\epsilon(j\delta) \|_{L^2(\TT, \RR)}^4\,+\, 2\int_{j\delta}^t
			\|\nabla w_\epsilon(s)\|_{L^2(\TT,\RR^2)}^{2}\, dM_{2,j}(s)\\&+\,4\int_{j\delta}^t \|\nabla w_\epsilon(s) \|_{L^2(\TT,\RR^2)}^{2} \left<\nabla {w_\epsilon}(s),\nabla A^\epsilon ({w_\epsilon}(s))\right>\, ds
			\\&+\,2\int_{j\delta}^t \|\nabla w_\epsilon(s) \|_{L^2(\TT,\RR^2)}^{2} \sum_{l=1}^\infty \lambda_l^2 \int_{j\delta}^t\|\nabla 
			\divv({w_\epsilon}(s)\psi_l)\|_{L^2(\TT, \RR^2)}^2\, ds
			\\&+\,4
			\sum_{l=1}^\infty \lambda_l^2 \int_{j\delta}^t
			\left<\nabla \divv({w_\epsilon}(s)\psi_l), \nabla w_\epsilon(s) \right>^2\, ds
			.
	\end{align*}
Using \eqref{Eq_coerc_2}, \eqref{Eq_aux1} we conclude for the endpoint $t=(j+1)\delta$ that 
\begin{align*}&
	\|\nabla v_\epsilon((j+1)\delta) \|_{L^2(\TT, \RR^2)}^4\,- \, \|\nabla w_\epsilon(j\delta) \|_{L^2(\TT, \RR)}^4\,-\,2\int_{j\delta}^{(j+1)\delta}
	\|\nabla w_\epsilon(s)\|_{L^2(\TT,\RR^2)}^{2}\, dM_{2,j}(s) \\ \lesssim_\Lambda  \, & \int_{j\delta}^{(j+1)\delta}{\| 
{w_\epsilon}(s)\|_{H^1(\TT)}^4}\, ds
\end{align*}
Summing up over $j$, taking the $\frac{p}{4}$-power, applying the BDG-inequalities, \eqref{Eq_aux1} and \eqref{Eq_Energy_est_reg_scheme} yields that
\begin{align*}&
	E\left[\left|\sum_{j=0}^{N-1} \|\nabla v_{\epsilon}((j+1)\delta)\|^4_{L^2(\TT,\RR^2)}-\|\nabla w_{\epsilon}(j\delta)\|^4_{L^2(\TT,\RR^2)}\right|^\frac{p}{4} \right]\\\lesssim_{\Lambda, p}
	\,&E\left[\left(\int_0^{T-\delta}\| w_\epsilon(s)\|_{H^1(\TT)}^{8}\, ds\right)^\frac{p}{8}\,+\,\left(\int_0^{T-\delta} \|
	{w_\epsilon}(s)\|_{H^1(\TT)}^4\, ds\right)^\frac{p}{4}\right]
	\\\lesssim_{p,T}
	\,&E\left[
	\sup_{0\le t\le T} \|w_\epsilon(t)\|^p
	\right]\, \lesssim_{\Lambda, p, T} \, E\left[\|u_0\|^p_{H^1(\TT)}\right].
\end{align*} Similarly, by summing up over $j$
 in \eqref{Eq_aux14} and taking the power $\frac{p}{4}$ we obtain 
\begin{align*}&
	\left|\left(\int_{\TT}u_0\, dx\right)^2
	\sum_{j=0}^{N-1} \|\nabla v_{\epsilon}((j+1)\delta)\|^2_{L^2(\TT,\RR^2)}-\|\nabla w_{\epsilon}(j\delta)\|^2_{L^2(\TT,\RR^2)}\right|^\frac{p}{4}\\\lesssim_{\Lambda, p}\,&
	\left|
	\left(\int_{\TT}u_0\, dx\right)^2 M_2(T-\delta)
	\right|^\frac{p}{4}\,+\, \left(\left(\int_{\TT}u_0\, dx\right)^2
	\int_{0}^{T-\delta} \|w_\epsilon(s)\|_{H^1(\TT)}^2\, ds
	\right)^\frac{p}{4}
\end{align*}
Taking the expectation, using the Burkholder-Davis-Gundy inequality, \eqref{Eq_aux1}and \eqref{Eq_Energy_est_reg_scheme} yields the estimate
\begin{align*}&
	E\left[
	\left|\left(\int_{\TT}u_0\, dx\right)^2
	\sum_{j=0}^{N-1} \|\nabla v_{\epsilon}((j+1)\delta)\|^2_{L^2(\TT,\RR^2)}-\|\nabla w_{\epsilon}(j\delta)\|^2_{L^2(\TT,\RR^2)}\right|^\frac{p}{4}
	\right]\\\lesssim_{ \Lambda, p}\,&
	E\left[\left(
	\left(\int_{\TT}u_0\, dx\right)^4\int_0^{T-\delta} \|w_\epsilon(s)\|^4\, ds\right)^\frac{p}{8}\,+\,
	\left(
	\left(\int_{\TT}u_0\, dx\right)^2\int_0^{T-\delta} \|w_\epsilon(s)\|^2\, ds\right)^\frac{p}{4}
	\right]\\\lesssim_{p,T}\,&
	\sqrt{E\left[\left(\int_{\TT}u_0\, dx\right)^p\right]
E\left[\sup_{0\le t\le T} \|w_\epsilon(t)\|^p_{H^1(\TT)}\right]	
}\,\lesssim_{\Lambda,  p, T}\, E\left[\|u_0\|_{H^1(\TT)}^p\right].
\end{align*}
Finally, taking the expectation of \eqref{Eq_aux10} and using the estimates on the individual summands yields the claim.
\end{proof}
We also show tail estimates of the powers of $v_\epsilon$ in their respective space. We note that the obtained bound depends on $N$ and will therefore be improved to a bound, which is uniform in $N$ after letting $\epsilon\searrow 0$.
\begin{lemma}\label{Lemma_power_bounds_I}
	We have for $K\in (1, \infty)$ the estimate
	\begin{align*}&
	P\left(\left\{
	\left\|v_{\epsilon}^{\frac{\alpha+3}{2}}\right\|^2_{L^2(0,T;H^2(\TT))}+\,
	\left\|v_{ \epsilon}^{\frac{\alpha+3}{4}}\right\|^4_{L^4(0,T;W^{1,4}(\TT))}\,>
	\,K
	\right\}\right)\\ \lesssim_{\Lambda, \alpha, T, N} \,& \frac{E\left[\|u_0\|_{H^1(\TT)}^{\alpha+1}\right] \,+\,
		E\left[\|u_0\|_{H^1(\TT)}^2\right]}{K^\frac{2}{\alpha+3}}.
	\end{align*}
\end{lemma}
\begin{proof}
	 As a consequence of  Theorem \ref{Thm_ex_sol_det} (iv), \eqref{Eq_alpha_entropy_expl} and H\"older's inequality we have
	\begin{align*}&
		\int_{j\delta}^{(j+1)\delta} \int_{\TT} |H v_{\epsilon}^{\frac{\alpha+3}{2}}|^2+|\nabla v_{\epsilon}^{\frac{\alpha+3}{4}}|^4 \, dx\, dt\,\lesssim \, 
		\int_{\TT} G_\alpha( v_{\epsilon}(j\delta)) -  G_\alpha( w_{\epsilon}(j\delta)) \,dx \\\lesssim_\alpha& \, \|v_{\epsilon}(j\delta)\|_{L^2(\TT)}^{\alpha+1}\,+\,\|v_{\epsilon}(j\delta)\|_{L^2(\TT)}\, +\,\|w_{\epsilon}(j\delta)\|_{L^2(\TT)}^{\alpha+1}\, +\,\|w_{\epsilon}(j\delta)\|_{L^2(\TT)}.
	\end{align*}
	Summing up over $j$ and taking the expectation yields that 
	\begin{equation}\label{Eq_aux8} E \left[
	\int_{0}^{T} \int_{\TT} |H v_{\epsilon}^{\frac{\alpha+3}{2}}|^2+|\nabla v_{\epsilon}^{\frac{\alpha+3}{4}}|^4 \, dx\, dt\right]\, \lesssim_{\alpha, N} \, E \left[ \sup_{0\le  t\le T} \|u_\epsilon(t)\|_{H^1(\TT)}^{\alpha+1} + \sup_{0\le  t\le T} \|u_\epsilon(t)\|_{H^1(\TT)} \right].
	\end{equation}
	Moreover, by the Sobolev embedding theorem we have the estimate 
	\begin{equation*}
		\int_0^T 
		\int_{\TT} \left|v_{\epsilon}^{\frac{\alpha+3}{2}}(t)\right|^2\, dx
		\, dt \, =\,  \int_0^T 
		\int_{\TT}  \left|v_{\epsilon}^{\frac{\alpha+3}{4}}(t)\right|^4\, dx
		\, dt\,\lesssim_{\alpha} \, \int_0^T \|v_\epsilon\|_{H^1(\TT)}^{\alpha+3}
		\, dt
	\end{equation*}
	which implies by taking the $\frac{2}{\alpha+3}$-th power and the expectation that
	\begin{equation}\label{Eq_aux9}
		E\left[\left(\int_0^T 
		\int_{\TT} \left|v_{\epsilon}^{\frac{\alpha+3}{2}}(t)\right|^2 +  \left|v_{\epsilon}^{\frac{\alpha+3}{4}}(t)\right|^4\, dx
		\, dt\right)^\frac{2}{\alpha+3}\right]\,\lesssim_{\alpha,T}\,  E \left[ \sup_{0\le  t\le T} \|u_\epsilon(t)\|_{H^1(\TT)}^{2}\right].
	\end{equation}
	Combining \eqref{Eq_aux8} and \eqref{Eq_aux9} with Chebyshev's inequality yields respectively that 
	\begin{align*}&
	P\left(
	\left\{
	\int_{0}^{T} \int_{\TT} |H v_{\epsilon}^{\frac{\alpha+3}{2}}|^2+|\nabla v_{\epsilon}^{\frac{\alpha+3}{4}}|^4 \, dx\, dt
	\,>\, K\right\}
	\right)\\  \lesssim_{\alpha, N} \, &
	\frac{1}{K}E \left[ \sup_{0\le  t\le T} \|u_\epsilon(t)\|_{H^1(\TT)}^{\alpha+1} + \sup_{0\le  t\le T} \|u_\epsilon(t)\|_{H^1(\TT)} \right], \\&
	P\left(
	\left\{\int_0^T 
	\int_{\TT} \left|v_{\epsilon}^{\frac{\alpha+3}{2}}(t)\right|^2 +  \left|v_{\epsilon}^{\frac{\alpha+3}{4}}(t)\right|^4\, dx
	\, dt
	\,>\, K\right\}
	\right)\\
	\lesssim_{\alpha, T  }\,&\frac{1}{K^\frac{2}{\alpha+3}}E\left[\sup_{0\le  t\le T} \|u_\epsilon(t)\|_{H^1(\TT)}^2\right].
	\end{align*}Combining these estimates, the assumption  $K\in (1, \infty)$ and the interpolation inequality
	\begin{equation}\label{Eq_Interp1}\|f\|_{H^2(\TT)}^2\, \lesssim  \, 
	\int_{\TT} |
	f|^2 +|Hf|^2\,dx, \;\; f\in H^2(\TT),
	\end{equation} we obtain that
	\begin{align*}&
		P\left(\left\{ 
		\left\|v_{\epsilon}^{\frac{\alpha+3}{2}}\right\|^2_{L^2(0,T;H^2(\TT))}+\,
		\left\|v_{ \epsilon}^{\frac{\alpha+3}{4}}\right\|^4_{L^4(0,T;W^{1,4}(\TT))}\,>
		\,K\right\}\right)\\\lesssim_{\alpha, T, N}\, &\frac{1}{K^\frac{2}{\alpha+3}}E\left[\sup_{0\le t\le T} \|w_\epsilon(t)\|_{H^1(\TT)}^{\alpha+1}+\sup_{0\le t\le T} \|w_\epsilon(t)\|_{H^1(\TT)}^2\right].
	\end{align*}It is left to apply  \eqref{Eq_Energy_est_reg_scheme} to conclude the claim.
\end{proof}
In the final part of the analysis of the approximate scheme, we show H\"older regularity in time of $u_\epsilon$.
\begin{lemma} Let $\gamma \in (0, \frac{1}{2})$ and $K\in (1, \infty)$, then
	\begin{equation}\label{Eq_aux2}
	P\left(\left\{\left[u_\epsilon\right]_{\gamma ,W^{-1, q'}(\TT)} >K \right\}\right)\,
	\lesssim_{\Lambda, q, \gamma, T}\, \frac{1+E\left[\|u_0\|_{H^1(\TT)}^2\right]}{K}.
	\end{equation}
\end{lemma}
\begin{proof} We divide the proof into three steps. 

\textbf{Step 1} (Deterministic integrals).
By H\"older's inequality we have 
	\begin{align}\begin{split}\label{Eq_Hoelder_J_Int}
			\left\|
\int_{s}^{t } \divv J_{\epsilon}(\tau)\,d\tau\right\|_{W^{-1,q'}(\TT)}\,&\le \,
\int_{s}^{t } \left\|\divv J_{\epsilon}(\tau)\right\|_{W^{-1,q'}(\TT)}\,d\tau\\& \lesssim\,
|t-s|^\frac{1}{2}\|  J_{\epsilon}\|_{L^2(0,T,L^{q'}(\TT))}
		\end{split}
\end{align}for any $s,t\in [0,T]$. Analogously, using that $A^\epsilon$ maps $H^1(\TT)$ continuously to $H^{-1}(\TT)$ due to \eqref{boundedness_psil} (with operator norm depending solely on $\Lambda$) we obtain that
\begin{align}\begin{split}\label{Eq_Hoelder_A_Int}
		\left\|
	\int_{s}^{t } A^\epsilon w_\epsilon(\tau) \,d\tau\right\|_{H^{-1}(\TT)}\,&\le \,
	\int_{s}^{t } \left\| A^\epsilon w_\epsilon(\tau)  \right\|_{H^{-1}(\TT)}\,d\tau\\& \lesssim_\Lambda\,
	|t-s|\sup_{0\le \tau\le T} \| w_\epsilon(\tau) \|_{H^1(\TT)}.
\end{split}
\end{align}
From the above inequalities, Lemma \ref{Lemma_J_new}, and \eqref{Eq_Energy_est_reg_scheme}, we deduce that 
\begin{align}\begin{split}
		\label{Eq_aux16}
	P\left(\left\{\left[\int_0^\cdot
	\divv J_{\epsilon}(s)\,ds\right]_{\frac{1}{2}, W^{-1, q'}(\TT)}>K
	\right\}\right)\, &
	\lesssim_{\Lambda,q , T}\, \frac{E\left[\|u_0\|^2_{H^1(\TT)}\right]}{K},\\
	P\left(\left\{\left[\int_0^\cdot
	 A^\epsilon w_{\epsilon}(s)\,ds\right]_{1, H^{-1}(\TT)}>K
	\right\}\right)\, &
	\lesssim_{\Lambda , T}\, \frac{E\left[\|u_0\|^2_{H^1(\TT)}\right]}{K^2}.
	\end{split}
\end{align}
\textbf{Step 2} (Stochastic integral). By \cite[ Theorem 3.2]{Ondrejat_Veraar_2020} we conclude that for 
\[
P\left(\left\{
\left[ 
\sum_{l=1}^\infty\lambda_l \int_0^\cdot \divv (w_\epsilon(s) \psi_l)\, d\beta^{(l)}_s \right]_{\gamma, L^2(\TT)}> K \,\wedge\,  \sup_{0\le t\le T}\|B( w_\epsilon(t))\|_{L_2(H^2(\TT, \RR^2), L^2(\TT))} \le \sqrt{K}
\right\}\right)\] is dominated by $
2e^{-C_{\gamma, T} K}$, where $C_{\gamma, T}\in (0,\infty)$ is a suitable constant.  We observe that due to \eqref{boundedness_psil}, $B$ maps continuously from $H^1(\TT)$  to $L_2(H^2(\TT, \RR^2), L^2(\TT))$   with operator norm only depending on $\Lambda$. Using additionally \eqref{Eq_Energy_est_reg_scheme}, we obtain the estimate 
\begin{align}\begin{split}\label{Eq_aux17}
		&
	P\left(\left\{
	\left[ 
	\sum_{l=1}^\infty\lambda_l \int_0^\cdot \divv (w_\epsilon(s) \psi_l)\, d\beta^{(l)}_s \right]_{\gamma, L^2(\TT)}> K \right\}\right)
	\\\le\, & 
	2e^{-C_{\gamma, T} K} \, +\, \frac{E\left[\sup_{0\le t\le T}\|B(w_\epsilon(t))\|_{H^1(\TT)}^2\right]}{K}
	\, \lesssim_{\Lambda, \gamma,  T}  \frac{1+E\left[\|u_0\|_{H^1(\TT)}^2\right]}{K}.
	\end{split}
\end{align}
\textbf{Step 3} (Combination of the estimates). Let $0\le s<t\le T$.
Splitting the process $u_\epsilon$ in its increments corresponding to $v_\epsilon$ and $w_\epsilon$ we obtain that
\[
u_\epsilon(t)-u_\epsilon(s)\,=\, \int_{s'}^{t'} \divv J_{\epsilon}(\tau)\,d\tau\,+\, 
\int_{s''}^{t''} A^\epsilon(w_{\epsilon}(\tau))\, d\tau\,+\, 
\sum_{l=1}^\infty\lambda_l \int_{s''}^{t''} \divv (w(\tau) \psi_l)\, d\beta^{(l)}_\tau
\]
with an appropriate choice of $s', s'', t', t''\in [0,T]$ satisfying in particular $|t'-s'|, |t''-s''|< 2|t-s|$. Therefore, we can estimate $
\left[u_\epsilon(t)-u_\epsilon(s)\right]_{\gamma ,W^{-1, q'}(\TT)}$ by 
\begin{align*} C_{q,T} \, 
\left(\left[\int_0^\cdot
\divv J_{\epsilon}(s)\,ds\right]_{\frac{1}{2}, W^{-1, q'}(\TT)}\,+\,\left[\int_0^\cdot
A^\epsilon w_{\epsilon}(s)\,ds\right]_{1, H^{-1}(\TT)}\right.&\\+\,\left. \left[ 
\sum_{l=1}^\infty\lambda_l \int_0^\cdot \divv (w_\epsilon(s) \psi_l)\, d\beta^{(l)}_s \right]_{\gamma, L^2(\TT)}\right),&
\end{align*}
where $C_{q,T}<\infty$ is an appropriate constant. Invoking the estimates \eqref{Eq_aux16}, \eqref{Eq_aux17} as well as the assumption $K\in (1,\infty)$ we conclude \eqref{Eq_aux2}.
\end{proof}

\subsection{The vanishing viscosity limit} \label{Sec_deg_lim} 
In this subsection we let  $T,q,\alpha,u_0$  and $R^{(k)}$ as  in Theorem \ref{Thm_degenerate_sol}
and assume, as in the previous subsection, that $u_0$ is an $\mathfrak{F}_0$-measurable random variable on a filtered probability space subject to the usual conditions with a family of independent Brownian motions $(\beta^{(l)})_{l\in \NN}$.
We let $\mathcal{I}$ be a sequence converging to zero and apply for every $k\in \NN$ and $\epsilon\in \mathcal{I}$ the construction from the previous subsection to  the initial value $u^{(k)}_0= 
\mathbbm{1}_{R^{(k)}} u_0$ and obtain a regularized splitting scheme consisting of 
$u^{(k)}_\epsilon, v_\epsilon^{(k)}, w_\epsilon^{(k)}, J_\epsilon^{(k)}$. We consider the sequence 
\begin{equation}\label{Eq_approx_seq}
\left( \left(
\mathbbm{1}_{R^{(l)}}, \beta^{(l)}, u_\epsilon^{(l)}, v_\epsilon^{(l)}, w_\epsilon^{(l)}, J_\epsilon^{(l)}, 
(v_\epsilon^{(l)})^\frac{\alpha+3}{2}, 
(v_\epsilon^{(l)})^\frac{\alpha+3}{4} \right)_{l\in \NN}
\right)_{\epsilon \in \mathcal{I}}
\end{equation}
in the topological product space 
\begin{align}\begin{split}
		\label{Eq_space}
	\prod_{l=1}^{\infty}\;&
	\RR \times C([0,T])\times C(0,T; L^2(\TT)) \times  L_{w*}^\infty(0,T; H^1(\TT))\times  L_{w*}^\infty(0,T; H^1(\TT)) \\&\times  L_{w}^2(0,T; L^{q'}(\TT))\times
	L_w^2(0,T; H^2(\TT))\times L_w^4(0,T; W^{1,4}(\TT)).
	\end{split}
\end{align}
\begin{prop}\label{Prop_tNess1}
	The sequence \eqref{Eq_approx_seq} is tight on \eqref{Eq_space}
\end{prop}
\begin{proof}
	By Tychonoff's theorem it is sufficient to show tightness of every component of \eqref{Eq_approx_seq} separately, so we fix an $l\in \NN$. The distribution of $\mathbbm{1}_{R^{(l)}}$ and $\beta^{(l)}$ is independent of $\epsilon$ and since the corresponding space is  a Radon space, the sequences $\left(\mathbbm{1}_{R^{(l)}}\right)_{\epsilon \in \mathcal{I}}$,  $\left(\beta^{(l)}\right)_{\epsilon \in \mathcal{I}}$ are tight. Using \eqref{Eq_Energy_est_reg_scheme} we deduce that 
	\[
	P\left(\left\{
	\|v_\epsilon^{(l)}\|_{L^\infty(0,T; H^1(\TT))}>K
	\right\}\right)\, \lesssim_{\Lambda,  T} \, \frac{E\left[\|u_0^{(l)}\|^2_{H^1(\TT)}\right]}{K^2}\, \to \,0
	\] 
	as $K\to \infty$ uniformly in $\epsilon$
	such that tightness is a consequence of the Banach-Alaoglu theorem. The components $w_\epsilon^{(l)}$, $ J_\epsilon^{(l)}$, $(v_\epsilon^{(l)})^\frac{\alpha+3}{2}$, $
	(v_\epsilon^{(l)})^\frac{\alpha+3}{4}$ can be treated analogously using \eqref{Eq_Energy_est_reg_scheme} and  Lemmas \ref{Lemma_J_new} and \ref{Lemma_power_bounds_I}. Lastly, we obtain from \eqref{Eq_Energy_est_reg_scheme} and \eqref{Eq_aux2} that
	\[
	P\left(\left\{\max\left\{
	\|u_\epsilon^{(l)}\|_{L^\infty(0,T; H^1(\TT))}, \left[u_\epsilon^{(l)}\right]_{\gamma ,W^{-1, q'}(\TT)} 
	\right\}>K\right\}\right)\, \lesssim_{\Lambda, q, \gamma, T} \frac{1+E\left[
		\|u_0^{(l)}\|_{H^1(\TT)}^2
		\right]}{K}
	\]
	for $K\in (1, \infty)$
	such that tightness of $u_\epsilon^{(l)}$ follows by \cite[Theorem 5]{Simon1987}.
\end{proof}
An application of \cite[Theorem 2]{jakub98} yields that there exists for a subsequence, which we  index again by $\mathcal{I}$, a complete probability space $(\tilde{\Omega}, \tilde{\mathfrak{A}}, \tilde{P})$ and a sequence of $\mathfrak{B}$-measurable random variables
\begin{equation*}
	\left( \left(
	\mathbbm{1}_{\tilde{R}_\epsilon^{(l)}}, \tilde{\beta}_\epsilon^{(l)}, \tilde{u}_\epsilon^{(l)}, \tilde{v}_\epsilon^{(l)}, \tilde{w}_\epsilon^{(l)}, \tilde{J}_\epsilon^{(l)}, 
	\tilde{f}_\epsilon^{(l)}, 
	\tilde{g}_\epsilon^{(l)} \right)_{l\in \NN}
	\right)_{\epsilon \in \mathcal{I}}
\end{equation*} 
with values in \eqref{Eq_space} such that 
\begin{equation}\label{Eq_elem_of_new_seq_I}
\left(
\mathbbm{1}_{\tilde{R}_\epsilon^{(l)}}, \tilde{\beta}_\epsilon^{(l)}, \tilde{u}_\epsilon^{(l)}, \tilde{v}_\epsilon^{(l)}, \tilde{w}_\epsilon^{(l)}, \tilde{J}_\epsilon^{(l)}, 	\tilde{f}_\epsilon^{(l)}, 
\tilde{g}_\epsilon^{(l)}  \right)_{l\in \NN}
\end{equation} 
has the same distribution as 
\[
\left(
\mathbbm{1}_{R^{(l)}}, \beta^{(l)}, u_\epsilon^{(l)}, v_\epsilon^{(l)}, w_\epsilon^{(l)}, J_\epsilon^{(l)}, 
(v_\epsilon^{(l)})^\frac{\alpha+3}{2}, 
(v_\epsilon^{(l)})^\frac{\alpha+3}{4} \right)_{l\in \NN}
\]
for every $\epsilon \in \mathcal{I}$. Moreover, as $\epsilon \searrow 0$, \eqref{Eq_elem_of_new_seq_I} converges to a $\mathfrak{B}$-measurable random variable
\begin{equation*}
	\left(
	\mathbbm{1}_{\tilde{R}^{(l)}}, \tilde{\beta}^{(l)}, \tilde{u}^{(l)}, \tilde{v}^{(l)}, \tilde{w}^{(l)}, \tilde{J}^{(l)}, 	\tilde{f}^{(l)}, 
	\tilde{g}^{(l)}  \right)_{l\in \NN}
\end{equation*} 
in \eqref{Eq_space}.
\begin{rem}
	In order to apply \cite[Theorem 2]{jakub98} one needs to check that there exists a countable sequence of $[-1,1]$-valued continuous functions on \eqref{Eq_space} which separate the points. Such a sequence is straightforward to construct using point-evaluations for the spaces of continuous functions and separability of the respective (pre-) dual for the spaces equipped with weak (weak-*) topology, see \cite[Proposition 1.2.29; Corollary 1.3.22]{AnainBanachspaces1}.
\end{rem}
\begin{lemma}\label{lemma_elem_propI}The sets $(\tilde{R}^{(k)})_{k\in \NN}$ form up to $\tilde{P}$-null sets a disjoint partition of $\tilde{\Omega}$. Moreover, the following holds $\tilde{P}$-almost surely for every $k\in\NN$.
	\begin{enumerate}[label=(\roman*)]
		\item The random variables 
		\[
		\tilde{u}^{(k)}, \tilde{v}^{(k)}, \tilde{w}^{(k)}, \tilde{J}^{(k)}, 	\tilde{f}^{(k)}\;\; \text{and}\;\; 
		\tilde{g}^{(k)}
		\]
		vanish outside of  of ${\tilde{R}^{(k)}}$.
		\item $\tilde{u}^{(k)}(t)\ge 0$ for all $t\in [0,T]$.
		\item For almost all $t\in [0,T]$ we have
		\begin{equation}\label{Eq_splitting_Limiting_scheme}
			\tilde{u}^{(k)}(t )=\begin{cases}
				\tilde{v}^{(k)}(2(t-j\delta)+j\delta ),& j\delta \le t< (j+\frac{1}{2})\delta,\\
				\tilde{w}^{(k)}(2(t-(j+\frac{1}{2})\delta)+j\delta ),& (j+\frac{1}{2})\le t< (j+1)\delta.
			\end{cases}
		\end{equation}
		\item The tuples $(\tilde{v}^{(k)},\tilde{J}^{(k)})$ are solutions to the deterministic thin-film equation on $[j\delta, (j+1)\delta)$ satisfying property (iv) from Theorem \ref{Thm_ex_sol_det} with initial value $\tilde{u}^{(k)} (j\delta)$ for every $j=0,\dots, N$.
		\item
		$\tilde{f}^{(k)}=(\tilde{v}^{(k)})^\frac{\alpha+3}{2}$ and
		$\tilde{g}^{(k)}=(\tilde{v}^{(k)})^\frac{\alpha+3}{4}$.
	\end{enumerate}
\end{lemma}\begin{proof}For every $\epsilon\in \mathcal{I}$ we have 
\[\tilde{E}\left[
\mathbbm{1}_{\tilde{R}^{(k_1)}_\epsilon}
\mathbbm{1}_{\tilde{R}^{(k_2)}_\epsilon}\right] \,=\, \delta_{k_1, k_2} P \left(
{R}^{(k_1)}\right),
\] 
such that by letting $\epsilon \searrow 0$ we conclude the first part of the claim. Part (i) follows by letting $\epsilon \searrow 0$ in $\mathbbm{1}_{\tilde{R}^{(k)}_\epsilon} \|\tilde{u}^{(k)}_\epsilon\|_{C(0,T; L^2(\TT))}\,=\, 0$ and the same argument for the other random variables. Part (ii) is a consequence of $\tilde{u}_\epsilon^{(k)}(t)\ge 0$ together with conservation of this property under limits in $C(0,T; L^2(\TT))$. Analogously, we deduce (iii) from the respective property in of $\tilde{u}^{(k)}_\epsilon$, $\tilde{v}^{(k)}_\epsilon$ and $\tilde{w}^{(k)}_\epsilon$. For (iv) and (v) we observe first that by measurability of $
\mathcal{S}_{\alpha, q,\delta}$ we have 
\begin{equation}\label{Eq_aux3}
\mathcal{S}_{\alpha, q,\delta} \tilde{u}^{(k)}_\epsilon (j\delta)\, =\, 
\left(\tilde{v}^{(l)}_\epsilon|_{[j\delta, (j+1)\delta)},\tilde{J}^{(k)}_\epsilon|_{[j\delta, (j+1)\delta)}, \tilde{f}^{(k)}_\epsilon|_{[j\delta, (j+1)\delta)}, \tilde{g}^{(k)}_\epsilon|_{[j\delta, (j+1)\delta)}\right).
\end{equation}
In particular, we have $\tilde{f}^{(k)}_\epsilon = (\tilde{f}^{(k)}_\epsilon)^\frac{\alpha+3}{2}$, $\tilde{g}^{(k)}_\epsilon = (\tilde{f}^{(k)}_\epsilon)^\frac{\alpha+3}{4}$ and the right hand-side of  \eqref{Eq_aux3} fulfills the properties stated in Theorem \ref{Thm_ex_sol_det}. If we let $\epsilon \searrow 0$ we deduce that the  limit 
 \[
 \left(\tilde{v}^{(k)}|_{[j\delta, (j+1)\delta)},\tilde{J}^{(k)}|_{[j\delta, (j+1)\delta)}\right)\]
is  a solution to the thin-film equation  and  that (v) holds true by Proposition \ref{Prop_conv_sol_det}. In light of \ref{Rem_version_det_sol} the initial value of \eqref{Eq_aux3} is indeed $\tilde{u}^{(k)}_\epsilon(j\delta)$. It is left to observe that property (iv) of Theorem \ref{Thm_ex_sol_det} is preserved due to lower semi-continuity of the norm with respect to weak convergence.
\end{proof}

By \eqref{Eq_splitting_Limiting_scheme} we deduce that $\tilde{u}^{(k)}_\epsilon$ converges to $\tilde{u}^{(k)}$ also in $   L_{w*}^\infty(0,T; H^1(\TT)) $ and that $\tilde{u}^{(k)}$ is weakly continuous in $H^1(\TT)$ again. Moreover, we identify in the following $\tilde{v}^{(k)}$ and $\tilde{w}^{(k)}$ with their versions such that \eqref{Eq_splitting_Limiting_scheme} holds for all $t\in [0,T]$. We define $\tilde{\mathfrak{F}}$ as the augmentation of the filtration $\tilde{\mathfrak{G}}$ given by 
\[
\tilde{\mathfrak{G}}_t\,=\, \sigma \left(\left\{\left.
\mathbbm{1}_{\tilde{R}^{(l)}}, \tilde{J}^{(l)}|_{[0,t]} 
\right| \,l\in \NN
\right\}\cup 
\left\{\left. \tilde{u}^{(l)}(s), \tilde{\beta}^{(l)}(s)\right|
0\le s\le t, \,l\in \NN
\right\}\right),
\]
where we consider $\tilde{J}^{(l)}|_{[0,t]}$ as a $\mathfrak{B}$-random variable in $L^2(0,t; L^{q'}(\TT))$. 
\begin{lemma}\label{IS_SOL_1}
	The processes $(\tilde{\beta}^{(l)})_{l\in \NN}$ are a family of independent $\tilde{\mathfrak{F}}$-Brownian motions. Moreover, we have for every $k\in \NN$, $j\in \{0,\dots, N\}$ and $\varphi\in H^1(\TT)$ that $\tilde{P}$-almost surely
	\begin{align*}\left< \tilde{w}^{(k)}(t),
		\varphi\right>\,-\,\left< \tilde{w}^{(k)}(j\delta),
		\varphi\right>\,=\,&  \frac{1}{2}\sum_{l=1}^{\infty} \lambda_l^2\int_{j\delta }^t \left<  \divv (\divv(\tilde{w}^{(k)}(s)\psi_l) \psi_l ), \varphi\right>\, ds
		\\&+\, \sum_{l=1}^\infty \lambda_l \int_{j
			\delta}^t\left< \divv(\tilde{w}^{(k)}(s)\psi_l), \varphi \right>\, d\beta^{(l)}_s
	\end{align*}
	for all $t\in [j\delta, (j+1)\delta)$.
\end{lemma}
The proof of the lemma above is a simpler version of the proof of Theorem \ref{Thm_final_SPDE} and is therefore omitted. Finally, we observe that many of the estimates from the previous subsection carry over to their limit.

\begin{prop}\label{Prop_Energy_Est_in_limit}For every $k\in \NN$ and  $p\in (0, \infty)$ we have 
	\begin{align}\begin{split}
			\label{Eq_aux4}
		\tilde{E}\left[
		\sup_{0\le t\le T} \|\tilde{u}^{(k)}\|_{H^1(\TT)}^p
		\right]\, &\lesssim_{\Lambda, p, T} \, E\left[\|u_0^{(k)}\|^p_{H^1(\TT)}\right],\\
	\tilde{E}\left[ \|\tilde{J}^{(k)}\|_{L^2(0,T; L^{q'}(\TT))}^\frac{p}{2} \right]\, &\lesssim_{\Lambda, p,q , T}\, E\left[\|u_0^{(k)}\|^p_{H^1(\TT)}\right].
		\end{split}
\end{align}
Moreover, for any $\gamma \in (0, \frac{1}{2})$ and $K\in (1, \infty)$ it holds
\[
\tilde{P}\left(\left\{\left[\tilde{u}^{(k)}\right]_{\gamma ,W^{-1, q'}(\TT)} >K \right\}\right)
\,\lesssim_{\Lambda, q, \gamma, T}\, \frac{1+E\left[\|{u}^{(k)}_0\|_{H^1(\TT)}^2\right]}{K}.
\]
\end{prop}
\begin{proof}
	The estimates \eqref{Eq_aux4} follow from lower semi-continuity of the norm with respect to weak (weak-*) convergence, \eqref{Eq_Energy_est_reg_scheme}, Lemma \ref{Lemma_J_new} and Fatou's lemma. Moreover, we observe that
	\[
	[\tilde{u}^{(l)}]_{\gamma, W^{-1,q'}(\TT)}\,\le \,\liminf_{\epsilon \searrow 0}
	[\tilde{u}_{\epsilon}^{(l)}]_{\gamma, W^{-1,q'}(\TT)}
	\]
	by convergence in $C(0,T; L^2(\TT))$. Therefore, we have
	\begin{align*}&
	P\left(
	\left\{
	[\tilde{u}^{(l)}]_{\gamma, W^{-1,q'}(\TT)} >K
	\right\}
	\right)\, \le \, 
		P\left(
	\left\{\liminf_{\epsilon \searrow 0}
	[\tilde{u}_{\epsilon}^{(l)}]_{\gamma, W^{-1,q'}(\TT)} >K
	\right\}
	\right)\\\le \,&	P\left(\liminf_{\epsilon \searrow 0}
	\left\{
	[\tilde{u}_{\epsilon}^{(l)}]_{\gamma, W^{-1,q'}(\TT)} >K
	\right\}
	\right)\, \le \,\liminf_{\epsilon \searrow 0}
		P\left(
	\left\{
	[\tilde{u}_{\epsilon}^{(l)}]_{\gamma, W^{-1,q'}(\TT)} >K
	\right\}\right)
	\end{align*}
and it is left to apply  \eqref{Eq_aux2}.
\end{proof}
\begin{proof}[Proof of Theorem \ref{Thm_degenerate_sol}]
	The limiting random variables 
	$(\beta^{(l)})_{l\in \NN}$, $\mathbbm{1}_{\tilde{R}^{(k)}}$, $\tilde{u}^{(k)}$ and $\tilde{J}^{(k)}$ have the desired properties. Indeed, \ref{Item_IV} is a consequence of
	\[
	(\mathbbm{1}_{\tilde{R}_\epsilon^{(k)}}, \tilde{u}^{(k)}_\epsilon(0))_{k\in \NN}\, \sim\, 
	(\mathbbm{1}_{R^{(k)}}, u^{(k)}_0)_{k\in \NN},\;\; \epsilon\in \mathcal{I},
	\]
	the convergence 
	\[
	(\mathbbm{1}_{\tilde{R}_\epsilon^{(k)}}, \tilde{u}^{(k)}_\epsilon(0))_{k\in \NN} \, \to \, 
	(\mathbbm{1}_{\tilde{R}^{(k)}}, \tilde{u}^{(k)}(0))_{k\in \NN}
	\]
	in $(\RR\times L^2(\TT))^\infty$ and Lemma \ref{lemma_elem_propI} (i). Part \ref{Item_meas} follows by the choice of $\tilde{\mathfrak{F}}$. Parts \ref{Item_sol_Det_deg_case}, \ref{Item_SPDE_deg}, \ref{Item_est_deg} and \ref{Item_est_Hoelder} are the content of  Lemma \ref{lemma_elem_propI} (iv), Lemma \ref{IS_SOL_1} and Proposition \ref{Prop_Energy_Est_in_limit}.
\end{proof}
\section{Construction of solutions}\label{Sec_time_limit}
Let finally $\mu,T, q,\alpha$ as in  Theorem \ref{Thm_main}. We  apply Theorem \ref{Thm_degenerate_sol} for every $N\in \NN$ to a random variable which is distributed according to $\mu$ and obtain processes $u^{(k)}_N$, families of Brownian motions $\beta_N^{(l)}$ and random variables $\mathbbm{1}_{R_N^{(k)}}$, $J^{(k)}_N$ for $l,k \in \NN$ satisfying the stated properties. We assume that these random variables are defined on the same probability space $(\Omega, \mathfrak{A}, P)$ with filtration $\mathfrak{F}$ and moreover,  that $\beta^{(l)}=\beta^{(l)}_N$ is independent of $N$. This does not influence the mathematical analysis since we analyze the solutions for each $N$ separately and serves only notational convenience. 
\begin{rem}Alternatively, one can also apply the limiting procedure from subsection \ref{Sec_deg_lim} for all step numbers $N\in N$ simultaneously to end up in the assumed situation. This, however, would lead to a notational mess in the previous section.
\end{rem} 
We remark also that we  have dropped the $\sim$-notation since we want to pass to another probability space once more.
For $k,N\in \NN$ and we define ${v}_N^{(k)}$ and ${w}_N^{(k)}$ by
	\begin{equation}\label{Eq52}
\begin{cases}
	{u}_N^{(k)}(t)\,=\, {v}_N^{(k)}(2(t-j\delta)+j\delta ),& j\delta \le t< (j+\frac{1}{2})\delta,\\
	{u}_N^{(k)}(t)\,=\,{w}_N^{(k)}(2(t-(j+\frac{1}{2})\delta)+j\delta ),& (j+\frac{1}{2})\le t< (j+1)\delta,
\end{cases}
\end{equation} where again $\delta=\delta(N)=\frac{T}{N+1}$.
\subsection{Additional tightness properties}\label{Sec_add_Tness}
The approximate solutions $u^{(k)}_N$, $J^{(k)}_N$ satisfy the bounds from  Theorem \ref{Thm_degenerate_sol} \ref{Item_est_deg}, \ref{Item_est_Hoelder}, which can as in Proposition \ref{Prop_tNess1} be used to derive tightness in suitable spaces.
\begin{rem}
In light of Theorem \ref{Thm_degenerate_sol} \ref{Item_IV},  the right-hand sides of the aforementioned bounds can be expressed in terms of the cut-off moments
\begin{equation}\label{Eq43}\nu_{k,p} \,=\,
\int \mathbbm{1}_{\left\{
	k-1\le \|\cdot \|_{H^1(\TT)}<k
	\right\}} \|\cdot \|_{H^1(\TT)}^p\, d\mu
\end{equation}
of the initial distribution $\mu$. We remark that the notation \eqref{Eq43} will be used during the remainder of this section.
\end{rem}
 In this subsection we provide an additional tightness property, which can be seen as the adaption of the compactness statement \cite[Lemma 2.5]{Passo98ona} to our setting. Its proof relies on deriving a version of \eqref{Lemma_power_bounds_I} with a uniform estimate in $N$ and a simplified proof of  \cite[Lemma 2.5]{Passo98ona}. The former is based on a control of the entropy production along the stochastic dynamics. We point out that the simplification of the compactness proof is only possible due to the assumption $\alpha \in (0,1)$, which is less general than the situation in \cite[Lemma 2.5]{Passo98ona}. 
\begin{lemma}\label{Bound_powers_2}It holds for every $k, N\in \NN$ that
	\[E\left[
	\left\|\left(v_{N}^{(k)}\right)^{\frac{\alpha+3}{2}}\right\|^2_{L^2(0,T;H^2(\TT))}+\,
	\left\|\left(v_{N}^{(k)}\right)^{\frac{\alpha+3}{4}}\right\|^4_{L^4(0,T;W^{1,4}(\TT))}
	\right]\, \lesssim_{ \Lambda, \alpha, T}\,1\,+\, \nu_{k, \alpha+3}.
	\]
\end{lemma}
\begin{proof}
	An application of Theorem \ref{Thm_ex_sol_det} \ref{Item_det_entr_est} yields the estimate
	\begin{align}\begin{split}\label{Eq7}
			&
	\int_0^T \int_{\TT} \left|H \left(v^{(k)}_N\right)^{\frac{\alpha+3}{2}}\right|^2+\left|\nabla \left(v^{(k)}_N\right)^{\frac{\alpha+3}{4}}\right|^4 \, dx\, dt\, \lesssim_{\alpha} \, \sum_{j=0}^N \int_{\TT} G_\alpha(v_N^{(k)}(j\delta))- G_\alpha(w^{(k)}_N(j\delta))\, dx
	\\
=\,&
\int_{\TT} G_\alpha(v^{(k)}_N(0))-G_\alpha(w_N^{(k)}(N\delta))\, dx\, +\, \sum_{j=0}^{N-1} \int_{\TT} G_\alpha(v^{(k)}_N((j+1)\delta))-G_\alpha(w_N^{(k)}(j\delta))\, dx
		\end{split}
\end{align}
The first summand can be estimated directly using  the expression \eqref{Eq_alpha_entropy_expl} by
\begin{align*}&
C_\alpha\left(\|v^{(k)}_N(0)\|_{L^2(\TT)}^{\alpha+1}\,+\,\|v^{(k)}_N(0)\|_{L^2(\TT)}\,+\,\|w^{(k)}_N(0)\|_{L^2(\TT)}^{\alpha+1}\,+\,\|w^{(k)}_N(0)\|_{L^2(\TT)}\right)\\ \lesssim_\alpha\,&
 \sup_{0\le t\le T}\|u^{(k)}_N(t)\|_{L^2(\TT)}\,+\,1.
\end{align*}
Taking the expectation and employing Theorem \ref{Thm_degenerate_sol} \ref{Item_est_deg} we obtain the estimate
\begin{equation}\label{Eq14}
E\left[
\int_{\TT} G_\alpha(v^{(k)}_N(0))-G_\alpha(w_N^{(k)}(N\delta))\, dx
\right]\, \lesssim_{\Lambda, \alpha, T}\, 1\,+\,\nu_{k,1}.
\end{equation}
To estimate the second summand of the right-hand side of \eqref{Eq7}, we fix a function  $\eta\in C^\infty(\mathbb{R})$ such that $0\le\eta\le 1$, $\eta=1$ on $ [2,\infty)$ and $\eta= 0$ on $(-\infty,1]$. We define smooth functions
\[ \eta_\kappa(x)\,=\,\eta\left(\frac{x}{\kappa}\right),\;\;
G_{\alpha,\kappa} (x)\,=\, G_\alpha(x)\eta_\kappa(x).
\]	for $\kappa>0$. Correspondingly, we define the regularized functional \begin{equation}\label{Eq44}
\phi_\kappa\colon \, L^2(\mathbb{\TT})\to \mathbb{R}, \, w\mapsto \int_{\TT} G_{\alpha,\kappa}(w)\, dx\end{equation}
We observe that there is a constant $C_{\alpha,\kappa}<\infty$ such that
\begin{equation}\label{Eq9}|G_{\alpha,\kappa}(x)|\le C_{\alpha,\kappa}|x|^2, \;\;|G_{\alpha,\kappa}'(x)|\le C_{\alpha,\kappa}|x|\;\; \text{ and } \;\;|G_{\alpha,\kappa}''(x)|\le C_{\alpha,\kappa}
\end{equation}
for each $x\in \RR$. 
An application of It\^o's formula to the composition of the functional $\phi_\kappa$ with the process  $w_N^{(k)}$ satisfying the SPDE from Theorem \ref{Thm_degenerate_sol} \ref{Item_SPDE_deg} yields that
	\begin{align}\begin{split}\label{Eq10}
	\phi_\kappa(w^{(k)}_N(t))\,=\,&\phi_\kappa({w}^{(k)}_N(j\delta))\,+\, \sum_{l=1}^\infty \lambda_l \int_{j\delta}^t  \int_{\TT} G_{\alpha,\kappa}'({w}^{(k)}_N(s)) \divv(w^{(k)}_N(s) \psi_l)\, dx\, d{\beta}^{(l)}(s)\\&{-\, \frac{1}{2}	\sum_{l=1}^\infty \int_{j\delta}^t \int_{\TT} 
	\lambda_l^2 G_{\alpha,\kappa}''({w}^{(k)}_N(s)) [\divv ({w}^{(k)}_N(s)\psi_l)] [ \psi_l\cdot \nabla {w}^{(k)}_N(s)]\, dx\, ds}\\&+\,\frac{1}{2}\sum_{l=1}^\infty 
	\lambda_l^2
	\int_{j\delta}^t  \int_{\TT}G_{\alpha,\kappa}''(w^{(k)}_N(s))\left[\divv ({w}^{(k)}_N(s) \psi_l) \right]^ 2\,dx\, ds
	\end{split}
\end{align}
 for $t\in [j\delta, (j+1)\delta)$. For a justification of the applicability of It\^o's formula see appendix \ref{AppendixB}. We note that as pointed out there, the above formula is also valid for the end-point $t=(j+1)\delta$, but then the term on the left-hand side has to be replaced by $\phi_\kappa(v^{(k)}_N((j+1)\delta))$. 
 Due to the smoothness of $\psi_l$ it holds
 \begin{align*}
 	[
 	\divv (w\psi_l)] [\psi_l\cdot \nabla w]
 	\,&=\,[
 	\divv (w\psi_l)]^2- [
 	\divv (w\psi_l)][ w\divv\psi_l ]\\&=\, [
 	\divv (w\psi_l)]^2-\left[
 	w\divv \psi_l
 	\right]^2-w\nabla w \cdot \left[
 	\psi_l\divv \psi_l
 	\right].
 \end{align*} 
The derivative of $ 
\zeta_\kappa(x)\,=\, \int_0^x y G_{\alpha,\kappa}''(y)\,dy$ is bounded such that an application of \cite[Proposition 9.5]{brezis2010functional} yields
	\begin{align*}
	\int_{\TT}G_{\alpha,\kappa}''(w)
	w\nabla w \cdot \left[
	\psi_l\divv \psi_l
	\right]\, dx\,=\, -
	\int_{\TT}\zeta_\kappa(w)
	\divv \left[
	\psi_l\divv\psi_l
	\right]\, dx,
\end{align*}
for $w\in H^1(\TT)$. We also introduce the function  $\theta_\kappa(x)= x^2 G_{\alpha,\kappa}''(x)$ and rewrite \eqref{Eq10} using the previous identities as
	\begin{align}\begin{split}\label{aux_eq4}
	\phi_\kappa(w^{(k)}_N(t))\,=\,\phi_\kappa({w}^{(k)}_N(j\delta))\,&+\, \sum_{l=1}^\infty \lambda_l \int_{j\delta}^t  \int_{\TT} G_{\alpha,\kappa}'({w}^{(k)}_N(s)) \divv(w^{(k)}_N(s) \psi_l)\, dx\, d{\beta}^{(l)}(s)\\&+\frac{1}{2}\sum_{l=1}^\infty \lambda_l^2
		\int_{j\delta}^t \int_{\TT} \theta_\kappa({w}^{(k)}_N(s))(
		\divv\psi_l)^2\, dx\, ds\\&-\,\frac{1}{2}\sum_{l=1}^\infty 
		\lambda_l^2
		\int_{j\delta}^t \int_{\TT}\zeta_\kappa({w}^{(k)}_N(s))
		\divv \left[
		\psi_l\divv\psi_l
		\right]\, dx\, ds.\end{split}
\end{align}
Using that
	\begin{equation}\label{eq15}
	G_{\alpha,\kappa}''(x)= \frac{1}{\kappa^2}\eta''\left(\frac{x}{\kappa}\right) G_\alpha(x)+
	\frac{2}{\kappa}\eta'\left(\frac{x}{\kappa}\right) \left[\frac{x^\alpha}{\alpha}+ r_\alpha'(x)\right]+\eta \left(
	\frac{x}{\kappa}
	\right)x^{\alpha-1}
\end{equation}
and that $\eta'$ and $\eta''$ vanish outside of $[1,2]$ we deduce that
$|\theta_\kappa(x)|\le C_\alpha (1+|x|)$. The same argument yields that $|\zeta_\kappa(x)|\le C_\alpha (1+|x|)$, indeed we can estimate for example
\begin{align*}
	\left|\int_0^x  \frac{y}{\kappa^2}\eta''\left(\frac{y}{\kappa}\right) G_\alpha(y) \, dy\right|\, \lesssim \, 
	\frac{1}{\kappa} \int_0^{|x|\wedge 2\kappa}  \left| G_\alpha(y)\right| \, dy\, \lesssim_\alpha\,
	\frac{|x|\wedge 2\kappa+(|x|\wedge 2\kappa)^2}{\kappa} \,\lesssim\, 1+|x|.
\end{align*}
Using \eqref{boundedness_psil} we obtain the estimates
\begin{align*}\left|
\frac{1}{2}\sum_{l=1}^\infty \lambda_l^2
\int_{j\delta}^{(j+1)\delta} \int_{\TT} \theta_\kappa({w}^{(k)}_N(s))(
\divv\psi_l)^2\, dx\, ds\right|\,&\lesssim_{\Lambda ,\alpha} \, \delta \left(1+\esssup_{0\le t\le T}\|w^{(k)}_N(t)\|_{L^2(\TT)}\right),\\\left|\frac{1}{2}\sum_{l=1}^\infty 
\lambda_l^2
\int_{j\delta}^{(j+1)\delta} \int_{\TT}\zeta_\kappa({w}^{(k)}_N(s))
\divv \left[
\psi_l\divv\psi_l
\right]\, dx\, ds\right|\,&\lesssim_{\Lambda, \alpha} \, \delta \left(1+\esssup_{0\le t\le T}\|w^{(k)}_N(t)\|_{L^2(\TT)}\right).
\end{align*}
For the series of stochastic integrals in \eqref{aux_eq4} we observe that 
\begin{align*}&E\left[
	\sum_{l=1}^\infty \lambda_l^2
	\int_{j\delta}^{(j+1)\delta}\left(\int_{\TT} G_{\alpha,\kappa}'({w}^{(k)}_N(s)) \divv(w^{(k)}_N(s) \psi_l)\, dx\right)^2\, ds\right]\\ \lesssim_{\alpha, \kappa} \, & E\left[\sum_{l=1}^\infty \lambda_l^2
	\int_{j\delta}^{(j+1)\delta}\left(\|{w}^{(k)}_N(s)\|_{H^1(\TT)}\right)^2\, ds\right]\, <\infty
\end{align*}
We  used \eqref{boundedness_psil} and that the function $G_{\alpha,\kappa}'$ is bounded in the first inequality and Theorem \ref{Thm_degenerate_sol} \ref{Item_est_deg} for the second one. Hence, the series of stochastic integrals has integrable quadratic variation and is therefore a martingale. Therefore, taking the expectation of \eqref{aux_eq4} with $t=(j+1)\delta$, using the previous estimates, as well as Theorem \ref{Thm_degenerate_sol} \ref{Item_est_deg} once more, yields that
\begin{align*}E\left[\phi_\kappa(v^{(k)}_N((j+1)\delta))\,-\,\phi_\kappa({w}^{(k)}_N(j\delta))\right]\,
	&\lesssim_{\Lambda, \alpha} \, \delta  \left(1+ E\left[\esssup_{0\le t\le T}\|w^{(k)}_N(t)\|_{L^2(\TT)}\right]\right)\\& \lesssim_{\Lambda,  T} \,\delta \left(1\,+\,\nu_{k,1}\right)
\end{align*}
Finally, taking the expectation of \eqref{Eq7} and using additionally \eqref{Eq14}, we end up with 
\begin{align*}E\left[
		\int_0^T \int_{\TT} \left|H \left({v}_N^{(k)}\right)^{\frac{\alpha+3}{2}}\right|^2+\left|\nabla \left({v}_N^{(k)}\right)^{\frac{\alpha+3}{4}}\right|^4 \, dx\, dt
	\right]\lesssim_{ \Lambda, \alpha, T}  1\,+\,\nu_{k,1}.
\end{align*}
As in the proof of Lemma \ref{Lemma_power_bounds_I}, we use that
	\begin{align*}
	E\left[\int_0^T 
	\int_{\TT} \left|\left(v_{N}^{(k)}(t)\right)^{\frac{\alpha+3}{2}}\right|^2 +  \left|\left(v_{N}^{(k)}(t)\right)^{\frac{\alpha+3}{4}}\right|^4\, dx
	\, dt\right]\,& \lesssim_{\alpha,T}\, E \left[ \sup_{0\le  t\le T} \|{u}_N^{(k)}(t)\|_{H^1(\TT)}^{\alpha+3}\right]\\&
	\lesssim_{\Lambda, \alpha, T} \, \nu_{k, \alpha+3}
\end{align*}
as a consequence of Theorem \ref{Thm_degenerate_sol} \ref{Item_est_deg} and therefore 
\[E\left[
\left\|\left(v_{N}^{(k)}\right)^{\frac{\alpha+3}{2}}\right\|^2_{L^2(0,T;H^2(\TT))}+\,
\left\|\left(v_{N}^{(k)}\right)^{\frac{\alpha+3}{4}}\right\|^4_{L^4(0,T;W^{1,4}(\TT))}
\right]\, \lesssim_{ \Lambda, \alpha, T}\,1\,+\, \nu_{k, \alpha+3}
\]
by \eqref{Eq_Interp1}.
\end{proof}

\begin{lemma}\label{Lemma_add_Tness}For every $k\in\NN$, the laws of $\left(\left(v_{N}^{(k)}\right)^{\frac{\alpha+3}{2}}\right)_{N\in \NN}$ are tight on $L^2(0,T; H^1(\TT))$.
\end{lemma}
\begin{proof}We divide the proof into three steps.

	\textbf{Step 1} (H\"older regularity of $u_N^{(k)}$ in $L^2(\TT)$).  First, we observe that the paths of ${u}^{(k)}_N$ are weakly continuous in $H^1(\TT)$ and in particular
	$
	\|u^{(k)}_N(t)\|_{H^1(\TT)} \le  \|u^{(k)}_N\|_{L^\infty(0,T ; H^1(\TT))}
	$
	for every $t\in [0,T]$. The Sobolev embedding theorem, see \cite[Section 3.5.5]{schmeisser1987topics}, states that $W^{-1, q'}(\TT)\hookrightarrow H^{\frac{-2}{q'}}(\TT)$ and therefore we can estimate the H\"older seminorm
	$[u^{(k)}_N]_{\gamma,H^{\frac{-2}{q'}}(\TT) }$  by $C_q[u^{(k)}_N]_{\gamma,W^{-1, q'(\TT)}(\TT)} $ for $\gamma \in (0,1)$.  The interpolation inequality
	\[
	\|f\|_{L^2(\TT)}\, \le \|f\|_{H^1(\TT)}^{\theta}
	\|f\|_{ H^{\frac{-2}{q'}}(\TT)}^{1-\theta}, \;\; f\in H^1(\TT),
	\]
	with $\theta= \frac{\frac{2}{q'} }{1+\frac{2}{q'}}$ can be  derived using Fourier methods. We obtain the estimate
	\begin{align*}
		\|u^{(k)}_N(t)-u^{(k)}_N(s)\|_{L^2(\TT)}\, &\le\, \|u^{(k)}_N(t)-u^{(k)}_N(s)\|^{\theta}_{H^{1}(\TT)} 	\|u^{(k)}_N(t)-u^{(k)}_N(s)\|_{ H^{\frac{-2}{q'}}(\TT)}^{1-\theta}
		\\&\lesssim_q \, [u^{(k)}_N]_{\gamma,W^{-1, q'(\TT)}(\TT)} ^{1-\theta}\|u^{(k)}_N\|_{L^\infty(0,T;H^1(\TT))}^\theta|t-s|^{(1-\theta)\gamma}
	\end{align*}
on the increments, which yields 
	 that 
	\begin{equation}\label{Eq11}
	[u^{(k)}_N]_{(1-\theta)\gamma , L^2(\TT)}\,\lesssim_q \, [u^{(k)}_N]_{\gamma,W^{-1, q'(\TT)}(\TT)} ^{1-\theta}\|u^{(k)}_N\|_{L^\infty(0,T;H^1(\TT))}^\theta.
	\end{equation}
	
	\textbf{Step 2} (integral  estimate on the increments of $v_N^{(k)}$). In this part, we deduce from the first step an estimate on 
		\[
	\|\tau_hv_N^{(k)}- v_N^{(k)}\|_{L^4(0, T-h; L^2(\TT))},
	\]
	in similar terms, where $\tau_h$ denotes the translation operator by $h>0$ in the time variable. To quantify the jumps in the paths of of $v_N^{(k)}$ we introduce the function
	\[
	\phi_{N,h}\colon [0, T] \to \mathbb{N}, \, N\mapsto   \lfloor t+h \rfloor_\delta -  \lfloor t \rfloor_\delta,
	\]
	which counts how many discretization points lie between $t$ and $t+h$. The function $\lfloor \cdot \rfloor_\delta$ denotes here the  biggest integer multiple of $\delta$ which is less or equal to its input value. Then we have
	\begin{equation}\label{Eq8}\|
	v^{(k)}_N (t+h)- 
	v^{(k)}_N (t)\|_{L^2(\TT)}\, \le \, [u^{(k)}_N]_{(1-\theta)\gamma , L^2(\TT)}\left(h + \phi_{N,h} (t)\delta\right)^{(1-\theta)\gamma}
	\end{equation}
	for $t\in [0, T-h]$.
	We introduce the sets
	\[
	C_{N,h} \,=\,
	\left\{
	t\in [0,T-h]| \phi_{N,h}(t)\ne 1
	\right\} \; \text{ and }\;
	D_{N,h} \,=\,
	\left\{
	t\in [0,T-h]| \phi_{N,h}(t)=1
	\right\},
	\]
	distinguishing between points $t$, where one can estimate the right-hand side of \eqref{Eq8} in terms of $h$ or not. 
	Indeed, if $t\in C_{N,h}$ it holds either $\phi_{N,h}(t)=0$ or $\phi_{N,h}(t)\ge 2$ such that in any case  ${\phi_{N,h}(t)}\delta \le 2h$  and therefore 
	\[\left(h + \phi_{N,h} (t)\delta\right)^{(1-\theta)\gamma}\, \le \left(3h\right)^{(1-\theta)\gamma}.
	\]
	We deduce that
	\begin{align}
		\begin{split}
			\label{eq51}	
			&\|\tau_hv^{(k)}_N - v^{(k)}_N \|_{L^4(0, T-h; L^2(\TT))}^4
			\, = \,\int_0^{T-h} \|v^{(k)}_N (t+h)- v^{(k)}_N (t)\|_{L^2(\TT)}^4\, dt \\\le \,&
			[u^{(k)}_N]_{(1-\theta)\gamma , L^2(\TT)}^4\left(  (3h)^{4(1-\theta)\gamma} |C_{N,h}| +  \left(h+\delta\right)^{4(1-\theta)\gamma}|D_{N,h}|\right).
		\end{split}
	\end{align}
	If $h\ge \delta$, we use the trivial estimate $|C_{N,h}|+|D_{N,h}|\le T$ to conclude
	\begin{equation}\label{Eq51}(3h)^{4(1-\theta)\gamma} |C_{N,h}|\, +\,  \left(h+\delta\right)^{4(1-\theta)\gamma}|D_{N,h}|\, \le \,(3h)^{4(1-\theta)\gamma}T.
	\end{equation}
	 For $h< \delta$ we use instead that 
	\[
	t\in [j\delta, (j+1)\delta-h) \;\; \Rightarrow\;\;  \phi_{N,h}(t)\,=\,0 \;\; \Rightarrow\;\;  t\in C_{N,h}
	\] 
	and consequently $ |D_{N,h}|\le (N+1)h$. We define the function
	 \[
	 f_h(x)\,=\, \left(h+\frac{T}{x+1}\right)^{4(1-\theta)\gamma} (x+1)h , \;\; x\in \left[1, \frac{T}{h}-1\right]
	 \]
	  such that 
	 \[
	 \left(h+\delta\right)^{4(1-\theta)\gamma}|D_{N,h}|\,\le \, 
	 f_h(N)
	 \] and
	 it suffices  to estimate the maximum of $f_h$. Its derivative is given by
	 \[
	 f_h'(x)\,=\, h\left(h+\frac{T}{x+1}\right)^{4(1-\theta)\gamma} \,-\, \frac{4(1-\theta)\gamma T\left(h+\frac{T}{x+1}\right)^{4(1-\theta)\gamma-1} h(x+1)}{(x+1)^2},
	 \]
	 which can vanish only if
	 \[
		\frac{4(1-\theta)\gamma T }{(x+1)\left(h+\frac{T}{x+1}\right)}	\,=\,1	\; \;\Rightarrow\;\;h(x+1)\,=\,
		(4(1-\theta)\gamma-1)T.
	 \]
	 We choose for the rest of the proof that $\gamma =\frac{1}{4}$ such that the above is not feasible for $x\in \left[1, \frac{T}{h}-1\right] $. Hence $f_h$ can attain its maximum only at the boundary points $1$ and $\frac{T}{h}-1$. Evaluating $f_h$ gives
	 \[
	 f_h(1)\,=\,\left(h+\frac{T}{2}\right)^{1-\theta}2h\, \le \,2T^{1-\theta}h,\;\; f_h\left(\frac{T}{h}-1\right)\,=\, (2h)^{1-\theta}T.
	 \] 
	 We end up with the estimate
	 \[
	 (3h)^{1-\theta} |C_{N,h}|\, +\,  \left(h+\delta\right)^{1-\theta}|D_{N,h}|\,\le \,
	 2T(3h)^{1-\theta}\,+\,2T^{1-\theta}h.
	 \]
	 We define the right-hand side as $g_{\theta,T}(h)$ and obtain from \eqref{Eq11} and \eqref{eq51} that 
	 \begin{equation}\label{Eq12}
	 \|\tau_hv^{(k)}_N - v^{(k)}_N \|_{L^4(0, T-h; L^2(\TT))}\, \lesssim_q \, 
	 [u^{(k)}_N]_{\frac{1}{4},W^{-1, q'(\TT)}(\TT)} ^{1-\theta}\|u^{(k)}_N\|_{L^\infty(0,T;H^1(\TT))}^\theta (g_{\theta,T}(h))^\frac{1}{4}.
	 \end{equation}
	 By \eqref{Eq51}, this holds also if $h\ge \delta$.
	 
	 \textbf{Step 3} (proof of tightness). 
	 Due to Theorem \ref{Thm_degenerate_sol} \ref{Item_est_deg}, \ref{Item_est_Hoelder} and Lemma \ref{Bound_powers_2} we have
	 \begin{align*}	P\left(\left\{
	 	\sup_{0\le t\le T} \|u_N^{(k)}\|_{H^1(\TT)} > K
	 	\right\}\right)\, &\lesssim_{\Lambda,  T} \, \frac{\nu_{k,2}}{K^2},
	 	\\ \tilde{P}\left(\left\{\left[u_N^{(k)}\right]_{\frac{1}{4} ,W^{-1, q'}(\TT)} >K \right\}\right)
	 	\,&\lesssim_{\Lambda, q, T}\, \frac{1\,+\,\nu_{k,2}}{K},
	 	\\P\left\{
	 	\left\|\left(v_{N}^{(k)}\right)^{\frac{\alpha+3}{2}}\right\|_{L^2(0,T;H^2(\TT))}>K
	 	\right\}\,  &\lesssim_{ \Lambda, \alpha, T}\,\frac{1\,+\, \nu_{k,\alpha+3}}{K^2}
	 \end{align*}
 	for $K\in (1, \infty)$. In particular,
 	\begin{equation}\label{Eq15}
 	P\left(\left(F^{(k)}_{N,K}\right)^c \right) \lesssim_{\Lambda, \alpha, q, T} 
 	\frac{1\,+\, \nu_{k, \alpha+3}}{K},
 	\end{equation}
 	where we define 
 	\[F^{(k)}_{N,K}\, =\, \left\{
 	\max\left\{
 	\sup_{0\le t\le T} \|u_N^{(k)}\|_{H^1(\TT)},\left[u_N^{(k)}\right]_{\frac{1}{4} ,W^{-1, q'}(\TT)}, 
 	\left\|\left(v_{N}^{(k)}\right)^{\frac{\alpha+3}{2}}\right\|_{L^2(0,T;H^2(\TT))}
 	\right\}\le K
 	\right\}.
 	\]
 	Moreover, using that for $a,b \ge 0$ we have 
 	\[
 	|a^\frac{\alpha+3}{2}-b^\frac{\alpha+3}{2}|\, \le \,  \frac{\alpha+3}{2}|a-b|\max(a,b)^{\frac{\alpha+1}{2}}
 	\, \lesssim \, |a-b|\left[a^{\frac{\alpha+1}{2}} + b^{\frac{\alpha+1}{2}} \right]
 	\]
 	as a consequence of the fundamental theorem of calculus, we deduce that
 	\begin{align*}&
 		\left\|\tau_h \left(v_{N}^{(k)}\right)^{\frac{\alpha+3}{2}}- \left(v_{N}^{(k)}\right)^{\frac{\alpha+3}{2}} \right\|_{L^2\left(0,T-h;L^{\frac{4}{\alpha+3}}(\TT)\right)} \\
 		\lesssim \,&
 		\left\|\left(\tau_h v_{N}^{(k)}- v_{N}^{(k)}\right) \left(\tau_h \left(v_{N}^{(k)}\right)^{\frac{\alpha+1}{2}}+ \left(v_{N}^{(k)}\right)^{\frac{\alpha+1}{2}}\right) 
 		\right\|_{L^2\left(0,T-h;L^{\frac{4}{\alpha+3}}(\TT)\right)}
 		\\\lesssim\,& 
 		 \|\tau_h v_{N}^{(k)}- \tilde{v}_N\|_{L^4(0,T-h; L^2(\TT))} \left\| \left(v_{N}^{(k)}\right)^{\frac{\alpha+1}{2}}\right\|_{L^4\left(0,T, L^\frac{4}{\alpha+1}(\TT)\right)}
 	\end{align*}
 	from H\"older's inequality. We estimate the latter term by 
 	\begin{align*}
 		 \left\| \left(v_{N}^{(k)}\right)^{\frac{\alpha+1}{2}}\right\|_{L^4\left(0,T, L^\frac{4}{\alpha+1}(\TT)\right)} \,& =\,
 		\left(\int_0^T \left( \int_{\TT} \left(v_{N}^{(k)}\right)^2\,dx\right)^{\alpha+1} dt\right)^\frac{1}{4}
 		\, \lesssim_T \,	\sup_{0\le t\le T} \|v_N^{(k)}\|_{L^2(\TT)}^\frac{\alpha+1}{2}.
 	\end{align*} We conclude by \eqref{Eq12} that
 \begin{align*}&
 	\left\|\tau_h \left(v_{N}^{(k)}\right)^{\frac{\alpha+3}{2}}- \left(v_{N}^{(k)}\right)^{\frac{\alpha+3}{2}} \right\|_{L^2\left(0,T-h;L^{\frac{4}{\alpha+3}}(\TT)\right)}\\\lesssim_{q,T}\,&
 	\, 
 	[u^{(k)}_N]_{\frac{1}{4},W^{-1, q'(\TT)}(\TT)} ^{1-\theta}\|u^{(k)}_N\|_{L^\infty(0,T;H^1(\TT))}^{\theta+\frac{\alpha+1}{2}} (g_{\theta,T}(h))^\frac{1}{4}.
 \end{align*}
Hence,  for $\omega\in F^{(k)}_{N,K}$ we have that
 \begin{align*}
	\left\|\tau_h \left(v_{N}^{(k)}(\omega)\right)^{\frac{\alpha+3}{2}}- \left(v_{N}^{(k)}(\omega)\right)^{\frac{\alpha+3}{2}} \right\|_{L^2\left(0,T-h;L^{\frac{4}{\alpha+3}}(\TT)\right)}\,\lesssim_{q,T}\,
	\, K^{\frac{\alpha+3}{2}} (g_{\theta,T}(h))^\frac{1}{4}.
\end{align*}
and 
\[
\left\|\left(v_{N}^{(k)}(\omega)\right)^{\frac{\alpha+3}{2}}\right\|_{L^2(0,T;H^2(\TT))}\, \le \, K
\]
and therefore $v_{N}^{(k)}(\omega)$ lies in a  compact subset of $L^2(0, T; H^1(\TT))$ by \cite[Theorem 5, p.84]{Simon1987}, which we denote by $\chi_{q, \alpha,T,K}$. From \eqref{Eq15} we deduce that 
\[
P\left(\left\{v_{N}^{(k)}\notin \chi_{q, \alpha,T,K}\right\} \right)\,  \lesssim_{\Lambda, \alpha, q, T} 
\frac{1+ \nu_{k,\alpha+3}}{K}.
\]
The tightness assertion follows since the right hand side goes uniformly in $N$ to $0$ as $K\to \infty$.
\end{proof}
\subsection{The time-step limit}

In this last subsection, we finally let $N\to \infty$ and show that the limit satisfies the assertions of Theorem \ref{Thm_main}. This time we consider the sequence 
\begin{equation}\label{Eq_seq_2}
\left( \left(
\mathbbm{1}_{R^{(l)}_N}, \beta^{(l)}, u_N^{(l)}, v_N^{(l)}, w_N^{(l)}, J_N^{(l)}, 
(v_N^{(l)})^\frac{\alpha+3}{2},
(v_N^{(l)})^\frac{\alpha+3}{4} , (v_N^{(l)})^\frac{\alpha+3}{2}\right)_{l\in \NN}
\right)_{N \in \mathbb{N}}\end{equation}
in the topological space 
\begin{align}\begin{split}
		\label{Eq_space_2}
		\prod_{l=1}^{\infty}\;&
		\RR \times C([0,T])\times C(0,T; L^2(\TT)) \times  L_{w*}^\infty(0,T; H^1(\TT))\times  L_{w*}^\infty(0,T; H^1(\TT)) \\&\times  L_{w}^2(0,T; L^{q'}(\TT))\times
		L_w^2(0,T; H^2(\TT))\times L_w^4(0,T; W^{1,4}(\TT))\times  L^2(0,T; H^1(\TT)).
	\end{split}
\end{align}
Notice that this differs from \eqref{Eq_space} by the additional appearance of the space $ L^2(0,T; H^1(\TT))$.
\begin{cor}
	The sequence \eqref{Eq_seq_2} is tight on \eqref{Eq_space_2}.
\end{cor}
\begin{proof}
	This can be shown analogously to Proposition \ref{Prop_tNess1}, using Theorem \ref{Thm_degenerate_sol} \ref{Item_est_deg} and \ref{Item_est_Hoelder} and invoking additionally the findings from Lemma \ref{Bound_powers_2} and  Lemma \ref{Lemma_add_Tness}.
\end{proof}
As in subsection \ref{Sec_deg_lim}, we obtain that for a subsequence indexed by $\mathcal{N}\subset \mathbb{N}$, there exists a sequence of $\mathfrak{B}$-measurable random variables
\begin{equation*}
	\left( 
	\left(
	\mathbbm{1}_{\tilde{R}_N^{(l)}}, \tilde{\beta}_N^{(l)}, \tilde{u}_N^{(l)}, \tilde{v}_N^{(l)}, \tilde{w}_N^{(l)}, \tilde{J}_N^{(l)}, 	\tilde{f}_N^{(l)}, 
	\tilde{g}_N^{(l)} , 	\tilde{h}_N^{(l)} \right)_{l\in \NN}
	\right)_{N \in \mathcal{N}}
\end{equation*} 
defined on a complete probability space $(\tilde{\Omega}, \tilde{\mathfrak{A}}, \tilde{P})$, such that 
\begin{equation}\label{Eq_elem_of_new_seq_II}
	\left(
	\mathbbm{1}_{\tilde{R}_N^{(l)}}, \tilde{\beta}_N^{(l)}, \tilde{u}_N^{(l)}, \tilde{v}_N^{(l)}, \tilde{w}_N^{(l)}, \tilde{J}_N^{(l)}, 	\tilde{f}_N^{(l)}, 
	\tilde{g}_N^{(l)} , 	\tilde{h}_N^{(l)} \right)_{l\in \NN}
\end{equation} 
has the same distribution on \eqref{Eq_space_2} as 
\begin{equation}\label{Eq_elem_of_old_seq_II}
\left(
\mathbbm{1}_{R^{(l)}_N}, \beta^{(l)}, u_N^{(l)}, v_N^{(l)}, w_N^{(l)}, J_N^{(l)}, 
(v_N^{(l)})^\frac{\alpha+3}{2},
(v_N^{(l)})^\frac{\alpha+3}{4} , (v_N^{(l)})^\frac{\alpha+3}{2}\right)_{l\in \NN}
\end{equation}
for every $N \in \mathcal{N}$. Moreover, as $N \to \infty$, \eqref{Eq_elem_of_new_seq_II} converges to a $\mathfrak{B}$-measurable random variable
\begin{equation}\label{Eq1000}
	\left(
	\mathbbm{1}_{\tilde{R}^{(l)}}, \tilde{\beta}^{(l)}, \tilde{u}^{(l)}, \tilde{v}^{(l)}, \tilde{w}^{(l)}, \tilde{J}^{(l)}, 	\tilde{f}^{(l)}, 
	\tilde{g}^{(l)} , 	\tilde{h}^{(l)} \right)_{l\in \NN}
\end{equation} 
in \eqref{Eq_space_2}. The following properties are inherited from the approximating sequence.
\begin{lemma}\label{Lemma_Basic_prop2}
	The sets $(\tilde{R}^{(k)})_{k\in \NN}$ form, up to $\tilde{P}$-null sets, a disjoint partition of $\tilde{\Omega}$. Moreover, the following holds $\tilde{P}$-almost surely for every $k\in\NN$.
		\begin{enumerate}[label=(\roman*)]
		\item \label{Item_1} The random variables 
		\[
		\tilde{u}^{(k)}, \tilde{v}^{(k)}, \tilde{w}^{(k)}, \tilde{J}^{(k)}, 	\tilde{f}^{(k)}, 	\tilde{g}^{(k)}\;\; \text{and}\;\; 
		\tilde{h}^{(k)}
		\]
		vanish outside of  of ${\tilde{R}^{(k)}}$.
		\item \label{Item_2} $\tilde{u}^{(k)}(t)\ge 0$ for all $t\in [0,T]$.
		\item \label{Item_3}	 $\tilde{u}^{(k)}=\tilde{v}^{(k)}=\tilde{w}^{(k)}$.
		\item \label{Item_4}
		$\tilde{f}^{(k)}=\tilde{h}^{(k)}=(\tilde{u}^{(k)})^\frac{\alpha+3}{2}$ and
		$\tilde{g}^{(k)}=(\tilde{u}^{(k)})^\frac{\alpha+3}{4}$.
	\end{enumerate}
\end{lemma}
\begin{proof}
	The claim regarding the sets $\tilde{R}^{(k)}$, as well as part \ref{Item_1} and \ref{Item_2} follow as in the proof of Lemma \ref{lemma_elem_propI}. For part \ref{Item_3} 
	we conclude first form \eqref{Eq52} that $\tilde{P}$-almost surely
	\begin{equation}\label{Eq16}
	\tilde{v}^{(k)}_N(t)= \tilde{u}^{(k)}_N\left(
	j\delta + \tfrac{t-j\delta}{2}
	\right), \;\; t\in [j\delta, (j+1)\delta).
	\end{equation}
	for almost all $t\in [0,T]$. Fixing such $t$ that \eqref{Eq16} holds for all $N\in \mathcal{N}$ and  using that $\hat{u}_N$ converges uniformly to an $L^2(\TT)$-continuous function we conclude that 
	\begin{equation}\label{Eq25}
		\|	\tilde{v}^{(k)}_N(t) -\tilde{u}^{(k)}(t) \|_{L^2(\TT)}<\epsilon
	\end{equation}
	for sufficiently large $N$. It follows that $\tilde{v}^{(k)}_N\to \tilde{u}^{(k)}$ in $L^\infty(0,T; L^2(\TT))$ and therefore the limit has to coincide with $\tilde{v}^{(k)}$. The proof of $\tilde{u}^{(k)}=\tilde{w}^{(k)}$ works analogously. Since in contrast to the proof of Lemma \ref{lemma_elem_propI} we cannot just rely on Proposition \ref{Prop_conv_sol_det} for the identification of  powers in \ref{Item_4}, we carry out the argument  by hand. Since \eqref{Eq_elem_of_new_seq_II} and \eqref{Eq_elem_of_old_seq_II} have the same distribution, it holds
	\begin{equation}\label{Eq18}
	\tilde{f}^{(k)}_N=(\tilde{v}^{(k)}_N)^\frac{\alpha+3}{2}
	\end{equation}
	for every $N\in \mathcal{N}$. Due to the previously verified convergence  $\tilde{v}^{(k)}_N\to \tilde{u}^{(k)}$ in $L^\infty(0,T; L^2(\TT))$ it follows that the same convergence holds $[{0,T}]\times \TT$-almost everywhere up to taking a subsequence. Moreover, since $\tilde{v}^{(k)}_N$ is also weakly convergent in $L^\infty(0,T; H^1(\TT))$, we conclude that it is uniformly in $N$ bounded in $L^r([{0,T}]\times \TT)$ for every $r>0$ by the Sobolev embedding theorem. Vitali's convergence theorem yields that
	\begin{equation}\label{Eq17}
	(\tilde{v}^{(k)}_N)^\frac{\alpha+3}{2} \,\to\, (\tilde{u}^{(k)})^\frac{\alpha+3}{2}
	\end{equation}
	in $L^r([{0,T}]\times \TT)$ for every $r>0$. Invoking \eqref{Eq18} and that $\tilde{f}^{(k)}_N\rightharpoonup f^{(k)}$ in $L^2(0,T; H^2(\TT))$ the identification $ \tilde{f}^{(k)}=(\tilde{u}^{(k)})^\frac{\alpha+3}{2}$ follows. The remaining part of \ref{Item_4} can be shown analogously.
\end{proof}
\begin{prop}\label{Prop_PDE_part}
	 For all $\eta \in L^\infty(0,T; W^{2, \infty}(\TT))$ it holds
	\begin{align*}\begin{split}\int_0^T\int_{\TT} \tilde{J}^{(k)}\cdot \eta \, dx\,dt\,=\,&
			\int_0^T\int_{\{\tilde{u}^{(k)}(s)>0\}} |\nabla \tilde{u}^{(k)}|^2 \nabla \tilde{u}^{(k)} \cdot \eta \, dx \,ds\\&\,+\,
			\int_0^T\int_{\{\tilde{u}^{(k)}(s)>0\}} \tilde{u}^{(k)} |\nabla \tilde{u}^{(k)}|^2 \divv \eta\, dx \,ds\\&\,+\,
			2\int_0^T\int_{\{\tilde{u}^{(k)}(s)>0\}} \tilde{u}^{(k)}\,\nabla^T \tilde{u}^{(k)} D\eta \nabla \tilde{u}^{(k)}\, dx \,ds\\&\,+\,
			\int_0^T\int_{\TT} {(\tilde{u}^{(k)})}^2\nabla \tilde{u}^{(k)}\cdot \nabla \divv\eta\, dx \,ds
		\end{split}
	\end{align*}
$\tilde{P}$-almost surely.
\end{prop}
\begin{proof}
	Since \eqref{Eq_elem_of_new_seq_II} and \eqref{Eq_elem_of_old_seq_II} have the same distribution, we conclude that
	\begin{align*}\begin{split}\int_{j\delta}^{(j+1)\delta}\int_{\TT}\tilde{J}^{(k)}_N\cdot \eta \, dx\,dt\,=\,&
			\int_{j\delta}^{(j+1)\delta}\int_{\{\tilde{v}^{(k)}_N(s)>0\}} |\nabla \tilde{v}^{(k)}_N|^2 \nabla \tilde{v}^{(k)}_N \cdot \eta \, dx \,ds\\&\,+\,
			\int_{j\delta}^{(j+1)\delta}\int_{\{\tilde{v}^{(k)}_N(s)>0\}} \tilde{v}^{(k)}_N |\nabla \tilde{v}^{(k)}_N|^2 \divv \eta\, dx \,ds\\&\,+\,
			2\int_{j\delta}^{(j+1)\delta}\int_{\{\tilde{v}^{(k)}_N(s)>0\}} \tilde{v}^{(k)}_N\,\nabla^T \tilde{v}^{(k)}_N D\eta \nabla \tilde{v}^{(k)}_N\, dx \,ds\\&\,+\,
			\int_{j\delta}^{(j+1)\delta}\int_{\TT} (\tilde{v}^{(k)}_N)^2\nabla \tilde{v}^{(k)}_N\cdot \nabla \divv\eta\, dx \,ds
		\end{split}
	\end{align*}by Theorem \ref{Thm_degenerate_sol} \ref{Item_sol_Det_deg_case}
	for every $N\in \mathcal{N}$ and $j\in \{0, \dots , N\}$. Summing up over $j$ yields that
	\begin{align*}\begin{split}\int_0^T\int_{\TT} \tilde{J}^{(k)}_N\cdot \eta \, dx\,dt\,=\,&
			\int_0^T\int_{\{\tilde{v}^{(k)}_N(s)>0\}} |\nabla \tilde{v}^{(k)}_N|^2 \nabla \tilde{v}^{(k)}_N \cdot \eta \, dx \,ds\\&\,+\,
			\int_0^T\int_{\{\tilde{v}^{(k)}_N(s)>0\}} \tilde{v}^{(k)}_N |\nabla \tilde{v}^{(k)}_N|^2 \divv \eta\, dx \,ds\\&\,+\,
			2\int_0^T\int_{\{\tilde{v}^{(k)}_N(s)>0\}} \tilde{v}^{(k)}_N\,\nabla^T \tilde{v}^{(k)}_N D\eta \nabla \tilde{v}^{(k)}_N\, dx \,ds\\&\,+\,
			\int_0^T\int_{\TT} \tilde{v}^{(k)}_N\nabla \tilde{v}^{(k)}_N\cdot \nabla \divv\eta\, dx \,ds.
		\end{split}
	\end{align*}
	It is left to take the limit $N\to \infty$ in the above equality, which works exactly as in the deterministic case \cite[Corollary 2.7, Theorem 3.2]{Passo98ona}. 
\end{proof}
\begin{rem}
	We stress  the importance of the additional convergence  $ (\tilde{v}^{(k)}_N)^{\frac{\alpha+3}{2}} \to (\tilde{u}^{(k)})^\frac{\alpha+3}{2}$ in $L^2(0,T; H^1(\TT))$ for the limiting argument \cite[Corollary 2.7, Theorem 3.2]{Passo98ona}. 
\end{rem}
The previous statement shows that the weak formulation of $\tilde{J}^{(k)}=(\tilde{u}^{(k)})^2\nabla \Delta(\tilde{u}^{(k)})$ as in \eqref{Defi_J_is_nonlin_part} is satisfied. 
We gather some more convergence and integrability results before we recover  \eqref{Defi_SPDE} as well. We note that we use again the convention to identify $\tilde{v}^{(k)}_N\in L^\infty(0,T; H^1(\TT))$ with its version defined by \eqref{Eq16} as well as the rounding function $\lfloor \cdot \rfloor_\delta$ from the proof of Lemma \ref{Lemma_add_Tness}.
\begin{lemma}\label{Lemma_conv_prop}
	For every $\varphi\in W^{1,q}(\TT)$, $k\in \NN$ and $t\in [0,T]$ it  holds 
	\begin{align}&\label{Eq21}
		\left<\tilde{v}_N^{(k)}(t), \varphi\right>\, \to \, 
		\left<\tilde{u}^{(k)}(t), \varphi\right>,\\&\label{Eq22}
		\int_0^t \left<\divv(\tilde{J}_N^{(k)}), \varphi\right> \, ds\, \to \, \int_0^t \left<\divv(\tilde{J}^{(k)}_N), \varphi\right> \, ds,\\&\label{Eq23}
		\sum_{l=1}^\infty\int_0^{\lfloor t\rfloor_\delta} \lambda_l^2\left<\divv (\divv(\tilde{w}^{(k)}_N(s)\psi_l)\psi_l), \varphi\right>\, ds \, \to \,
		\sum_{l=1}^\infty\int_0^t \lambda_l^2\left<\divv (\divv(\tilde{u}^{(k)}(s)\psi_l)\psi_l), \varphi\right>\, ds,\\
		&\label{Eq24}\int_{{0}}^{{\lfloor t\rfloor_\delta}}
		\sum_{l=1}^\infty \lambda_l^2 \left<
		\divv(\psi_l \tilde{w}^{(k)}_N), \varphi
		\right>^2\, ds \, \to \,\int_{0}^{{t}}
		\sum_{l=1}^\infty \lambda_l^2 \left<
		\divv(\psi_l \tilde{u}^{(k)}), \varphi
		\right>^2\, ds,\\
		&\label{Eq26}\int_{0}^{{\lfloor t\rfloor_\delta}}  \lambda_l \left<
		\divv(\psi_l \tilde{w}^{(k)}_N), \varphi
		\right> d\tau \, \to \, \int_{{ 0 }}^{{ t}}  \lambda_l \left<
		\divv(\psi_l \tilde{u}^{(k)}), \varphi
		\right> d\tau, \\&
		\label{Eq27}
		\tilde{\beta}_N^{(l)}({{\lfloor t\rfloor_\delta}})\, \to \, 
		\tilde{\beta}^{(l)}({t})
	\end{align}
$\tilde{P}$-almost surely as $N\to \infty$.
\end{lemma}
\begin{proof}The convergence \eqref{Eq21}
	follows by \eqref{Eq25}. Part \eqref{Eq22} is a direct consequence of $\tilde{J}_N^{(k)}\rightharpoonup \tilde{J}^{(k)}$ in $L^2(0,T; L^{q'}(\TT))$. Next, we observe that
	\[\|
	\psi_l\cdot \nabla (\psi_l\cdot \nabla \varphi) \|_{H^{-1}(\TT)} \, \lesssim \,\|\varphi\|_{H^1(\TT)} 
	\]
	due to \eqref{boundedness_psil}. Using $\tilde{w}^{(k)}_N\rightharpoonup^* \tilde{u}^{(k)} $ in $L^\infty(0,T; H^1(\TT))$ we deduce that
		\begin{equation}\label{Eq30}
	\sum_{l=1}^\infty\int_0^{t} \lambda_l^2\left<\tilde{w}^{(k)}_N(s), \psi_l\cdot \nabla (\psi_l\cdot \nabla \varphi)\right>\, ds \, \to \,
	\sum_{l=1}^\infty\int_0^t \lambda_l^2\left<\tilde{u}^{(k)}(s), \psi_l\cdot \nabla (\psi_l\cdot \nabla \varphi)\right>\, ds.
	\end{equation}
	Since weak-* convergent sequences are norm bounded, we obtain  \eqref{Eq23} by combining 
	\[
	\left|\sum_{l=1}\lambda_l^2
	\int_{\lfloor t\rfloor_\delta}^{t}  \left<\tilde{w}^{(k)}_N(s), \psi_l\cdot \nabla (\psi_l\cdot \nabla \varphi)\right> ds
	\right|\, \le \,\delta\sum_{l=1}^\infty \lambda_l^2  \|\tilde{w}^{(k)}_N\|_{L^\infty(0,T; H^1(\TT))} \|\psi_l\cdot \nabla (\psi_l\cdot \nabla \varphi)\|_{H^{-1}}\]
	with	\eqref{Eq30}. 
	For \eqref{Eq24}, we estimate 
	\begin{align*}&\left|
		\int_{{0}}^{{ t}}
		\sum_{l=1}^\infty \lambda_l^2 \left<
		\divv(\psi_l \tilde{w}^{(k)}_N), \varphi
		\right>^2\, ds \, - \,\int_{0}^{{t}}
		\sum_{l=1}^\infty \lambda_l^2 \left<
		\divv(\psi_l \tilde{u}^{(k)}), \varphi
		\right>^2\, ds\right|\\\le\,&
		 \sum_{l=1}^\infty \lambda_l^2 	\int_{{0}}^{{T}}
		 \left|\left<
		  \tilde{w}^{(k)}_N, \psi_l\cdot \nabla\varphi
		 \right>^2-\left<
		  \tilde{u}^{(k)}, \psi_l\cdot \nabla \varphi
		 \right>^2\right|\, ds
		 \\\lesssim\,&
		  \sum_{l=1}^\infty \lambda_l^2 	\int_{{0}}^{{T}}
		 \left|\left<
		 \tilde{w}^{(k)}_N, \psi_l\cdot \nabla\varphi
		 \right>-\left<
		 \tilde{u}^{(k)}, \psi_l\cdot \nabla \varphi
		 \right>\right| \left( \left|\left<
		 \tilde{w}^{(k)}_N, \psi_l\cdot \nabla\varphi
		 \right>\right|\,+\,\left|\left<
		 \tilde{u}^{(k)}, \psi_l\cdot \nabla \varphi
		 \right>\right|\right)\, ds
		 ,
	\end{align*}
where we employed that 
\[
a^2-b^2\,\le \, 2|a-b|\max (|a|, |b|)\,\le\, 2|a-b| (|a|+|b|)
\]
for $a,b\in \RR$. Since $\| \psi_l\cdot \nabla \varphi\|_{L^2(\TT)}$ is uniformly in $l$ bounded, we obtain further that
\begin{align}\begin{split}\label{Eq36}
		&\left|
		\int_{{0}}^{{ t}}
	\sum_{l=1}^\infty \lambda_l^2 \left<
	\divv(\psi_l \tilde{w}^{(k)}_N), \varphi
	\right>^2\, ds \, - \,\int_{0}^{{t}}
	\sum_{l=1}^\infty \lambda_l^2 \left<
	\divv(\psi_l \tilde{u}^{(k)}), \varphi
	\right>^2\, ds\right|\\\lesssim_{\Lambda, \varphi}\,&\int_0^T\|\tilde{w}^{(k)}_N- \tilde{u}^{(k)}\|_{L^2(\TT)}\left(
	\|\tilde{u}^{(k)}\|_{L^2(\TT)}\,+\, \|\tilde{w}_N^{(k)}\|_{L^2(\TT)}
	\right)\,dt
	\\\lesssim_{T}\,&\|\tilde{w}^{(k)}_N- \tilde{u}^{(k)}\|_{L^\infty(0,T;L^2(\TT))}\left(\| \tilde{u}^{(k)}\|_{L^\infty(0,T;L^2(\TT))}\,+\,\sup_{N\in \mathcal{N}} \|\tilde{w}_N^{(k)}\|_{L^\infty(0,T;L^2(\TT))}\right)
	\end{split}
\end{align}
As a consequence of the proof of Lemma \ref{Lemma_Basic_prop2} \ref{Item_3} we have $\tilde{w}^{(k)}_N\to \tilde{u}^{(k)}$ in $L^\infty(0,T;L^2(\TT))$ as $N\to \infty$ and therefore the right-hand side of \eqref{Eq36} tends to $0$. Using that by the same arguments
\[
\left|\int_{{\lfloor t\rfloor_\delta}}^t
\sum_{l=1}^\infty \lambda_l^2 \left<
\divv(\psi_l \tilde{w}^{(k)}_N), \varphi
\right>^2\, ds  \right|\,\lesssim_{\Lambda, \varphi} \,\delta \|\tilde{w}^{(k)}_N\|_{L^\infty(0,T; L^2(\TT))}^2
\]
we obtain indeed \eqref{Eq24}.
	The convergence \eqref{Eq26} can be derived analogously to \eqref{Eq23}.
	The last assertion \eqref{Eq27} is a consequence of $\tilde{\beta}_N^{(l)}\to \tilde{\beta}^{(l)}$ in $C([0,T])$.
\end{proof}
\begin{lemma}\label{Lemma_bounds_final}It holds for every $k\in \NN$ that 
	\begin{align*}
		\tilde{E}\left[
		\sup_{0\le t\le T} \|\tilde{u}^{(k)}\|_{H^1(\TT)}^p
		\right]\, &\lesssim_{\Lambda, p, T} \, \nu_{k,p},\\
		\tilde{E}\left[ \|\tilde{J}^{(k)}\|_{L^2(0,T; L^{q'}(\TT))}^\frac{p}{2} \right]\, &\lesssim_{\Lambda, p,q , T}\,\nu_{k,p},
			\\ \tilde{E}\left[
		\left\|\left(\tilde{u}^{(k)}\right)^{\frac{\alpha+3}{2}}\right\|^2_{L^2(0,T;H^2(\TT))}+\,
		\left\|\left(\tilde{u}^{(k)}\right)^{\frac{\alpha+3}{4}}\right\|^4_{L^4(0,T;W^{1,4}(\TT))}
		\right]\, &\lesssim_{ \Lambda, \alpha, T}\,1+\nu_{k,\alpha+3}.
	\end{align*}
\end{lemma}
\begin{proof}
	This follows since \eqref{Eq_elem_of_new_seq_II} and \eqref{Eq_elem_of_old_seq_II} have the same distribution, lower semi-continuity of the norm with respect to weak and weak-* convergence, as well as the bounds from Theorem \ref{Thm_degenerate_sol} \ref{Item_est_deg} and Lemma \ref{Bound_powers_2}.
\end{proof}
Finally, we define, as in subsection \ref{Sec_deg_lim}, $\tilde{\mathfrak{F}}$ as the augmentation of the filtration $\tilde{\mathfrak{G}}$ given by 
\[
\tilde{\mathfrak{G}}_t\,=\, \sigma \left(\left\{\left.
\mathbbm{1}_{\tilde{R}^{(l)}}, \tilde{J}^{(l)}|_{[0,t]} 
\right| \,l\in \NN
\right\}\cup 
\left\{\left. \tilde{u}^{(l)}(s), \tilde{\beta}^{(l)}(s)\right|
0\le s\le t, \,l\in \NN
\right\}\right),
\]
where we consider $J|_{[0,t]}$ again as a $\mathfrak{B}$-random variable in $L^2(0,t; L^{q'}(\TT))$.
\begin{rem}\label{Rem_meas2}
	The smallest $\sigma$-field $\tilde{\mathfrak{H}}_t$ on $\tilde{\Omega}$ such that $\phi(X)$ with
	\[X=	\left(
	\mathbbm{1}_{\tilde{R}^{(l)}}, \tilde{\beta}^{(l)}|_{[0,t]}, \tilde{u}^{(l)}|_{[0,t]}, \tilde{J}^{(l)}|_{[0,t]} \right)_{l\in \NN}
	\]
	is measurable for every bounded and continuous function 
	\begin{equation}\label{Eq31}
	\phi\colon\, \prod_{l=1}^{\infty}
	\RR \times C([0,t])\times C(0,t; L^2(\TT, \RR^2))\times L_w^2(0,t; L^{q'}(\TT)) \, \to \, \RR
	\end{equation}
	coincides with $\tilde{\mathfrak{G}}_t$. Indeed, the inclusion $\tilde{\mathfrak{G}}_t\subset \tilde{\mathfrak{H}}_t$ follows since one can choose $\phi$ as a function depending only on one of the components of
	\begin{equation*}
		\prod_{l=1}^{\infty}
		\RR \times C([0,t])\times C(0,t; L^2(\TT, \RR^2))\times L_w^2(0,t; L^{q'}(\TT)).
	\end{equation*}
	For the reverse inclusion $\tilde{\mathfrak{H}}_t\subset \tilde{\mathfrak{G}}_t$, we assume that $\phi$ as in \eqref{Eq31} is bounded and continuous, such that it suffices to show that $\phi(X)$ is measurable with respect to $\tilde{\mathfrak{G}}_t$.
	In particular, $\phi$ is continuous as mapping from 
		\begin{equation}\label{Eq32}
		\prod_{l=1}^{\infty}
		\RR \times C([0,t])\times C(0,t; L^2(\TT, \RR^2))\times L^2(0,t; L^{q'}(\TT))
	\end{equation}
into $\RR$. But \eqref{Eq32} is a complete separable metric space such that $\tilde{\mathfrak{G}}_t$-$\mathfrak{B}$ measurability of the \eqref{Eq32}-valued random variable $X$ can be checked using a suitable family of functions separating the points  by \cite[Theorem 6.8.9]{bogachev2007measure}.
\end{rem}
\begin{thm}\label{Thm_final_SPDE}The processes $(\tilde{\beta}^{(l)})_{l\in \NN}$ are a family of independent $\tilde{\mathfrak{F}}$-Brownian motions. Moreover, we have for every $k\in \NN$, $\varphi\in W^{1, q'}(\TT)$ and $t\in [0,T]$ 
	\begin{align*}
		\left<
		\tilde{u}^{(k)}(t), \varphi
		\right>\, - \,\left<
		\tilde{u}^{(k)}(0), \varphi
		\right>\,=\,&
		\int_0^t -\left<\divv(\tilde{J}^{(k)}), \varphi\right> \, ds
		\, +\,\sum_{l=1}^\infty\int_0^t \lambda_l^2\left<\divv (\divv(\tilde{u}^{(k)}(s)\psi_l)\psi_l), \varphi\right>\, ds\\&+
		\sum_{l=1}^\infty \lambda_l \int_{0}^t\left< \divv(\tilde{u}^{(k)}(s)\psi_l), \varphi \right>\, d\tilde{\beta}^l_s
	\end{align*}
$\tilde{P}$-almost surely.
\end{thm}
\begin{proof}For the claim regarding the family $(\tilde{\beta}^{(l)})_{l\in \NN}$ we refer to \cite{fischer_gruen_2018}
	For the remainder of the proof we fix $k\in \NN$ and $\varphi\in W^{1, q}(\TT)$ and define the $\tilde{\mathfrak{F}}$-adapted process
	\begin{align*}
	M(t) \,=\, &	\left<
	\tilde{u}^{(k)}(t), \varphi
	\right>\, - \,\left<
	\tilde{u}^{(k)}(0), \varphi
	\right>\,+\,
	\int_0^t \left<\divv(\tilde{J}^{(k)}), \varphi\right> \, ds
	\\& -\,\sum_{l=1}^\infty\int_0^t \lambda_l^2\left<\divv (\divv(\tilde{u}^{(k)}(s)\psi_l)\psi_l), \varphi\right>\, ds
	\end{align*}
	and the approximating processes
	\begin{align*}
		M_N(t) \,=\, &	\left<
		\tilde{v}^{(k)}_N(t), \varphi
		\right>\, - \,\left<
		\tilde{u}_N^{(k)}(0), \varphi
		\right>\,+\,
		\int_0^t \left<\divv(\tilde{J}^{(k)}), \varphi\right> \, ds
		\\& -\,\sum_{l=1}^\infty\int_0^{\lfloor t\rfloor_\delta} \lambda_l^2\left<\divv (\divv(\tilde{u}^{(k)}(s)\psi_l)\psi_l), \varphi\right>\, ds
	\end{align*}
As a consequence of  \eqref{Eq21}, \eqref{Eq22} and \eqref{Eq23}, $M_N(t)$ converges to $M(t)$ for every $t\in [0,T]$ as $N\to \infty$. Moreover, we let
\begin{equation}\label{Eq91}
	\phi\colon\, \prod_{l=1}^{\infty}
	\RR \times C([0,s])\times C(0,s; L^2(\TT, \RR^2))\times L_w^2(0,s; L^{q'}(\TT))\,  \to \, \RR
\end{equation}
be bounded and continuous and consider the  random variables
\begin{align*}&
\rho \,=\, \phi\left(	\left(
\mathbbm{1}_{\tilde{R}^{(l)}}, \tilde{\beta}^{(l)}|_{[0,s]}, \tilde{u}^{(l)}|_{[0,s]}, \tilde{J}^{(l)}|_{[0,s]} \right)_{l\in \NN}\right),\\
\rho_N \,&=\, \phi\left(	\left(
\mathbbm{1}_{\tilde{R}_N^{(l)}}, \tilde{\beta}_N^{(l)}|_{[0,s]}, \tilde{u}_N^{(l)}|_{[0,s]}, \tilde{J}_N^{(l)}|_{[0,s]} \right)_{l\in \NN}\right).
\end{align*}
As consequence of the convergence of \eqref{Eq_elem_of_new_seq_II} to \eqref{Eq1000} in \eqref{Eq_space_2} we have $\rho_N\to \rho$ as $N\to \infty$.
Using that as a consequence of Theorem \ref{Thm_degenerate_sol} \ref{Item_sol_Det_deg_case} and \ref{Item_SPDE_deg}
	\begin{align*}&
	\left<
	{v}_{N}^{(k)}(t), \varphi
	\right>\, - \,\left<
	u^{(k)}_N(0), \varphi
	\right>\,+\,
	\int_0^t \left<\divv(J^{(k)}_N), \varphi\right> \, ds
	\\& -\,\sum_{l=1}^\infty\int_0^{\lfloor t\rfloor_\delta} 
	\lambda_l^2\left<\divv (\divv(w^{(k)}_N(s)\psi_l)\psi_l), \varphi\right>\, ds\\=\,& \sum_{l=1}^\infty\lambda_l\int_0^{\lfloor t\rfloor_\delta}\left<\divv(w^{(k)}_N\psi_l),\varphi\right>\, d{\beta}^{l}_s,
\end{align*}
 we conclude by additionally invoking Theorem \ref{Thm_degenerate_sol} \ref{Item_meas} that
 \begin{align}\begin{split}
 		\label{Eq19}&
 	\tilde{E}\left[(
 	M_N(t)-M_N(s+\kappa))\rho_N
 	\right]\,=\,0,\\&
 	\tilde{E}\left[\left(M_{N}^2(t)-M_{N}^2(s+\kappa)-\int_{{\lfloor{s+\kappa}\rfloor_\delta}}^{{\lfloor t\rfloor_\delta}}
 	\sum_{l=1}^\infty \lambda_l^2 \left<
 	\divv(\psi_l \tilde{w}^{(k)}_N), \varphi
 	\right>^2\, d\tau \right)\rho_N\right]\,=\, 0,\\&
 	\tilde{E}\left[\left(\tilde{\beta}_N^{(l)}({{\lfloor t\rfloor_\delta}})M_{ N}(t)-\tilde{\beta}^{(l)}_N({{\lfloor s+\kappa\rfloor_\delta}})M_{N}(s+\kappa)-\int_{{\lfloor s+\kappa \rfloor_\delta}}^{{\lfloor t\rfloor_\delta}}  \lambda_l \left<
 	\divv(\psi_l \tilde{w}^{(k)}_N), \varphi
 	\right> d\tau\right)\rho_N\right]\,=\, 0.
 	\end{split}
 \end{align}for $s,t\in [0,T]$, $\kappa>0$ such that $s+\kappa \le t$, and 
$N$  large enough so that $\lfloor s+\kappa\rfloor_\delta\ge s$. Due to
Theorem \ref{Thm_degenerate_sol} \ref{Item_est_deg}
and the Burkholder-Davis-Gundy inequality we have 
\begin{equation}
\sup_{N\in \NN} \tilde{E}\left[ \|\tilde{w}^{(k)}_N\|_{L^\infty(0,T; H^1(\TT))}^p\,+\,\sup_{ \tau\in [0,T]}|M_N(\tau)|^p\,+\, |\tilde{\beta}_N^{(l)}(\tau)|^p \right]\,<\, \infty
\end{equation} for every $p\in (0,\infty)$
such that  Vitali's convergence theorem and \eqref{Eq24}, \eqref{Eq26}, \eqref{Eq27} yield   
 \begin{align}\begin{split}\label{Eq92}
 		&
	\tilde{E}\left[(
	M(t)-M(s+\kappa))\rho
	\right]\,=\,0,\\&
	\tilde{E}\left[\left(M^2(t)-M^2(s+\kappa)-\int_{{s+\kappa}}^{{t}}
	\sum_{l=1}^\infty \lambda_l^2 \left<
	\divv(\psi_l \tilde{u}^{(k)}), \varphi
	\right>^2\, d\tau \right)\rho\right]\,=\, 0,\\&
	\tilde{E}\left[\left(\tilde{\beta}^{(l)}({t})M(t)-\tilde{\beta}^{(l)}({ s+\kappa})M(s+\kappa)-\int_{{ s+\kappa }}^{{ t}}  \lambda_l \left<
	\divv(\psi_l \tilde{u}^{(k)}), \varphi
	\right> d\tau\right)\rho\right]\,=\, 0.
 	\end{split}
\end{align}
by letting $N\to \infty$ in \eqref{Eq19}. Using that
\[
\left\{
\left\{\phi\left(	\left(\left.
\mathbbm{1}_{\tilde{R}_N^{(l)}}, \tilde{\beta}_N^{(l)}|_{[0,s]}, \tilde{u}_N^{(l)}|_{[0,s]}, \tilde{J}_N^{(l)}|_{[0,s]} \right)_{l\in \NN}\right)\in B\right\}\right|\phi \text{ as in \eqref{Eq91} cont. bdd.},\; B\in \mathfrak{B}(\RR)
\right\}
\]
is an intersection stable generator of $\tilde{\mathfrak{G}}_t$ by Remark \ref{Rem_meas2}, we conclude that \eqref{Eq92} holds for every $\tilde{\mathfrak{G}}_t$-measurable and bounded random variable $\rho$. Finally, we let $0\le s'\le t'\le T$ and $\rho$ be a $\tilde{\mathfrak{F}}_{s'}$-measurable and bounded random variable.
If we can show that
		\begin{align}\begin{split}
				\label{Eq29}
		&\tilde{E}\left[\left(M_{}(t')-M_{}(s')\right)\rho\right]\,=\, 0,\\&\tilde{E}\left[\left(M_{}^2(t')-M_{}^2(s')-\int_{s'}^{t'}
		\sum_{l=1}^\infty \lambda_l^2 \left<
		\divv(\psi_l \tilde{u}^{(k)}), \varphi
		\right>^2\, d\tau \right)\rho\right]\,=\, 0,\\ &\tilde{E}\left[\left(\tilde{\beta}^{(l)}(t)M(t')-\tilde{\beta}^{(l)}(s')M(s')-\int_{s'}^{t'} \lambda_l \left<
		\divv(\psi_l \tilde{u}^{(k)}), \varphi
		\right> d\tau\right)\rho\right]\,=\, 0.
			\end{split}
	\end{align}
the claim follows  by \cite[Proposition A.1]{hofmanova2013}, because $M$ is $\tilde{\mathfrak{F}}$-adapted and square-integrable due to Lemma \ref{Lemma_bounds_final}. To this end, we let  $\kappa'>0$ and define $\kappa=\frac{\kappa'}{2}$, $s=s'+\kappa$ and $t=t'+2\kappa$. By definition of the augmented filtration, there exists a $\tilde{P}$-version of $\rho$ which is $\tilde{\mathfrak{G}_s}$-measurable and moreover we have $s+\kappa\le t$. Therefore, we can apply \eqref{Eq92} and rephrase in terms of $s', t', \kappa'$ to deduce that
 \begin{align*}\begin{split}
		&
		\tilde{E}\left[(
		M(t'+\kappa')-M(s'+\kappa'))\rho
		\right]\,=\,0,\\&
		\tilde{E}\left[\left(M^2(t'+\kappa')-M^2(s'+\kappa')-\int_{{s'+\kappa'}}^{{t'+\kappa'}}
		\sum_{l=1}^\infty \lambda_l^2 \left<
		\divv(\psi_l \tilde{u}^{(k)}), \varphi
		\right>^2\, d\tau \right)\rho\right]\,=\, 0,\\&
		\tilde{E}\left[\left(\tilde{\beta}^{(l)}({t'+\kappa'})M(t'+\kappa')-\tilde{\beta}^{(l)}({ s'+\kappa'})M(s'+\kappa')-\int_{{ s'+\kappa' }}^{{ t'+\kappa'}}  \lambda_l \left<
		\divv(\psi_l \tilde{u}^{(k)}), \varphi
		\right> d\tau\right)\rho\right]\,=\, 0.
	\end{split}
\end{align*}
Since $\kappa'>0$ was arbitrary, we can use continuity of $\tilde{\beta}^{(l)}$ and $M$, Vitalis's theorem and the consequence
\[
E\left[\sup_{\tau\in [0,T]}
\|\tilde{u}^{(k)}(\tau)\|^p_{H^1(\TT)}\,+\,|M(\tau)|^p\,+\,|\tilde{\beta}^{(l)}(\tau)|^p\right]\,<\,\infty
\]
of Lemma \ref{Lemma_bounds_final} to let $\kappa'\searrow 0$ and obtain \eqref{Eq29}.
\end{proof}

Finally, we put
\begin{equation}\label{Eq28}
\tilde{u}\,=\, \sum_{k=0}^\infty \tilde{u}^{(k)}, \;\; \tilde{J}\,=\, \sum_{k=0}^\infty \tilde{J}^{(k)},
\end{equation}
which is in light of Lemma \ref{Lemma_Basic_prop2} equivalent to require
\begin{equation}\label{Eq33}
\tilde{u}\,=\, \tilde{u}^{(k)}\;\; \text{ and }
\;\; \tilde{J}\,=\, \tilde{J}^{(k)} \;\; \text{ on }\tilde{R}^{(k)}. \end{equation}
\begin{proof}[Proof of Theorem \ref{Thm_main}.]
	We first show that $(\tilde{\Omega}, \tilde{\mathfrak{A}}, \tilde{P})$, $\tilde{\mathfrak{F}}$, $(\tilde{\beta}^{(l)})_{l\in \NN}$, $\tilde{u}$ together with $\tilde{J}$ constitute a solution to the stochastic thin-film equation with $q'$-regular non linearity in the sense of Definition \ref{Defi_main}. By definition, $\tilde{\mathfrak{F}}$ fulfills the usual conditions and $(\tilde{\beta}^{(l)})_{l\in \NN}$ is a family of independent Brownian motions by Theorem \ref{Thm_final_SPDE}. Furthermore,  $\tilde{u}$ and $\tilde{J}$ are, as each of their summands, an $H^1_w(\TT)$-continuous, $\tilde{\mathfrak{F}}$-adapted process and a random variable in $L^2(0,T; L^{q'}(\TT))$, respectively. Moreover,  $\tilde{J}|_{[0,t]}$ is $\tilde{\mathfrak{F}}_t$-measurable by definition of $\tilde{\mathfrak{F}}$ and  we have $\sup_{0\le  t\le T} \|\tilde{u}(t)\|_{H^1(\TT)}<\infty$  because of Lemma \ref{Lemma_bounds_final}.  Proposition \ref{Prop_PDE_part} together with \eqref{Eq28} yield that \eqref{Defi_J_is_nonlin_part}
	is indeed fulfilled. Similarly, we obtain \eqref{Defi_SPDE} from Theorem \ref{Thm_final_SPDE}, \eqref{Eq28} and the fact that one can pull the $\tilde{\mathfrak{F}}_0$-measurable random variable $\mathbbm{1}_{\tilde{R}^{(k)}}$ outside of the stochastic integrals in \eqref{Defi_SPDE}. 
	For the initial condition, we observe that 
	\[
	(\mathbbm{1}_{\tilde{R}_N^{(k)}}, \tilde{u}^{(k)}_N(0))_{k\in \NN} \,\sim\, 
	(\mathbbm{1}_{R^{(k)}_N}, u^{(k)}_N(0))_{k\in \NN},\;\; N\in \mathcal{N},
	\]
	and
	\[
	(\mathbbm{1}_{\tilde{R}_N^{(k)}}, \tilde{u}^{(k)}_N(0))_{k\in \NN} \, \to \, 
	(\mathbbm{1}_{\tilde{R}^{(k)}}, \tilde{u}^{(k)}(0))_{k\in \NN}
	\]
	as $N\to \infty$ in $(\RR\times L^2(\TT))^\infty$. Hence, we have
	\[
		(\mathbbm{1}_{\tilde{R}^{(k)}}, \tilde{u}^{(k)}(0))_{k\in \NN}
	\,\sim\,
	(\mathbbm{1}_{{R}_N^{(k)}}, {u}^{(k)}_N(0))_{k\in \NN} 
	\]
	and therefore 
		\[\tilde{u}(0)\, =\, \sum_{k=1}^{\infty} \tilde{u}^{(k)}(0)\, \sim \,\sum_{k=1}^{\infty} {u}^{(k)}_N(0)\, \sim \,
	\mu
	\]
	by
Theorem \ref{Thm_degenerate_sol} \ref{Item_IV}.	
	The non-negativity of $\tilde{u}(t)$ follows from Lemma \ref{Lemma_Basic_prop2} \ref{Item_2}. From Lemma \ref{Lemma_bounds_final} together with the monotone convergence theorem we deduce the energy estimates \eqref{Eq34} and \eqref{Eq35}. Finally due to Lemma \ref{Lemma_Basic_prop2} \ref{Item_4} we conclude that the additional spatial regularity property \eqref{Eq13} is by construction fulfilled. 
\end{proof}

\appendix
\section{Properties of solutions to the deterministic thin-film equation}\label{AppendixA}

	\begin{proof} [Proof of Theorem \ref{Thm_ex_sol_det}] 
	Since $\alpha\in (-1,0)$ and $v\in H^1(\TT)$ we have due to  \eqref{Eq_alpha_entropy_expl} that
	\[
	\int_{\TT}G_\alpha(v_0)\,dx\,<\,\infty.
	\]
	Therefore, \cite[Theorem 3.2]{Passo98ona}
	applies and yields that a weak solution $(v,J)$ with $q'$-regular non-linearity and initial value $v_0$ exists. The first part of \ref{Item_det_cons_of_mass}  follows by testing the equation $\partial_t u=- \divv J$ with $\varphi\otimes\mathbbm{1}_{\TT} $ for arbitrary $\varphi\in C_c^\infty((0,T))$.
	For the other properties we consider the approximation procedure in \cite{Passo98ona}, which takes place in two  steps. First problems of the form
	\begin{align}\label{aux_2}\begin{split}
			\begin{cases}\partial_t
				(v_{\delta\epsilon})\,+\,\divv (J_{\delta\epsilon})\,=\, 0
				& \text{in }L^2(0,T; H^{-1}(\TT)), \\
				J_{\delta\epsilon}\,=\,m_{\delta\epsilon}(v_{\delta\epsilon})\nabla \Delta v_{\delta\epsilon}& \text{weakly},\\
				\esslim_{t\to 0} v_{\delta\epsilon}(t,\cdot)=v_0+\delta+\epsilon^{\theta} &\text{in }H^1(\TT),
			\end{cases}
		\end{split}\tag{$P_{\delta\epsilon}$}
	\end{align}
	are solved by \cite[Theorem 1.1]{Grn1995DegeneratePD}. Letting $\epsilon \searrow 0$ yields solutions to
	\begin{align}\label{aux_1}\begin{split}
			\begin{cases}\partial_t
				(u_\delta)\,+\,\divv (J_{\delta\epsilon})\,=\, 0 \\
				J_{\delta}\,=\,m_{\delta}(v_{\delta})\nabla \Delta v_{\delta} \\
				v_{\delta}(0,\cdot)=v_0+\delta 
			\end{cases}
		\end{split}\tag{$P_\delta$}
	\end{align}
	in the sense of \cite[Definition 3.1]{Passo98ona} which again are used to construct $v$. The functions $m_\delta$ and $m_{\delta\epsilon}$ are auxiliary mobilities, which take the form $m_\delta(\tau)= \frac{\tau^2}{1+\delta \tau^2}$ and $m_{\delta\epsilon}=\frac{\tau^sm_\delta(\tau)}{\epsilon m_\delta(\tau)+\tau^s}$ for some  $s>4$, see \cite[p.323, p.331]{Passo98ona}, so we can choose for example $s=5$. The number $\theta$ from  \eqref{aux_2} is a sufficiently small constant.  By \cite[Lemma 2.1]{Passo98ona} it follows that
	\begin{equation}\label{est_for_regregProb}
		\|\nabla v_{\delta\epsilon} \|_{L^\infty(0,T;L^2(\TT,\RR^2))}\,\le \, \|\nabla v_0\|_{L^2(\TT,\RR^2)}.
	\end{equation}
	Since $v_{\delta\epsilon}\rightharpoonup^* v_\delta$ and 	$v_{\delta}\rightharpoonup^* v$ in $L^\infty(0,T;H^1(\TT))$, see \cite[Proposition 2.6]{Passo98ona}, we conclude under additional consideration of Remark \ref{Rem_version_det_sol} that part \ref{Item_det_energy_est} holds true. Moreover, since also $ v_{\delta\epsilon}\rightharpoonup v_\delta$ and 
	$v_{\delta}\rightharpoonup v$ in $H^1(0,T; W^{-1,q'}(\TT))$, it follows that strong convergence takes place in $C(0,T; L^2(\TT))$ by Remark \ref{Rem_version_det_sol}. Hence non negativity is preserved and the second part of \ref{Item_det_cons_of_mass} follows by the non negativity  of $v_{\delta\epsilon}$, see \cite[Lemma 2.1]{Passo98ona}. Furthermore, in \cite[Equation (2.26)]{Passo98ona} one finds the estimate \ref{Item_det_entr_est}.
	Finally, to convince ourselves also of part \ref{Item_det_energy_est_J}, we conclude that as consequence of \cite[Lemma 2.1]{Passo98ona} it holds
	\begin{equation}\label{Eq1}
	\esssup_{T-\rho\le t\le T}
	\|\nabla v_{\delta\epsilon}(t)\|_{L^2(\TT,\RR^2)}^2
	\,+\,2\int_0^{T-\rho}\|\sqrt{m_{\delta\epsilon}(v_{\delta\epsilon}(t))} \nabla \Delta v_{\delta\epsilon}(t)\|_{L^2(\TT, \RR^2)}^2\,dt\,\le \, \|\nabla v_0\|_{L^2(\TT, \RR^2)}^2
	\end{equation}
	for any $\rho>0$. Since by definition
	\[
	m_{\delta\epsilon}(\tau)\,\le\,
	m_{\delta}(\tau)\,\le \,\tau^2 
	\]
	we obtain by Sobolev's inequality, see \cite[Theorem 4.51]{adams2003sobolev}, the periodic Poincar\'e inequality  and \eqref{est_for_regregProb} that
	\begin{align*}
		\|\sqrt{m_{\delta\epsilon}(v_{\delta\epsilon}(t))}\|_{L^r(\TT)}^2\,&\le\,\|v_{\delta\epsilon}(t)\|_{L^r(\TT)}^2\,\lesssim_r \, \|v_{\delta\epsilon}(t)\|_{H^1(\TT)}^2
	 {	\,\lesssim \, \|\nabla v_{\delta\epsilon}(t) \|_{L^2(\TT, \RR^2)}^2 \,+\, \left|
		\int_{\TT} v_{\delta \epsilon}(t)\, dx\right|^2 }
		\\&\le\, \|\nabla v_{0}\|_{L^2(\TT,\RR^2)}^2+\left|
		\int_{\TT} v_0+\delta+\epsilon^{\theta}\, dx\right|^2
	\end{align*}
	for any $0\le t\le T$ and $r\in[1,\infty)$. Because we have 
	$J_{\delta\epsilon}(t)\,=\,m_{\delta\epsilon}(v_{\delta\epsilon}(t))\nabla \Delta v_{\delta\epsilon}(t) $ for almost all $0\le t\le T$, see \cite[Lemma 2.1]{Passo98ona}, we obtain by \eqref{est_for_regregProb}, \eqref{Eq1}, H\"older's inequality and the choice $\frac{1}{2}+\frac{1}{r}= \frac{1}{q'}$ that
	\begin{align*}&
		C_q \esssup_{T-\rho\le t\le T}\left[
		\|\nabla v_{\delta\epsilon}(t)\|_{L^2(\TT,\RR^2)}^2 \left(\|\nabla v_{\delta\epsilon}(t)\|_{L^2(\TT,\RR^2)}^2+\left|
		\int_{\TT} v_0+\delta+\epsilon^{\theta}\, dx\right|^2
		\right)\right]
		\\&+\,\int_0^{T-\rho}\|J_{\delta\epsilon}(t)\|_{L^{q'}(\TT, \RR^2)}^2\,dt\,\le \, C_q\|\nabla v_0\|_{L^2(\TT, \RR^2)}^2 \left(\|\nabla v_0\|_{L^2(\TT,\RR^2)}^2+\left|
		\int_{\TT} v_0+\delta+\epsilon^{\theta}\, dx\right|^2
		\right).
	\end{align*}
	Using that $J_{\delta\epsilon}\rightharpoonup J_{\delta}$ and $J_{\delta}\rightharpoonup J$
	in $L^2(0,T;L^{q'}(\TT,\RR^2))$ as well as that $\nabla v_{\delta\epsilon}\rightharpoonup^* \nabla v_\delta $ and	$\nabla v_\delta\rightharpoonup^* \nabla v $	in $L^\infty(0,T;L^2(\TT,\RR^2))$, see \cite[Proposition 2.6]{Passo98ona}, we infer that
	\begin{align*}&
		C_q \esssup_{T-\rho\le t\le T}\left[
		\|\nabla v(t)\|_{L^2(\TT,\RR^2)}^2 \left(\|\nabla v(t)\|_{L^2(\TT,\RR^2)}^2+\left|
		\int_{\TT} v_0\, dx\right|^2
		\right)\right]
		\\&+\,\int_0^{T-\rho}\|J(t)\|_{L^{q'}(\TT, \RR^2)}^2\,dt\,\le \, C_q\|\nabla v_0\|_{L^2(\TT, \RR^2)}^2 \left(\|\nabla v_0\|_{L^2(\TT,\RR^2)}^2+\left|
		\int_{\TT} v_0\, dx\right|^2
		\right)
	\end{align*}
	The claimed estimate follows by letting $\rho\searrow 0$ together with the weak continuity of $v$  in $H^1(\TT)$.
\end{proof}
\section{Gelfand triple of Bessel potential spaces}\label{AppendixAB}
The purpose of this section is to verify that $H^2(\TT)\subset H^1(\TT)\subset L^2(\TT) $ is a Gelfand triple, when equipping $H^2(\TT)$ with the Bessel potential norm, as claimed in the proof of Theorem \ref{Thm_Sec_Reg_SPDE}.
We recall that the Bessel potential norm on $H^2(\TT)$  is defined by
\[
\vertiii{f}_{H^2(\TT)}^2\,=\, \sum_{k\in \ZZ^2} (1+|2\pi k|^2)^2|\hat{f}(k)|^2,
\] 
where
\begin{equation}\label{Eq97}
\hat{f}(k)\,=\,\int_{\TT}
f(x) e^{-2\pi i k\cdot x}
\, dx,\;\; k\in \ZZ^2
\end{equation}
is the $k$-th Fourier coefficient of a function $f\in L^2(\TT)$. Moreover, by definition of the Bessel potential spaces  
\[ H^2(\TT)\,=\,\left\{
f\in L^2(\TT)\,|\,
\vertiii{f}_{H^2(\TT)}<\infty
\right\}.
\] The pairing of two functions $f\in H^1(\TT),g\in H^2(\TT)$ in $H^1(\TT)$ can be rewritten by Parseval's relation \cite[Proposition 3.2.7]{grafakos2014classical} as
\begin{equation}\label{Eq93}
(f,g)_{H^1(\TT)}\,=\, (f,g)_{L^2(\TT)}\,+\,  (\nabla f,\nabla g)_{L^2(\TT, \RR^2)}\,=\, 
\sum_{k\in \ZZ^2} (1+|2\pi k|^2) \hat{f}(k) \overline{\hat{g}(k)}
\end{equation}
and therefore
\[ |(f,g)_{H^1(\TT)}|\,\le \, 
\left(\sum_{k\in \ZZ^2} |\hat{f}(k)|^2\right)
\left(\sum_{k\in \ZZ^2} (1+|2\pi k|^2)^2|\hat{g}(k)|^2\right)
\,=\, \|f\|_{L^2(\TT)}\vertiii{g}_{H^2(\TT)}
\]
by the Cauchy-Schwarz inequality. Hence,
\begin{equation}\label{Eq96}
\|
(f,\cdot)_{H^1(\TT)}
\|_{(H^2(\TT))'}\,\le \, \|f\|_{L^2(\TT)}.
\end{equation}
Moreover, since the  coefficients are square summable, the series
\[ \sum_{k\in\ZZ} \frac{\hat{f}(k)}{1+|2\pi k|^2} e^{2\pi i k\cdot x}
\]
converges to an element  $g_f\in L^2(\TT)$. Since 
\begin{equation}\label{Eq95}
\vertiii{g_f}_{H^2(\TT)}^2\,=\, \sum_{k\in \ZZ^2} (1+|2\pi k|^2)^2\left(\frac{|\hat{f}(k)|}{1+|2\pi k|^2}\right)^2\,=\, \|f\|_{L^2(\TT)}^2\,<\,\infty,
\end{equation}
it satisfies  $g_f\in H^2(\TT)$. Using \eqref{Eq93}, we obtain that
\[
(f,g_f)_{H^1(\TT)}\,=\, \sum_{k\in \ZZ^2} (1+|2\pi k|^2) \hat{f}(k) \frac{\overline{\hat{f}(k)}}{1+|2\pi k|^2}\,=\, \|f\|_{L^2(\TT)}^2,
\]
such that
\[
\|
(f,\cdot)_{H^1(\TT)}
\|_{(H^2(\TT))'}\,\ge \, \|f\|_{L^2(\TT)}
\]
by \eqref{Eq95}. Due to \eqref{Eq96},  the above inequality is  an equality. Consequently, identifying $f\in H^1(\TT)$ with $(f,\cdot)_{H^1(\TT)}$ and taking the completion of these functions with respect to 
\[\|
(f,\cdot)_{H^1(\TT)}
\|_{(H^2(\TT))'}\] yields the space $L^2(\TT)$.
With this identification, the dual pairing between a function $g\in H^2(\TT)$ and a function $f\in L^2(\TT)$ is given by
\[\left<\left<f,g \right>\right>_{H^1(\TT)}\,=\,
\left<
f,g
\right>\,+\, 
\left<
\nabla f,\nabla g
\right>,
\]
where we recall that $\left<\cdot, \cdot \right>$ is the dual pairing in $L^2(\TT)$. Indeed, for $f\in H^1(\TT)$ this follows since it was identified with $
(f,\cdot)_{H^1(\TT)}$. For $f\in L^2(\TT)$ we take a sequence $(f_n)_{n\in \NN}$ from $H^1(\TT)$ converging to $f$  in $L^2(\TT)$. Then also $\nabla f_n \to \nabla f$ in $H^{-1}(\TT,\RR^2)$ and hence
\[
\left<\left<f,g \right>\right>_{H^1(\TT)}\,\leftarrow\,\left<\left<f_n,g \right>\right>_{H^1(\TT)}\,=\,
\left<
f_n,g
\right>\,+\, 
\left<
\nabla f_n,\nabla g
\right>\,\rightarrow \, \left<
f,g
\right>\,+\, 
\left<
\nabla f,\nabla g
\right>
\]
for all $g\in H^2(\TT)$.

\section{Justifications of It\^o's formula}\label{AppendixB}
	First, we justify the use of It\^o's formula in the proof of Lemma \ref{Lemma_Est}.
To this end, we introduce the equivalence relation 
\[
f\,\sim\, g \;\; \iff\;\; \exists c\in \RR:\;f=g+c
\]
for $f,g\in H^s(\TT)$, $s\ge 0$ and write $\dot{f}$ for the respective equivalence class in $H^s(\TT)$. We recall that the  Bessel potential space is given by
\[
H^s(\TT)\,=\,\left\{
f\in L^2(\TT)\,|\,
\vertiii{f}_{H^s(\TT)}<\infty
\right\},
\] where the appearing Bessel potential norm is defined as 
\[
\vertiii{f}_{H^s(\TT)}^2\,=\, \sum_{k\in \ZZ^2} (1+|2\pi k|^2)^s|\hat{f}(k)|^2
\] 
with $\hat{f}(k)$ being the $k$-th Fourier coefficient \eqref{Eq97} of a function $f\in L^2(\TT)$.
Under this  norm, the quotient space $H^s(\TT)/ \sim$ is equipped with
\begin{align*}
\vertiii{\dot{f}}_{H^s(\TT)/\sim}^2\,=\, \inf_{g\in \dot{f}} \sum_{k\in \ZZ^2} (1+|2\pi k|^2)^s|\hat{g}(k)|^2\,=\,&\inf_{g\in \dot{f}} |\hat{g}(0)|^2\,+
\sum_{k\in \ZZ^2 \setminus \{(0,0)\}} (1+|2\pi k|^2)^s|\hat{f}(k)|^2
\\=&
\sum_{k\in \ZZ^2 \setminus \{(0,0)\}} (1+|2\pi k|^2)^s|\hat{f}(k)|^2.
\end{align*}
Here, we used in the second equality that
\[
\int_{\TT}
 e^{-2\pi ki\cdot x}
\, dx\,=\, 0
\]
for $k\in \ZZ^2 \setminus \{(0,0)\}$
and therefore 
\[
\hat{f}(k)\,=\, \hat{g}(k)
\]
for all $g\in \dot{f}$.
In the following, we write $\dot{H}^s(\TT)$ for  $H^s(\TT)/ \sim$ and equip it with the equivalent norm
\[
\|\dot{f}\|_{\dot{H}^s(\TT)}^2\,=\, \sum_{k\in \ZZ^2\setminus \{(0,0)\}} |2\pi k|^{2s} |\hat{f}(k)|^{2}.
\]
Analogously to appendix \ref{AppendixAB}, one verifies that $ \dot{H}^0(\TT)$ can be identified with the dual of $\dot{H}^2(\TT)$ under the pairing in $ \dot{H}^1(\TT)$. Moreover, the dual pairing is given by
\begin{equation}\label{Eq103}
\left<\left<
\dot{f},\dot{g}
\right>\right>_{\dot{H}^1(\TT)}\,=\, \sum_{k\in \ZZ^2 \setminus \{(0,0)\}}
|2\pi k|^{2} \hat{f}(k)\overline{\hat{g}(k)}\,=\, \left<\nabla f,\nabla g\right>
\end{equation}
for $\dot{f}\in \dot{H}^1(\TT)$ and $\dot{g}\in \dot{H}^2(\TT)$ by  Parseval's relation \cite[Proposition 3.2.7]{grafakos2014classical}. For general $\dot{f}\in \dot{H}^0(\TT)$ the equality \eqref{Eq103} holds  by an approximation argument as in appendix \ref{AppendixAB}. 

Next, we denote by $P_{hom}$ the operator mapping a function $f\in H^s(\TT)$ to its equivalence class in $\dot{H}^s(\TT)$, i.e. $ P_{hom}f=\dot{f} $.
Then \[P_{hom}\in L(L^2(\TT), \dot{H}^0(\TT)) \cap L(H^1(\TT), \dot{H}^1(\TT))
\]
and applying $P_{hom}$ to equation \eqref{Eq_reg_SPDE} satisfied by $w_\epsilon$ on $[j\delta, (j+1)\delta)$ yields that
\begin{equation}\label{SPDE_proj}
P_{hom} w_\epsilon(t)\,=\,  P_{hom}  w_\epsilon(j\delta)\,+\, 
\int_{j\delta}^t P_{hom}  A^\epsilon (w_\epsilon(s))\, ds
\,+\,  \int_{j\delta}^t P_{hom}B(w_\epsilon(s))\, dV_t,\;\; t\in [j\delta, (j+1)\delta)
\end{equation}
where $V$ is the cylindrical Wiener process in $H^2(\TT, \RR^2)$  given by
\begin{equation}
	\label{Eq_vt}
V_t \,=\, \sum_{l=1}^\infty  \beta^{(l)}_t\psi_l.
\end{equation}
Because of $P_{hom}\in L(H^2(\TT), \dot{H}^2(\TT))$, $w_\epsilon \in  L^2([j\delta, (j+1)\delta]\times \Omega, {H}^2(\TT)) $ by Theorem \ref{Thm_Sec_Reg_SPDE}, and the boundedness of the operators $A^\epsilon$ and $B$, it holds 
\begin{align}&
	P_{hom} w_\epsilon\in  L^2([j\delta, (j+1)\delta]\times \Omega, \dot{H}^2(\TT)),\\&
	P_{hom}  A^\epsilon (w_\epsilon) \in  L^2([j\delta, (j+1)\delta]\times \Omega,\label{Eq101} \dot{H}^0(\TT)),\\&
	P_{hom}  B(w_\epsilon) \in  L^2([j\delta, (j+1)\delta]\times \Omega, L_2 ( H^2(\TT, \RR^2) ,\dot{H}^0(\TT))).\label{Eq102}
\end{align}
Moreover, because of right-continuity in $H^1(\TT)$, $w_\epsilon$ admits a progressively measurable, $H^2(\TT)$-valued $dt\otimes P$-version by \cite[Exercise 4.2.3]{liu2015stochastic}. Since later in the proof of Lemma \ref{Lemma_Est} we integrate in time and take the expectation anyways, we denote this progressive version again by $w_\epsilon$  to ease the notation. By continuity of the involved operators, also the processes \eqref{Eq101} and \eqref{Eq102} are progressive when choosing this $dt\otimes P$-version of $w_\epsilon$,
such that It\^o's formula for the squared norm in $\dot{H}^1(\TT)$ from \cite[Theorem 4.2.5]{liu2015stochastic} is applicable to \eqref{SPDE_proj}. Noting that by Parseval's relation, the norm in $\dot{H}^1(\TT)$ can be written as
\[
\|\dot{f}\|_{\dot{H}^1(\TT)}^2\,=\, \|\nabla f\|_{L^2(\TT,\RR^2)}^2,
\] we obtain that
	\begin{align*}
		\|\nabla w_\epsilon(t) \|_{L^2(\TT,\RR^2)}^2\,=\, \|\nabla w_\epsilon(j\delta) \|_{L^2(\TT, \RR^2)}^2\,+2\int_{j\delta}^t\left<\nabla {w_\epsilon}(s),\nabla A^\epsilon ({w_\epsilon}(s))\right>\,ds&\\+\,2\int_{j\delta}^t
		\left<  \nabla w_\epsilon (s),\nabla B(w_\epsilon(s))\, dV_s \right>\,+\, \sum_{l=1}^\infty \lambda_l^2 \int_{j\delta}^t\|\nabla 
		\divv({w_\epsilon}(s)\psi_l)\|_{L^2(\TT, \RR^2)}^2
		\,ds&
\end{align*}
for $t\in [j\delta, (j+1)\delta)$.
Writing the stochastic integral with respect to $V$ as its series representation results in  \eqref{Eq_Ito1}. Moreover, its quadratic variation is given by 
\begin{equation}\label{Eq99}
4 \int_{j\delta}^t
\left\|\left<  \nabla w_\epsilon (s),\nabla B(w_\epsilon(s))\, \cdot  \right>\right\|_{L_2(H^2(\TT, \RR^2), \RR )}^2\, ds\,=\, 4
\sum_{l=1}^{\infty}\lambda_l^2\int_{j\delta}^t
\left<\nabla \divv(w_\epsilon(s)\psi_l), \nabla w_\epsilon(s) \right>^2\, ds.
\end{equation}

Secondly, we justify the use of It\^o's formula in the proof of Lemma \ref{Bound_powers_2}, where  we use instead \cite[Proposition A.1]{DHV_16}. Choosing $\psi= \mathbbm{1}_{\TT}$, $\varphi= G_{\alpha,\kappa}$ in the notation of this proposition, we see that the functional \eqref{Eq44} is of the required form. Next, we observe that as a consequence of Theorem \ref{Thm_degenerate_sol} \ref{Item_SPDE_deg}, the process $w_N^{(k)}$ satisfies 
\[
d w_N^{(k)} \,=\, \divv ( G (t))\, dt\,+\, H(t)\,dV_t
\]
on $[j\delta, (j+1)\delta)$,
where 
\begin{align*}&
G(t)\,=\, \frac{1}{2} \sum_{l=1}^{\infty} \lambda_l^2  \divv({w}_N^{(k)}(t)\psi_l) \psi_l, \\&
H(t)[v]\,=\, \sum_{l=1}^\infty\lambda_l (v,\psi_l)_{H^2(\TT{ ,\RR^2 })} \divv ({w}_N^{(k)}(t) \psi_l),\;\;v\in H^2(\TT, \RR^2)
\end{align*}
and $V_t$ as in \eqref{Eq_vt}. By Theorem \ref{Thm_degenerate_sol} \ref{Item_est_deg}, we have 
\begin{equation}\label{Eq_B1}
w_N^{(k)} \,\in \, L^2(\Omega, L^2(j\delta, (j+1)\delta;H^1(\TT) ))
\end{equation}
and 
\[
w_N^{(k)} \,\in \, L^2(\Omega, C(j\delta, (j+1)\delta; L^2(\TT) )),
\]
if we replace its terminal value $w_N^{(k)}( (j+1)\delta)$ by $ v_N^{(k)}( (j+1)\delta)$.
By \eqref{boundedness_psil} and $\Lambda\in l^2(\NN)$ we have that 
\begin{align*}
	\|G(t)\|_{L^2(\TT,\RR^2)}\,\le& \, \frac{1}{2}\sum_{l=1}^\infty \lambda_l^2 
	\|\divv ( w_N^{(k)} (t) \psi_l ) \psi_l \|_{L^2(\TT,\RR^2)}
	\\\lesssim &\, \sum_{l=1}^\infty \lambda_l^2 
	\| ( w_N^{(k)} (t) \|_{H^1(\TT)} 	\,\lesssim_\Lambda \, \| ( w_N^{(k)} (t) \|_{H^1(\TT)}
\end{align*}
and consequently 
\eqref{Eq_B1} implies that
\[
G\,\in\, L^2(\Omega, L^2(j\delta, (j+1)\delta ;L^2(\TT, \RR^2) )).
\]
Similarly, we estimate
\begin{align*}
	\|H(t)\|_{L_2(H^2(\TT, \RR^2), L^2(\TT))}^2 \,=&\,
	\sum_{l=1}^\infty \lambda_l^2 \|\divv ( w_N^{(k)} (t) \psi_l ) \|_{L^2(\TT)}^2\\
	\lesssim &\, 
	\sum_{l=1}^\infty \lambda_l^2 \| w_N^{(k)} (t)  \|_{H^1(\TT)}^2\,\lesssim_\Lambda\, \| w_N^{(k)} (t)  \|_{H^1(\TT)}^2,
\end{align*}
such that
\[
H\,\in \, 
L^2(\Omega, L^2(j\delta, (j+1)\delta;L_2(H^2(\TT, \RR^2), L^2(\TT)) )).
\]
Hence, all the assumptions of \cite[Proposition A.1]{DHV_16} are satisfied, which results in \eqref{Eq10}.

\section*{Acknowledgments}
The author thanks his doctoral advisor Manuel Gnann for many insightful discussions on this subject as well as the careful reading of this document. He thanks Mark Veraar for pointing out the possibility to relax integrability assumptions on the initial value.  He also thanks the anonymous referees for the careful reading of this document and their valuable suggestions.

\end{document}